\UseRawInputEncoding
\documentclass[a4paper,oneside,11pt]{article}
\usepackage{amsmath}
\usepackage{amssymb}
\usepackage{titlesec}
\usepackage{graphicx}
\usepackage{fancyhdr}
\usepackage{latexsym}
\usepackage{mathrsfs}
\usepackage{booktabs}
\usepackage{mathrsfs}
\usepackage{mathtools}
\usepackage{amsthm}
\usepackage{wasysym}
\usepackage{cite}
\usepackage{color}
\usepackage{amsfonts}
\usepackage{pgf,pgfarrows,pgfnodes,pgfautomata,pgfheaps}
\usepackage{amsmath,amssymb}
\usepackage{graphics}
\usepackage{multimedia}
\usepackage{graphics}
\usepackage{cite}
\usepackage{latexsym}
\usepackage{amsmath}
\usepackage{amssymb}
\usepackage[dvips]{epsfig}
\usepackage{amscd}
\allowdisplaybreaks[3]
\numberwithin{equation}{section}
\newtheorem{definition}{Definition}[section]
\newtheorem{theorem}{Theorem}[section]
\newtheorem{lemma}{Lemma}[section]
\newtheorem{remark}{Remark}[section]
\newtheorem{proposition}{Proposition}[section]
\newtheorem{corollary}{Corollary}[section]

\newcommand{\ba}{\begin{aligned}}
\newcommand{\ea}{\end{aligned}}
\newcommand{\be}{\begin{equation}}
\newcommand{\ee}{\end{equation}}

\newcommand{\bnn}{\begin{eqnarray*}}
\newcommand{\enn}{\end{eqnarray*}}

\newcommand{\thatsall}{\hfill$\Box$}
 
\date{}

\title{\bf\boldmath{Axisymmetric weak solutions to stationary compressible Navier-Stokes equations with critical indices}}
\author{Xinyu F{\small AN}$^{a}$, Song J{\small IANG}$^{b}$
\footnote{\emph{E-mail addresses\,: fanxinyu17@mails.ucas.ac.cn (X. Y. Fan), jiang@iapcm.ac.cn (S. Jiang)} }
\\  {\normalsize a.  Hua Loo-Keng Center for Mathematical Sciences, AMSS,}\\
{\normalsize Chinese Academy of Sciences, Beijing 100190, P. R. China;}\\
{\normalsize b. Institute of Applied Physics and Computational Mathematics}\\
{\normalsize Beijing 100088, PR China}}
\begin{document}
\maketitle
\begin{abstract}
This paper studies the isothermal stationary compressible Navier-Stokes equations on $\mathbb{R}^2\times\mathbb{T}$ and the cylinder $\mathbb{D}\times\mathbb{T}$ with $\mathbb{D}$ the 2-dimensional unit disc. There are two critical exponents in such settings: The heat ratio $\gamma=1$ is an end point of the classical theory on weak solutions; the axisymmetric solutions in the unbounded domain $\mathbb{R}^2\times\mathbb{T}$ involve the critical index of Sobolev's inequality. Some new observations based on the cancellation structure of the equations are collected to get over these obstacles due to critical exponents.

\end{abstract}
\section{Introduction and main results} 
\quad We consider the isentropic stationary compressible Navier-Stokes equations in $\mathbb{R}^3$ and the cylinder $\mathbb{D}\times\mathbb{R}$ with $\mathbb{D}$ the 2-dimensional unit disc:
\begin{equation}\label{C101}
\begin{cases}
\mathrm{div}(\rho u)=0,\\
\mathrm{div}(\rho u\otimes u)-\mu\Delta u-\lambda\nabla\mathrm{div}u+\nabla P=g,
\end{cases}
\end{equation}
where $\rho\geq 0$ is the density and $u=(u_1,u_2,u_3)$ is the velocity field. The pressures $P$ and viscosity coefficients $\mu$ and $\lambda$ satisfy:
\begin{equation}\label{C103}
P=a\rho^\gamma,~~\mu>0,~~\lambda\geq\mu/3,
\end{equation}
for some positive constant $a>0$ and the specific heat ratio $\gamma\geq 1$. The function $g=(g_1,g_2,g_3)$ is the external force imposed on the system.
This paper mainly investigates the isothermal case, in which we take
$$\gamma=1.$$ 
Without loss of generality, we set
\begin{equation}\label{C104}
\mu=a=1,~~ \lambda=0.
\end{equation}
The method below is also available for the general case \eqref{C103}.  

We will construct axisymmetric weak solutions to \eqref{C101}, which means that the velocity field $u$ and the external force $g$ consist of the radial components $u_1, g_1,$ and the axial components $u_2, g_2$, in addition, all functions involved are axisymmetric as well. Thus, according to \eqref{C104}, it is convenient to transform \eqref{C101} into
\begin{equation}\label{C102}
\begin{cases}
\frac{1}{x_1}\partial_{1}(x_1\rho u_1)+\partial_{2}(\rho u_2)=0,\\
\frac{1}{x_1}\partial_1(x_1\rho u_1^2)+\partial_{2}(\rho u_1u_2)-\frac{1}{x_1}\partial_{1}(x_1\partial_{1}u_1)-\partial_{22}u_1+\frac{1}{x_1^2}\,u_1+\partial_{1}\rho=g_1,\\
\frac{1}{x_1}\partial_1(x_1\rho u_1u_2)+\partial_2(\rho u_2^2)-\frac{1}{x_1}\partial_1(x_1\partial_1u_2)-\partial_{22}u_2+\partial_2\rho=g_2,
\end{cases}
\end{equation}
where $(x_1,x_2)$ are cylindrical coordinates denoting the radial and axial directions.
Moreover, we suppose that the flow is periodic in the $x_2$ direction, thus the domains under considerations are actually  tubes $\Omega=(0,\infty)\times\mathbb{T}$ and $\Omega=(0,1)\times\mathbb{T}$, where $\mathbb{T}=\mathbb{R}/\mathbb{Z}$ is the 1-dimensional torus. 

Let us clarify the boundary conditions and physical restrictions of the systems. For the unbounded domain $(0,\infty)\times\mathbb{T}$, we impose that
\begin{equation}\label{C105}
\begin{split}
&u_1\big|_{\{x_1=0\}}=\partial_1u_2\big|_{\{x_1=0\}}=0,\\
&u\rightarrow 0,~~\rho\rightarrow\rho_\infty\ \mbox{as $x_1\rightarrow\infty$,}
\end{split}
\end{equation}
for some constant $\rho_\infty>0$ describing the far field behaviour of the density. While for the bounded domain $(0,1)\times\mathbb{T}$, we suppose that
\begin{equation}\label{CC105}
\begin{split}
&\quad u_1\big|_{\{x_1=0\}}=\partial_1u_2\big|_{\{x_1=0\}}=0,~~u\big|_{\{x_1=1\}}=0,\\
&\quad~~\int_{(0,1)\times\mathbb{T}}\rho\, x_1dx_1dx_2=\int_{\mathbb{D}\times\mathbb{T}}\rho\,dx=M,
\end{split}
\end{equation} 
for some constant $M>0$ which gives the total mass of the system.

In conclusion, the central topic of the paper is  the system \eqref{C102} in $(0,\infty)\times\mathbb{T}$ with the condition \eqref{C105}   and the system \eqref{C102} in $(0,1)\times\mathbb{T}$ with the condition \eqref{CC105}. Meanwhile, we will also study them via \eqref{C101} by going back to $\mathbb{R}^2\times\mathbb{T}$ and $\mathbb{D}\times\mathbb{T}$.

There is a huge number of literature on the compressible Navier-Stokes equations, while in the last decades, the study on the well-posedness of the system is fruitful. The major break through in the framework of weak solutions is due to Lions \cite{L2}, where he obtained the global existence of weak solutions to the compressible Navier-Stokes equations, merely assuming that the initial energy is finite. For technical reasons, the method in \cite{L2} requires that the heat ratio $\gamma\geq 9/5$ for the 3 dimensional case, later Feireisl-Novotn\'{y}-Petzeltov\'{a} \cite{FNP} extended such index to the critical case $\gamma\geq 3/2$ by introducing a proper truncation. For the radial symmetric and axial symmetric initial data, Jiang-Zhang \cite{JZ01,JZ03} established the 3D global  weak solutions to the compressible system under the condition of $\gamma>1$. Recently, Bresch-Jabin \cite{BJ} studied the compressible system with anisotropic viscous stress tensor, while Hu \cite{Hu} made detailed analysis on the Hausdorff dimension of the concentration set of the 3D compressible system with $\gamma\in(6/5,3/2]$. We mention that the uniqueness and regularity of the Lions-Feireisl weak solutions still remain completely open.

A closely related problem is the existence of weak solutions to the corresponding stationary system \eqref{C101}, where the time derivatives are not involved. Lions \cite[Chapter 6]{L2} extended his result to the stationary system, under the condition of $\gamma\geq 5/3$ in dimension 3, while for the potential external force, Novo-Novotn\'{y} \cite{NN} improved the index to $\gamma\geq 3/2$. After that,  Frehse-Goj-Steinhauer \cite{FGS} and Plotnikov-Sokolowsk \cite{PS05,PSJF} independently established the $L^\infty$-estimates of $\Delta^{-1}P$ which improve the a priori estimates on $\rho$. Then Brezina-Novotn\'{y} \cite{BN} combined these results with the usual energy estimates, and obtained the existence of weak solutions with periodic boundary conditions for any $\gamma>(3+\sqrt{41})/8$. More recently, Jiang-Zhou \cite{JZ} constructed weak solutions to the 3D stationary system with $\gamma>1$ under periodic boundary conditions, by establishing a new coupled estimates in the Morrey space of both the kinetic energy and pressure. The result in \cite{JZ} was extended to the slip boundary conditions by Jessl\'{e}-Novotn\'{y} \cite{JN} and to the Dirichlet boundary conditions by Plotnikov-Weigant \cite{PWJMPA}.

The above arguments mainly deal with the case $\gamma>1$ where the pressure term gains some convex properties. However, the results for the critical case $\gamma=1$ are few, since the integrability of the density $\rho$ is much weaker. Lions \cite[Chapter 6]{L2} applied the concentration compactness assertions to establish the weak solutions to the 2D stationary isothermal system with damping terms. Then, Plotnikov-Sokolowski \cite{PS05} systematically studied the 3D stationary system with or without the damping structure and carefully calculated the Hausdorff dimension of the concentration set of the isothermal case. Later, Frehse-Steinhauer-Weigant \cite{FSW10} obtained the weak solution to the 2D isothermal system on the cube domain with slip-type boundary conditions by considering the stream function induced by \eqref{C101}$_1$. Recently, Plotnikov-Weigant \cite{PW} also constructed weak solutions to the 2D non-stationary isothermal  system by making good use of the Radon transformations.

However, the Lions-Feireisl weak solutions to the stationary system \eqref{C101} established up to now all concern with 2D bounded domains, leaving  the 3D problems open. The aim of this paper is to construct the axisymmetric weak solutions to the 3D stationary systems \eqref{C101} in bounded and unbounded domains.

Let us first clarify the notations through out the paper. We introduce $\Omega= (0,\infty)\times\mathbb{T}$ or $ (0,1)\times\mathbb{T}$, and denote $\hat\Omega=\mathbb{R}^2\times\mathbb{T}$ or $\mathbb{D}\times\mathbb{T}$ as the corresponding 3-dimensional domain. 
Note that a function $f$ defined on $\Omega$ induces an axisymmetric function $\hat{f}$ on $\hat\Omega$, and the integration over $\Omega$ can be transformed into $\hat\Omega$ by 
\begin{equation*} 
\int_{\Omega}f\,x_1dx_1dx_2=\frac{1}{2\pi}\int_{\hat\Omega} \hat{f}\,d\hat{x},
\end{equation*}
where $x_1$ and $x_2$ denote the radial and axial directions of $\mathbb{R}^3$ respectively.
In addition, suppose that $k\in\mathbb{N}$ and $1\leq p\leq\infty$, then the usual Lebesgue space and Sobolev space over the 2 or 3 dimensional domain $A$ are given by $L^p(A)$ and $W^{k,p}(A)$, while $H^k(A)\triangleq W^{k,2}(A)$. For convenient, we define the space $\mathcal{L}^p(\Omega)$ by
$$\mathcal{L}^p(\Omega)\triangleq\left\{f\in L_{loc}^1(\Omega)\,\big|\,\int_{\Omega}|f|^px_1dx_1dx_2<\infty\right\}.$$
Note that $\mathcal{L}^p(\Omega)$ stands for axisymmetric functions in $L^p(\hat\Omega)$. Moreover we introduce the notation
$$\int_A\int_B f(x_1,x_2)\,dx_1dx_2\triangleq \int_A\big(\int_Bf(x_1,x_2)\,dx_2\big)\,dx_1.$$
Next, we give the precise definition of weak solutions in our arguments.
\begin{definition} 
We call $(\rho\geq 0, u_1,u_2)$ is a weak solution to the system \eqref{C102} in the domain $\Omega=(0,\infty)\times\mathbb{T}$ or $(0,1)\times\mathbb{T}$, provided that 

i) Regularity conditions hold: for $i=1,2,$
\begin{equation}\label{C106}
\begin{split}
\rho\in \mathcal{L}^1_{loc}(\Omega),~~
\rho|u_i|^2\in \mathcal{L}^1_{loc}(\Omega),~~  \nabla u_i\in \mathcal{L}^2(\Omega).\\
\end{split}
\end{equation}

ii) The system \eqref{C102} is solved by $(\rho,u_1,u_2)$ in the sense of distributions: $\forall\phi\in C_0^\infty(\Omega)$, 
\begin{equation}\label{C108}
\int_{\Omega}\rho u\cdot\nabla\phi\,x_1dx_1dx_2=0;
\end{equation}
while $\forall\varphi\in \big(C_0^\infty(\Omega)\big)^2$ satisfying $\varphi_1\big|_{\{x_1=0\}}=\partial_1\varphi_2\big|_{\{x_1=0\}}=0$,  
\begin{equation}\label{C109}
\begin{split}
&\int_{\Omega}\big(\rho u_iu_j\cdot\partial_i\varphi_j+\rho\,\big(\frac{\varphi_1}{x_1}+\mathrm{div}\varphi\big)+g\cdot\varphi\big)\,x_1dx_1dx_2\\
&=\int_\Omega\big(\nabla u\cdot\nabla\varphi+(u_1/x_1)\cdot(\varphi_1/x_1)\big)\,x_1dx_1dx_2.
\end{split}
\end{equation}
\end{definition}
Following the classical theory, we will construct the weak solutions to the system \eqref{C102} by solving a sequence of approximation systems 
\begin{equation}\label{C113}
\begin{cases}
\frac{1}{x_1}\partial_{1}(x_1\rho^\varepsilon u_{1}^\varepsilon)+\partial_{2}(\rho^\varepsilon u_2^\varepsilon)=0,\\
\frac{1}{x_1}\partial_1(x_1\rho^\varepsilon (u_1^\varepsilon)^2)+\partial_{2}(\rho^\varepsilon u_1^\varepsilon u_2^\varepsilon)-\frac{1}{x_1}\partial_{1}(x_1\partial_{1}u_1^\varepsilon)-\partial_{22}u_1^\varepsilon+\frac{1}{x_1^2}u_1^\varepsilon+\partial_{1}P^\varepsilon=g_1,\\
\frac{1}{x_1}\partial_1(x_1\rho^\varepsilon u_1^\varepsilon u_2^\varepsilon)+\partial_2(\rho^\varepsilon (u_2^\varepsilon)^2)-\frac{1}{x_1}\partial_1(x_1\partial_1u_2^\varepsilon)-\partial_{22}u_2^\varepsilon+\partial_2P^\varepsilon=g_2,
\end{cases}
\end{equation}
where the pressure is given by $P^\varepsilon=\rho^\varepsilon+\varepsilon(\rho^\varepsilon)^{\alpha}$ with $\alpha\geq 10$ which ensures the existence of approximation solutions due to the arguments given by Lions \cite[Theorem 6.11]{L2}. 

For the unbounded domain $(0,\infty)\times\mathbb{T}$,
we suppose that
\begin{equation}\label{C114}
\begin{split}
u_1^\varepsilon\big|_{\{x_1=0\}}&=0,~~\partial_1u_2^\varepsilon\big|_{\{x_1=0\}}=0,\\
&u^\varepsilon\rightarrow 0\ \mbox{as $x_1\rightarrow\infty$}.
\end{split}
\end{equation}
Moreover, in order to describe the far field behaviour \eqref{C105}$_2$ of the density, we consider the following restriction that
\begin{equation}\label{CC101}
\int_\Omega P^\varepsilon\omega(x_1)\,dx_1dx_2=M,
\end{equation}
for some constant $M>0$. The weight function $\omega(x_1)\in C^\infty\big([0,\infty)\big)$ satisfies
\begin{equation}\label{C116}
\begin{cases}
\omega(x_1)=x_1\ \forall x_1\in[0,1],\\
\omega(x_1)=x_1^{-2}\ \forall x_1\in[2,\infty),\\
1/4<\omega\leq 4\ \forall x_1\in[1,2].
\end{cases}
\end{equation}
We mention that the condition \eqref{CC101} is introduced by Lions \cite[Section 6.8]{L2} which also ensures that $P^\varepsilon\rightarrow P^\varepsilon_\infty$ for some constant $P^\varepsilon_\infty>0$ as $x_1\rightarrow\infty$ in some sense due to the fast decay rate of $\omega$ (see arguments in \cite[Section 6.8]{L2}).  

While for the bounded domain $(0,1)\times\mathbb{T}$, we impose the boundary conditions
\begin{equation}\label{CC102}
\begin{split}
u_1^\varepsilon\big|_{\{x_1=0\}}&=\partial_1u_2^\varepsilon\big|_{\{x_1=0\}}=0,\\
&u^\varepsilon\big|_{\{x_1=1\}}=0, 
\end{split}
\end{equation}
as well as the physical restriction
\begin{equation}\label{CC103}
\int_\Omega P^\varepsilon\,x_1dx_1dx_2=M,
\end{equation}
for some constant $M>0$ providing the total mass of the system.

The existence results of the problem \eqref{C113} with conditions \eqref{C114}--\eqref{CC101} in $(0,\infty)\times\mathbb{T}$ and the problem \eqref{C113} with conditions \eqref{CC102}--\eqref{CC103} in $(0,1)\times\mathbb{T}$ are given by the next lemma which can be found in \cite[Sections 6.7--6.8]{L2}.
\begin{lemma}\label{L21}
Let $\Omega=(0,\infty)\times\mathbb{T}$ or $(0,1)\times\mathbb{T}$ and $M>0$ be a positive constant. If the external force $g$ satisfies
\begin{equation}\label{C115}
g\in\mathcal{L}^1(\Omega)\cap\mathcal{L}^\infty(\Omega).
\end{equation}
Then the system \eqref{C113} in $(0,\infty)\times\mathbb{T}$ under conditions \eqref{C114}--\eqref{CC101} and the system \eqref{C113} in $(0,1)\times\mathbb{T}$ under conditions \eqref{CC102}--\eqref{CC103} both admit at least one weak solution $(\rho^\varepsilon,u_1^\varepsilon,u_2^\varepsilon)$, which also gains the extra regularity that
\begin{equation}\label{C203}
\begin{split}
\rho^\varepsilon\in L^\infty_{loc}(\Omega),~~u^\varepsilon\in W^{1,p}_{loc}(\Omega),\, \forall p\in[1,\infty).
\end{split}
\end{equation} 
In addition, for the unbounded domain $(0,\infty)\times\mathbb{T}$, we can find a non-negative constant $P^\varepsilon_\infty$ and some $r_0>2$, such that
\begin{equation}\label{C118}
P^\varepsilon-P^\varepsilon_\infty\in\mathcal{L}^{r_0}(\Omega),
\end{equation}
which corresponds to the far field behaviour $P^\varepsilon\rightarrow P^\varepsilon_\infty$ as $x_1\rightarrow\infty$.

As the result, \eqref{C203} and \eqref{C118} ensure that \eqref{C113}$_1$ holds in the sense of renormalized solutions: $\forall\psi\in C_0^\infty(\Omega),$
\begin{equation}\label{C111}
\int_\Omega b(\rho^\varepsilon)u^\varepsilon\cdot\nabla\psi\,x_1dx_1dx_2
=\int_\Omega\big(b'(\rho^\varepsilon)\rho^\varepsilon-b(\rho^\varepsilon)\big)\,\big(\mathrm{div}u^\varepsilon+\frac{u^\varepsilon_1}{x_1}\big)\cdot\psi\,x_1dx_1dx_2,
\end{equation}
for any $b(t)\in C^1([0,\infty))$, such that $b'(t)=0$ when $t$ is large enough. 
\end{lemma}
\begin{remark}
The condition \eqref{C115} on the external force $g$ can be improved, see discussion in \cite[Chapter 6]{L2} for examples. Meanwhile, the conditions \eqref{CC101} and \eqref{CC103} help us determine the solution, since the system \eqref{C102} is multi-solvable if no restriction is imposed on $\rho$, see the counter examples provided by \cite[Section 6.7]{L2}.
\end{remark}

Now, we state the first result of the paper which guarantees the weak compactness of the approximation sequence.
\begin{lemma}\label{T1}
Under the assumptions of Lemma \ref{L21}, let $\{(\rho^\varepsilon,u_1^\varepsilon,u_2^\varepsilon) \}_{\varepsilon>0}$ be a sequence of solutions to the system \eqref{C113} in $(0,\infty)\times\mathbb{T}$ under conditions \eqref{C114}--\eqref{CC101} or the system \eqref{C113} in $(0,1)\times\mathbb{T}$ under conditions \eqref{CC102}--\eqref{CC103}, which is provided by Lemma \ref{L21} and satisfies \eqref{C203}--\eqref{C111}. Then if the external force $g$ also meets
\begin{equation}\label{C112}
\int_0^1g(x_1,x_2)\,dx_2=0,~~\forall x_1\in[0,\infty),
\end{equation}
we can find a subsequence 
converging weakly to a weak solution $(\rho, u_1, u_2)$ of the system \eqref{C102}. 
Moreover, the slices of the axial component $u_2$ satisfy that, $\forall x_1>0$ 
\begin{equation}\label{C130}
\bigg|\int_0^1 u_2(x_1,x_2)\,dx_2\bigg|\leq \max\{x_1^{-1},\mathbf{C}\},
\end{equation}
for some generic constant $\mathbf{C}$ determined by $\Omega$.
\end{lemma}
\begin{remark}
The estimates \eqref{C130} provide proper controls on the average of $u_2$ in any bounded domain away from the symmetric axis, and also lead to the weighted estimates on the $L^2$-norm of $u_2$. Note that it is not directly available from $\nabla u\in\mathcal{L}^2(\Omega)$ when $\Omega=(0,\infty)\times\mathbb{T}$ is unbounded.
\end{remark}
\begin{remark}
The cancellation condition \eqref{C112} imposed on the external force $g$ plays an important role in our arguments. If we expand $g$ in the $x_2$ direction by
\begin{equation*} 
g(x_1,x_2)=\sum_{k=-\infty}^\infty\beta_k(x_1)\,e^{2\pi ikx_2},
\end{equation*}
then the condition \eqref{C112} is simply interpreted as
$$\sum_{k=-\infty}^\infty\int_0^\infty|\beta_k(x_1)|^2x_1dx_1<\infty,~~\beta_0(x_1)\equiv 0.$$
We mention that Lemma \ref{T1} is valid for more general external force $g$ which meets
$$\sum_{k\in\mathbb{N},\,k\neq 0}\frac{1}{k^2}\int_0^\infty|\beta_k(x_1)|^2x_1dx_1<\infty,~~\beta_0(x_1)\equiv 0.$$

We are temporarily not able to treat  the general external force in form of $\rho f$. In fact, for the unknown $\rho$, it is not direct to check that \eqref{C112} holds:
$$\int_0^1\rho f(x_1,x_2)\,dx_2=0~~\forall x_1\geq 0.$$
The related problems for the general body external forces are left for the future.
\end{remark}
 
\begin{remark}\label{R12}
Suppose that $(\rho,u_1,u_2)$ is a weak solution obtained in Lemma \ref{T1}, for $\hat{x}=(\hat{x}',\hat{x}_3)\in\mathbb{R}^2\times\mathbb{T}$, we define $(\hat\rho,\hat{u})$ by
$$\hat\rho(\hat{x})=\rho(|\hat{x}'|,\hat{x}_3),~~\hat{u}(\hat{x})=\left(u_1(|\hat{x}'|,\hat{x}_3)\frac{\hat{x}'}{|\hat{x}'|},\,u_2(|\hat{x}'|,\hat{x}_3)\right),$$
then according to \cite[Theorem 5.7]{Hoff92}, $(\hat\rho,\hat{u})$ is a weak solution to \eqref{C101} in $\hat\Omega$.  
\end{remark}
 
The weak compactness assertion in Lemma \ref{T1} can be viewed as a local argument, and it is also sufficient to provide the existence of weak solutions in the case of the bounded domain $(0,1)\times\mathbb{T}$.

\begin{theorem}\label{TT1}
Suppose that the assumptions of Lemma \ref{L21} and the cancellation condition \eqref{C112} are valid, then for any constant $M>0$ and $\Omega=(0,1)\times\mathbb{T}$, we can construct at least one weak solution $(\rho,u_1,u_2)$ to the system \eqref{C102} in $\Omega$, satisfying \eqref{CC105} and
\begin{equation}\label{CC104}
\int_\Omega\rho\,x_1dx_1dx_2=M,
\end{equation}
which gives the total mass of the fluid.
\end{theorem}

In view of \eqref{CC104}, the weak solution obtained in Theorem \ref{TT1} can not be the trivial one that $\rho\equiv 0$. However, the weak compactness arguments in Lemma \ref{T1} are not strong enough to handle the case of unbounded domain $(0,\infty)\times\mathbb{T}$. In fact, the approximation sequence possibly converges to the trivial solution, due to the lack of strong compactness on unbounded domains (see the examples given at Section \ref{S4}), therefore we must establish the next concentration argument to exclude the trivial case, during the limit process in $(0,\infty)\times\mathbb{T}$.
 
\begin{theorem}\label{T12}
Suppose that the assumptions of Lemma \ref{L21} and the cancellation condition \eqref{C112} are valid, if $\Omega=(0,\infty)\times\mathbb{T}$ and 
$\{(\rho^\varepsilon,u_1^\varepsilon,u_2^\varepsilon) \}_{\varepsilon>0}$ is a sequence of solutions to the system \eqref{C113}--\eqref{CC101} in $\Omega$, which is provided by Lemma \ref{L21} and  satisfies \eqref{C203}--\eqref{C111}, then $\{\rho^\varepsilon\}_{\varepsilon>0}$ is concentrated on compact sets in the sense that, for any $\delta>0$, there is a constant $N>0$ such that $\forall\varepsilon>0$ 
\begin{equation}\label{CA102}
\int_{\{x_1\geq N\}}P^\varepsilon\,\omega\,dx_1dx_2<\delta.
\end{equation} 

Consequently, for any constant $M>0$, we can construct at least one weak solution $(\rho, u_1, u_2)$ to the system \eqref{C102} in $\Omega$, satisfying \eqref{C105} and
\begin{equation}\label{C119}
\int_\Omega\rho\,\omega\,dx_1dx_2=M,
\end{equation}
which corresponds to the far field behaviour $\rho\rightarrow\rho_\infty$ as $x_1\rightarrow\infty$.
\end{theorem}
\begin{remark}
We mention that \eqref{C119} also ensures that the solution is non-trivial. Compared with \eqref{CC104}, which gives the total mass of the fluid, the restriction \eqref{C119} is less physically meaningful. As indicated by Lions \cite[Section 6.8]{L2}, in the unbounded domain $(0,\infty)\times\mathbb{T}$, we are temporarily not able to construct a weak solution with the prescribing total mass
\begin{equation}\label{CC106}
\int_{(0,\infty)\times\mathbb{T}}\rho^\varepsilon\, x_1dx_1dx_2=M,
\end{equation}
which means taking $\omega(x_1)=x_1$ in \eqref{C119}, and the examples in Section \ref{S4} partly explain the reason: Note that \eqref{CC106} only guarantees that the weak limit satisfies
$$\int_{(0,\infty)\times\mathbb{T}}\rho\,x_1dx_1dx_2\leq M,$$
thus $\{\rho_\varepsilon\}_{\varepsilon>0}$ may vanish completely during the limit process. In contrast of it, \eqref{C119} leads to the concentration property \eqref{CA102}, which avoids the vanishing case and provides the key ingredient  of obtaining non-trivial solutions.

\end{remark}

The critical exponents cause essential difficulties in our analysis. Let us make some comments on the key technical points of the paper.

\textit{1. The critical heat ratio $\gamma=1$.}
 
The integrability of $\rho$ is quite low due to $\gamma=1$, especially in the region nearing the axis, where a singular weight $x_1^{-1}$ occurs.
We will carry out the interpolation arguments and obtain $\mathcal{L}_{loc}^{1+\varepsilon}(\Omega)$ estimates of $\rho$, which corresponds to the crucial extra integrability in classical arguments \cite{L2, JZ03}.

Motivated by \cite[Section 2]{JZ03}, the key ingredient is the weighted estimates on $\rho$ (see Proposition \ref{P32}),
$$\int_{\{x_1\leq 1\}}\rho\,x_1^{\varepsilon/2}\,dx_1dx_2\leq C.$$
Note that the index of weight $\varepsilon/2$ is strictly smaller than $1$, which controls the singularity near the axis. Then
let us multiply \eqref{C102} by $x_1^m$ with $m\geq 4$ to declare that
\begin{equation*} 
\begin{cases}
\mathrm{div}(\rho u\cdot x_1^m)=R_1,\\
\mathrm{div}(\rho u\otimes u\cdot x_1^m)-\Delta(u\cdot x_1^m)+\nabla(\rho\cdot x_1^m)=R_2.
\end{cases}
\end{equation*}
Thus,  
regarding $R_i$ as external forces, the above equations fall into the scope of the standard theory of 2-dimensional isothermal system in
\cite{PS05,PW}, which also provides the (weighted) extra integrability on $\rho$,
$$\int_{\{x_1\leq 1\}}\rho^{1+\alpha}\,x_1^m\,dx_1dx_2\leq C,$$
for some $\alpha\in(1/4,1)$. Combining these estimates, interpolation arguments lead to the desired $\mathcal{L}^{1+\epsilon}_{loc}(\Omega)$ estimates on $\rho$,  
$$\int_{\{x_1\leq 1\}}\rho^{1+\epsilon}x_1\,dx_1dx_2= \int_{\{x_1\leq 1\}}(\rho^{1+\alpha}\,x_1^m)^\theta\,(\rho\,x_1^{\varepsilon/2})^{1-\theta}\,dx_1dx_2\leq C.$$

\textit{2. The critical index of the Sobolev's inequality.}

When we deal with the unbounded domain $\Omega=(0,\infty)\times\mathbb{T}$,
the basic energy estimates of \eqref{C101} in Proposition \ref{P31} ensure that
\begin{equation*} 
\int_\Omega\left(|\nabla u|^2x_1+\frac{|u_1|^2}{x_1}\right)dx_1dx_2\leq C,
\end{equation*}
which merely provides controls on $\nabla u$ and $u_1$. However $\|\nabla u_2\|_{\mathcal{L}^2(\Omega)}$ fails to give any global $L^p$ estimates of $u_2$ on the 2D unbounded domain $\Omega$ and the compact embedding theorem also breaks down.
Inspired by the arguments on the thin domains (see \cite{RS93} for examples), we will turn to establish the upper bound of the slices 
\begin{equation}\label{C124}
\int_0^1 u_2(x_1,x_2)\,dx_2~~\forall x_1\in(0,\infty).
\end{equation}
The main observation is to integrate \eqref{C102}$_1$ and \eqref{C102}$_3$ with respect to the $x_2$ direction, and the equations will reduce largely to
\begin{equation*}
\begin{split}
& \partial_1\big(\int_0^1 x_1\,\rho u_1 \,dx_2\big)=0,\\
& \partial_1\left(\int_0^1 x_1\,\partial_1u_2 \,dx_2\right)=\partial_1\left(\int_0^1 x_1\rho u_1u_2 \,dx_2\right),
\end{split}
\end{equation*}
due to the periodic assumptions on the $x_2$ direction. Then the boundary conditions of $u_1$ and $u_2$ at the axis provide
the crucial cancellation condition on $\rho u_1$ along with the explicit formula of slices \eqref{C124} as follows (see Proposition \ref{P35}),
\begin{equation}\label{C125}
\begin{split}
&\int_0^1\rho u_1\,dx_2=0,~~\forall x_1\in(0,\infty),
\\
&\int_0^1 u_2(x_1,x_2)\,dx_2=-\int_{x_1}^\infty\int_0^1\rho u_1u_2\,dsdx_2,~~\forall x_1\in(0,\infty).
\end{split}
\end{equation}
Let us substitute \eqref{C125}$_1$ into \eqref{C125}$_2$ and infer that
\begin{equation}\label{C126}
\begin{split}
\int_0^1 u_2(x_1,x_2)\,dx_2
&=-\int_{x_1}^\infty\int_0^1\rho u_1(u_2-\int_0^1 u_2\,dx_2)\,dsdx_2\\
&\leq \int_{x_1}^\infty\big(\int_0^1\rho |u_1|\,dx_2\big)\big(\int_0^1|\nabla u_2|\,dx_2\big)\,ds\\
&\leq\big(\int_{x_1}^\infty(\int_0^1\rho |u_1|\,dx_2\big)^2s^{-1}\,ds\big)^{1/2}\|\nabla u_2\|_{\mathcal{L}^2(\Omega)}\\
&\leq C\big(\int_{x_1}^\infty(\int_0^1\rho |u_1|\,dx_2\big)^2s^{-1}\,ds\big)^{1/2}.
\end{split}
\end{equation}
Note that $u_1$ is properly controlled by energy estimates, thus \eqref{C126} gives desired bounds on the slices \eqref{C124}, which along with Poincar\'{e}'s inequality helps us treating the critical case of Sobolev's embedding theorem.

We specially mention that there is an extra weight $s^{-1}$ on the right hand side of \eqref{C126}, which cuts down the singularity at $\infty$ by order 1. Such fact plays an significant role in our compactness (concentration) arguments on unbounded domains, and it helps us excluding the trivial solutions in Theorem \ref{T12}. The cancellation method in \eqref{C126} seems
useful to improve the estimates on the tube-like domains.

The rest of the paper is organized as follows: Some basic properties and the elliptic estimates  are collected in Section \ref{S2}. Then we carry out a priori estimates under the conditions of Lemma \ref{T1} in Section \ref{S3} which is the central part of the paper. Next, in Subsection \ref{S4}, let us establish the weak compactness assertions, and Theorem \ref{TT1} in $(0,1)\times\mathbb{T}$ follows automatically. Finally, in Subsection \ref{S5}, we derive the concentration arguments in $(0,\infty)\times\mathbb{T}$, which leads to Theorem \ref{T12}.

\section{Preparations}\label{S2}
\quad In this section, we provide some basic facts  and properties which will be used frequently in our further arguments.  
We first quote the well known Hardy-Littlewood-Sobolev inequality to deal with the potential estimates.
\begin{lemma}
Suppose that $n\in\mathbb{N}^+$ and $\alpha\in(0,n)$, then for any $f\in C_0^\infty(\mathbb{R}^n)$, we define the potential
$$Tf(x)=\int_{\mathbb{R}^n}\frac{f(y)}{|x-y|^{n-\alpha}}\,dy.$$
Then $\forall p\in(1,\infty)$, there is a constant $C$ depending only on $n$, $p$, and $\alpha$, such that,
\begin{equation}\label{C204}
\|T(f)\|_{L^q(\mathbb{R}^n)}\leq C\|f\|_{L^p(\mathbb{R}^n)},
\end{equation}
where $1/q=1/p-\alpha/n$.
\end{lemma}

Next, we consider the elliptic estimates on $\mathbb{R}\times\mathbb{T}$ and $\mathbb{T}\times\mathbb{T}$ respectively. The complete proof of arguments below can be found in \cite[Chapter VII]{SW}.
\begin{lemma}\label{L23}
On the domain $\mathbb{R}\times\mathbb{T}$, we can introduce the operator $\Delta^{-1}$ which commutes with $\partial_i,i=1,2$, such that $u=\Delta^{-1}(f)$ is well defined for any $f\in C_0^\infty(\mathbb{R}\times\mathbb{T})$. Moreover, $\forall p\in(1,\infty)$, we can find a constant $C$ depending only on $p$ which guarantees 
\begin{equation}\label{C205}
\begin{split}
&\Delta u=f\ \mbox{in $\mathbb{R}\times\mathbb{T}$},~~\|\nabla^2 u\|_{L^p(\mathbb{R}\times\mathbb{T})}\leq C\|f\|_{L^p(\mathbb{R}\times\mathbb{T})}.
\end{split}
\end{equation}
In particular, if $f$ satisfies the cancellation condition in the $x_2$ direction, then $u$ satisfies it as well, say
\begin{equation}\label{C206}
\int_0^1 f(x_1,x_2)\,dx_2=0,\ \forall x_1\in\mathbb{R}\Rightarrow\int_0^1 u(x_1,x_2)\,dx_2=0,\ \forall x_1\in\mathbb{R}.
\end{equation} 
\end{lemma}

There is a minor modification for the corresponding elliptic estimates on $\mathbb{T}\times\mathbb{T}$.
\begin{lemma}\label{L24}
On the torus $\mathbb{T}\times\mathbb{T}$, we can introduce the operator $\Delta^{-1}$ which commutes with $\partial_i,i=1,2$, such that $u=\Delta^{-1}(f)$ is well defined for any $f\in C_0^\infty(\mathbb{T}\times\mathbb{T})$ satisfying the cancellation condition 
$$\int_{\mathbb{T}\times\mathbb{T}}f\,dx_1dx_2=0.$$ Moreover, $\forall p\in(1,\infty)$, we can find a constant $C$ depending only on $p$ which guarantees 
\begin{equation}\label{C207}
\begin{split}
\Delta u&=f\ \mbox{in $\mathbb{T}\times\mathbb{T}$},~~\int_{\mathbb{T}\times\mathbb{T}}u\,dx_1dx_2=0,\\
&\|\nabla^2 u\|_{L^p(\mathbb{T}\times\mathbb{T})}\leq C\|f\|_{L^p(\mathbb{T}\times\mathbb{T})}.
\end{split}
\end{equation}
\end{lemma}

Then we provide some basic interior estimates of the elliptic equations. For detailed discussions, one may refer to \cite[Chapter II, VI, and VIII]{GT}.
\begin{lemma}Suppose that $\Omega'$ and $\Omega$ are two smooth domains in $\mathbb{R}^n$ satisfying $\Omega'\subset\subset\Omega$.

1). Let $u$ be a harmonic function in $\Omega$, 
$$\Delta u=0\ \mbox{in $\Omega$}.$$
Then $\forall p\in[1,\infty]$, there is a constant $C$ depending on $n$, $p$, and the distance from $\partial\Omega'$ to $\partial\Omega$, such that 
\begin{equation}\label{C208}
\|u\|_{L^\infty(\Omega')}\leq C\,\|u\|_{L^p(\Omega)}.
\end{equation}

2). Let $u$ solve the Dirichlet problem of the Poisson equation in $\Omega$,
\begin{equation*}
\begin{cases}
\Delta u=f\ \mbox{in $\Omega$},\\
u=0\ \mbox{on $\partial\Omega$}.\\
\end{cases}
\end{equation*}
Then $\forall p\in(1,n/2)$, there is a constant $C$ depending on $n$, $p$, and the distance from $\partial\Omega'$ to $\partial\Omega$, such that
\begin{equation}\label{C209}
\|u\|_{L^\infty(\Omega')}\leq C\left(\sup_{x\in\Omega'}\int_\Omega\frac{|f(y)|}{|x-y|^{n-2}}\,dy+\|f\|_{L^p(\Omega)}\right).
\end{equation}
\end{lemma}
\begin{proof}
The estimate \eqref{C208} can be found in \cite[Chapter II]{GT}, thus we only prove \eqref{C209}. The standard $L^p$ elliptic estimate (see \cite[Chapter IX]{GT}) gives
\begin{equation}\label{C210}
\|u\|_{W^{2,p}(\Omega)}\leq C\,\|f\|_{L^p(\Omega)}.
\end{equation}

Let $\tilde{f}$ be the zero extension of $f$ to $\mathbb{R}^n$. We define $v=\Gamma\ast\tilde{f}$, where $\Gamma(x)$ is the fundamental solution of $\Delta$ in $\mathbb{R}^n$, and it holds that
\begin{equation}\label{C211}
|v(x)|\leq C\int_\Omega\frac{|f(y)|}{|x-y|^{n-2}}\,dy,~~\forall x\in\mathbb{R}^n. 
\end{equation}
Then Sobolev embedding theorem and the standard $L^p$ elliptic estimates (see \cite[Chapter IX]{GT}) ensure that, for some constant $C$ depending on $p$ and $1/q=1/p-2/n$,
$$\|v\|_{L^q(\mathbb{R}^n)}\leq C\,\|\nabla^2 v\|_{L^p(\mathbb{R}^n)}\leq C\,\|\tilde{f}\|_{L^p(\mathbb{R}^n)}\leq C\,\|f\|_{L^p(\mathbb{R}^n)}.$$
Moreover, the difference $w=u-v$ is harmonic in $\Omega$,
$$\Delta w=0,~~\mbox{in $\Omega$}.$$
Consequently, we apply \eqref{C208} and \eqref{C210} to deduce that
$$\|w\|_{L^\infty(\Omega')}\leq C\|w\|_{L^q(\Omega)}\leq C\big(\|u\|_{L^q(\Omega)}+\|v\|_{L^q(\Omega)}\big)\leq C\|f\|_{L^p(\Omega)},$$
which combined with \eqref{C211} gives \eqref{C209} and finishes the proof.
\end{proof}

We end this section by proving the modified Sobolev embedding theorem on $\mathbb{R}\times\mathbb{T}$.
\begin{lemma} 
Suppose that $f\in C_0^\infty(\mathbb{R}\times\mathbb{T})$ and satisfies the cancellation condition,
\begin{equation}\label{C212}
\int_0^1 f(x_1,x_2)\,dx_2=0,~~\forall x_1\in\mathbb{R}.
\end{equation}
Then for any $p\in[1,2)$, there is a constant $C$ depending only on $p$ such that
\begin{equation}\label{C213}
\|u\|_{L^q(\mathbb{R}\times\mathbb{T})}\leq C\|\nabla u\|_{L^p(\mathbb{R}\times\mathbb{T})},
\end{equation}
with $1/q=1/p-1/2$.
\end{lemma}
\begin{proof}
We first prove \eqref{C213} for $p=1$ and $q=2$.
According to \eqref{C212} and the intermediate value theorem, for any $(x_1,x_2)\in\mathbb{R}\times\mathbb{T}$, we can find $\xi(x_1)\in\mathbb{T}$ such that
\begin{equation}\label{CC201}
\begin{split}
f(x_1,x_2)=\int_{-\infty}^{x_1}\partial_1f(s,x_2)\,ds,~~ 
f(x_1,x_2)=\int_{\xi(x_1)}^{x_2}\partial_2f(x_1,s)\,ds.
\end{split}
\end{equation}
Consequently, for $\Omega=\mathbb{R}\times\mathbb{T}$, we calculate that
\begin{equation}\label{CC202}
\begin{split}
&\int_\Omega|f(x_1,x_2)|^2\,dx_1dx_2\\
&\leq\int_\Omega\left(\int_{-\infty}^{\infty}|\partial_1f(s,x_2)|\,ds\right)\cdot\left(\int_{\mathbb{T}}|\partial_2f(x_1,s)|\,ds\right)\,dx_1dx_2\\
&\leq C\left(\int_\Omega|\nabla f|\,dx_1dx_2\right)^2,
\end{split}
\end{equation}
which gives \eqref{C213} for $p=1$ and $q=2$. Then for any $q>2$, we check that
\begin{equation}\label{CC203}
\begin{split}
\int_\Omega|f(x_1,x_2)|^q\,dx_1dx_2 
&\leq\int_\Omega\left(|f(x_1,x_2)|^{q/2}-\int_0^1|f(x_1,s)|^{q/2}\,ds\right)^2\,dx_1dx_2\\
&\quad+\int_\mathbb{R}\left(\int_0^1|f(x_1,s)|^{q/2}\,ds\right)^2\,dx_1.
\end{split}
\end{equation}
For the first term on the right hand side of \eqref{CC203}, we apply \eqref{CC202} to derive that
\begin{equation}\label{CC204}
\begin{split}
&\int_\Omega\left(|f(x_1,x_2)|^{q/2}-\int_0^1|f(x_1,s)|^{q/2}\,ds\right)^2\,dx_1dx_2\\
&\leq C\left(\int_\Omega |f(x_1,x_2)|^{q/2-1}|\nabla f(x_1,x_2)|\,dx_1dx_2\right)^2\\
&\leq C\left(\int_\Omega|f(x_1,x_2)|^q\,dx_1dx_2\right)^{(q-2)/q}\left(\int_\Omega|\nabla f(x_1,x_2)|^{2q/(q+2)}\,dx_1dx_2\right)^{(q+2)/q}.
\end{split}
\end{equation}
While in view of \eqref{CC201}, we have
\begin{equation}\label{CC205}
\begin{split}
&\int_0^1|f(x_1,s)|^{q/2}\,ds\\
&\leq C\int_{-\infty}^{x_1}\int_0^1|f(t,s)|^{q/2-1}|\partial_1 f(t,s)|\,dtds\\
&\leq C \int_\Omega |f(x_1,x_2)|^{q/2-1}|\nabla f(x_1,x_2)|\,dx_1dx_2 \\
&\leq C\left(\int_\Omega|f(x_1,x_2)|^q\,dx_1dx_2\right)^{\frac{q-2}{2q}}\left(\int_\Omega|\nabla f(x_1,x_2)|^{2q/(q+2)}\,dx_1dx_2\right)^{\frac{q+2}{2q}},
\end{split}
\end{equation}
where the last line is due to \eqref{CC204}. Taking supremum of $x_1$ over $\mathbb{R}$ in \eqref{CC205} gives 
\begin{equation}\label{CC206}
\begin{split}
&\sup_{x_1\in\mathbb{R}}\int_0^1|f(x_1,s)|^{q/2}\,ds\\
&\leq C\left(\int_\Omega|f(x_1,x_2)|^q\,dx_1dx_2\right)^{\frac{q-2}{2q}}\left(\int_\Omega|\nabla f(x_1,x_2)|^{2q/(q+2)}\,dx_1dx_2\right)^{\frac{q+2}{2q}}.
\end{split}
\end{equation} 
Moreover, by virtue of \eqref{CC201}, we also check that
\begin{equation}\label{CC207}
\begin{split}
&\int_\Omega|f(x_1,x_2)|^{q/2}\,dx_1dx_2\\
&\leq C\int_\Omega|f(x_1,x_2)|^{q/2-1}\left(\int_0^1|\partial_2 f(x_1,s)|\,ds\right)dx_1dx_2\\
&\leq C\int_\mathbb{R}\left(\int_0^1|f(x_1,s)|^{q/2-1}\,ds\right)\cdot\left(\int_0^1|\partial_2 f(x_1,s)|\,ds\right)dx_1\\
&\leq C\left(\int_\Omega|f(x_1,s)|^q\,dx_1ds\right)^{(q-2)/2q}\left(\int_\Omega|\nabla f(x_1,s)|^{2q/(q+2)}\,dx_1ds\right)^{(q+2)/2q}.
\end{split}
\end{equation}
Combining \eqref{CC206} and \eqref{CC207} implies that the last line of \eqref{CC203} is bounded by
\begin{equation*}
\begin{split}
&\int_\mathbb{R}\left(\int_0^1|f(x_1,s)|^{q/2}\,ds\right)^2\,dx_1\\
&\leq\sup_{x_1\in\mathbb{R}}\int_0^1|f(x_1,s)|^{q/2}\,ds\cdot \int_\Omega|f(x_1,x_2)|^{q/2}\,dx_1dx_2\\
&\leq C\left(\int_\Omega|f(x_1,x_2)|^q\,dx_1dx_2\right)^{(q-2)/q}\left(\int_\Omega|\nabla f(x_1,x_2)|^{2q/(q+2)}\,dx_1dx_2\right)^{(q+2)/q},
\end{split}
\end{equation*}
which together with \eqref{CC204} leads to
\begin{equation}\label{CC208}
\begin{split}
&\int_\Omega|f(x_1,x_2)|^q\,dx_1dx_2\\
&\leq C\left(\int_\Omega|f(x_1,x_2)|^q\,dx_1dx_2\right)^{(q-2)/q}\left(\int_\Omega|\nabla f(x_1,x_2)|^{2q/(q+2)}\,dx_1dx_2\right)^{(q+2)/q}.
\end{split}
\end{equation}
Note that $f\in C_0^\infty(\Omega)$ means that $\|f\|_{L^q(\Omega)}$ is finite, thus we divide \eqref{CC208} by $\|f\|_{L^q(\Omega)}^{(q-2)/q}$ and arrive at
\begin{equation*}
\left(\int_\Omega|f(x_1,x_2)|^q\,dx_1dx_2\right)^{2/q}\leq C\left(\int_\Omega|f(x_1,x_2)|^p\,dx_1dx_2\right)^{2/p},
\end{equation*}
with $1/q=1/p-1/2$, which also gives \eqref{C213}    and finishes the proof.
\end{proof}

\section{A priori estimates}\label{S3}
\quad In this section, we will concentrate on a priori estimates of the system \eqref{C113} in the unbounded domain $\Omega=(0,\infty)\times\mathbb{T}$. While the corresponding arguments for the bounded domain $(0,1)\times\mathbb{T}$ follow by minor modifications.

Through out this section, we suppose that $(\rho^\varepsilon,u_1^\varepsilon,u_2^\varepsilon)$ is the approximation solution guaranteed by Lemma \ref{L21} satisfying
\eqref{C203}--\eqref{C111}. 
To simplify the notations, we omit the superscript $\varepsilon$, and the constant $C$ may vary form line to line, but it is independent of $\varepsilon$. Note that we will study the problem in $\Omega=(0,\infty)\times\mathbb{T}$, 
\begin{equation}\label{C301}
\begin{cases}
\frac{1}{x_1}\partial_{1}(x_1\rho u_{1})+\partial_{2}(\rho u_2)=0,\\
\frac{1}{x_1}\partial_1(x_1\rho u_1^2)+\partial_{2}(\rho u_1 u_2)-\frac{1}{x_1}\partial_{1}(x_1\partial_{1}u_1)-\partial_{22}u_1+\frac{1}{x_1^2}u_1+\partial_{1}P=g_1,\\
\frac{1}{x_1}\partial_1(x_1\rho u_1 u_2)+\partial_2(\rho u_2^2)-\frac{1}{x_1}\partial_1(x_1\partial_1u_2)-\partial_{22}u_2+\partial_2P=g_2,
\end{cases}
\end{equation}
where the pressure $P=\rho+\varepsilon \rho^{\alpha}$, and the boundary values are given by 
\begin{equation}\label{C302}
\begin{split}
u_1\big|_{\{x_1=0\}}&=0,~~\partial_1u_2\big|_{\{x_1=0\}}=0,\\
&u\rightarrow 0\ \mbox{as $x_1\rightarrow\infty$}.
\end{split}
\end{equation}
Meanwhile, the density $\rho$ satisfies the condition
\begin{equation}\label{CC503}
\int_\Omega P\,\omega\,dx_1dx_2= M,
\end{equation}
with the weigh function $\omega(x_1)$ satisfying
\begin{equation*} 
\begin{cases}
\omega(x_1)=x_1\ \forall x_1\in[0,1],\\
\omega(x_1)=x_1^{-2}\ \forall x_1\in[2,\infty),\\
1/4<\omega\leq 4\ \forall x_1\in[1,2].
\end{cases}
\end{equation*}

In addition, supposing that $\hat x=(\hat x',\hat x_3)\in\mathbb{R}^2\times\mathbb{T}$, we define $(\hat\rho,\hat u)$ by
\begin{equation}\label{C346}
\hat\rho(\hat{x})=\rho(|\hat{x}'|,\hat{x}_3),~~\hat{u}(\hat{x})=\left(u_1(|\hat{x}'|,\hat{x}_3)\frac{\hat{x}'}{|\hat{x}'|},\,u_2(|\hat{x}'|,\hat{x}_3)\right).
\end{equation}
Then, according to Remark \ref{R12}, $(\hat\rho,\hat u)$ is a weak solution to the following system in $\hat\Omega$,
\begin{equation}\label{C350}
\begin{cases}
\mathrm{div}(\hat\rho\, \hat u)=0,\\
\mathrm{div}(\hat\rho\, \hat u\otimes \hat u)-\Delta\hat u+\nabla\hat P=\hat g,
\end{cases}
\end{equation}
where $\hat P=\hat\rho+\varepsilon\hat\rho^{\alpha}$ and $\hat g=(g_1\cdot\hat{x}'/|\hat x'|,\, g_2)$.

Let us introduce another weight function $\Phi(x_1)\in C^\infty([0,\infty))$ to measure the integrability of $\rho|u|^2$. It satisfies
\begin{equation}\label{C309}
\begin{cases}
\Phi(x_1)=x_1~~ \forall x_1\in[0,2],\\
\Phi(x_1)=x_1^{-5}~~ \forall x_1\in[4,\infty),\\
0\leq\Phi(x_1)\leq 4~~ \forall x_1\in[0,\infty).
\end{cases}
\end{equation}
Note that $\Phi(x_1)$ decays faster than $\omega(x_1)$ as $x_1\rightarrow\infty$. The crucial estimates are given by the next lemma, which illustrates $\rho|u|^2\in\mathcal{L}_{loc}^1(\Omega)$.  
\begin{lemma}\label{L31}
Under the conditions of Lemma \ref{T1}, there is a constant $C$ depending on $\Phi$, such that
\begin{equation}\label{C310}
\int_\Omega\rho|u|^2\Phi(x_1)\,dx_1dx_2\leq C.
\end{equation}
\end{lemma}

The proof of Lemma \ref{L31} consists of three parts. We first establish the overview of the analysis on the domain $(0,\infty)\times\mathbb{T}$ in Subsection \ref{SS30}, which reflects the essential features of our method.
Then we consider
the domain away from the axis in Subsection \ref{SS31}, and end up with the estimates near the axis in Subsection \ref{SS32}. After finishing Lemma \ref{L31}, in Subsection \ref{SS33}, we will extend the arguments on $(0,\infty)\times\mathbb{T}$ to the bounded domain $(0,1)\times\mathbb{T}$ and shortly sketch the proof.

\subsection{Analysis on the tube domain $(0,\infty)\times\mathbb{T}$}\label{SS30}
\quad  Let us start with the basic energy estimates of the system. 
\begin{proposition}\label{P31}
Under the conditions of Lemma \ref{T1}, there is a constant $C$ determined by the external force $g$, such that
\begin{equation}\label{C303}
\int_\Omega\left(|\nabla u|^2x_1+\frac{|u_1|^2}{x_1}\right)dx_1dx_2\leq C.
\end{equation}
\end{proposition}
\begin{proof}
Adding \eqref{C301}$_2$ multiplied by $u_1\cdot x_1$ and \eqref{C301}$_3$ multiplied by $u_2\cdot x_1$ together, then integrating over $\Omega$ yields
\begin{equation}\label{C304}
\int_\Omega\left(|\nabla u|^2x_1+\frac{|u_1|^2}{x_1}\right)dx_1dx_2=\int_\Omega g\cdot u\,x_1dx_1dx_2,
\end{equation}
where we have applied \eqref{C108} and the fact $u_i\cdot\partial_iu_j\cdot u_j=u\cdot\nabla(|u|^2/2)$.

Let us expand $u$ and $g$ in the $x_2$ direction by
\begin{equation*}
u=\sum_{k=-\infty}^\infty\alpha_k(x_1)\, e^{2\pi ikx_2},~~g=\sum_{k=-\infty}^\infty\beta_k(x_1)\, e^{2\pi ikx_2}.
\end{equation*}
Note that \eqref{C112} ensures $\beta_0(x_1)\equiv 0$. Thus, the right hand side of \eqref{C304} is handled by
\begin{equation*}
\begin{split}
&\int_{0}^\infty\big(\int_0^1g\cdot u\, dx_2\big)\cdot x_1dx_1\\
&=\int_0^\infty\big(\sum_{k\neq 0}\alpha_k(x_1)\cdot\beta_k(x_1)\big)\,x_1dx_1\\
&\leq\big(\sum_{k\neq 0}\int_0^\infty k^2|\alpha_k(x_1)|^2x_1dx_1\big)^{\frac{1}{2}}\big(\sum_{k\neq 0}\int_0^\infty k^{-2}|\beta_k(x_1)|^2x_1dx_1\big)^{\frac{1}{2}}\\
&\leq C\,\big(\int_\Omega|\partial_2 u|^2x_1dx_1dx_2\big)^\frac{1}{2},
\end{split}
\end{equation*}
where the last line is due to \eqref{C112}.  
Substituting the above results into \eqref{C304} and applying Cauchy-Schwarz's inequality, we arrive at
$$\int_\Omega\left(|\nabla u|^2x_1+\frac{|u_1|^2}{x_1}\right)dx_1dx_2\leq C.$$
The proof of Proposition \ref{P31} is therefore completed.
\end{proof}

Next proposition deals with the weighted estimates of $\rho$ and $\rho |u_1|^2$, totally speaking, the estimates of the radial component $u_1$ are better than that of the axial component $u_2$. We introduce a truncation function $\zeta_0(x_1)\in C^\infty([0,\infty))$ by
\begin{equation}\label{C382}
\begin{cases}
\zeta_0=1~~\forall x_1\in[0,1],\\
\zeta_0=0~~\forall x_1\in[2,\infty),\\
0\leq\zeta_0\leq 1~~\forall x_1\in[0,\infty).
\end{cases}
\end{equation}
In other words, $\zeta_0$ is a truncation near the symmetric axis, while $\zeta_1\triangleq(1-\zeta_0)$ is a truncation away from the axis. Let us establish the crude estimates on $\rho u_1$ and $P$.
\begin{proposition}\label{P32}
Under the conditions of Lemma \ref{T1}, for any $m\geq 6$ and $0<{\varepsilon_0}\leq 1$, there is a constant $C$ depending only on $m$ and $\varepsilon_0$ such that
\begin{equation}\label{C305}
\begin{split}
\int_\Omega\big(\rho u_1^2+P\big)\,( x_1^{\varepsilon_0}\cdot\zeta_0+x_1^{-m}\cdot\zeta_1)\,dx_1dx_2\leq C\int_A \rho u_1^2\,dx_1dx_2+C,\\
\end{split}
\end{equation}
where $A=[10,20]\times\mathbb{T}\subset\Omega$. 
\end{proposition}
\begin{proof}
We mention that the first weight $x_1^{\varepsilon_0}\cdot\zeta_0$ deals with the estimate near the axis, while the second weight $x_1^{-m}\cdot\zeta_1$ deals with the domain away from the axis. 

Multiplying \eqref{C301}$_2$ by $x_1^{1+\varepsilon_0}\cdot\zeta_A$ with $\zeta_A(x_1)\triangleq \zeta_0(x_1/10)$ and integrating over $\Omega$ gives
\begin{equation*}
\begin{split}
&\int_\Omega\big(\varepsilon_0\,\rho u_1^2+(1+\varepsilon_0)P\big)\,x_1^{\varepsilon_0}\cdot\zeta_A\,dx_1dx_2\\
&=\int_\Omega\big(u_1 x_1^{-1+\varepsilon_0}+\varepsilon_0\,\partial_1u_1\,x_1^{\varepsilon_0}\big)\,\zeta_A\,dx_1dx_2
-\int_\Omega\big(\rho u_1^2+P-\partial_1u_1\big)x_1^{1+\varepsilon_0}\cdot\zeta'_A\,dx_1dx_2,\\
&\leq C(\varepsilon_0)\,\left(\int_\Omega\big(\frac{|u_1|^2}{x_1}+|\nabla u_1|^2x_1\big)dx_1dx_2\right)^{\frac{1}{2}}+C\int_A \rho u_1^2\,dx_1dx_2\\
&\leq C\int_A\rho u_1^2\,dx_1dx_2+C,
\end{split}
\end{equation*}
where the last two lines are due to \eqref{CC503} and \eqref{C303}. We therefore obtain that
\begin{equation}\label{C306}
\int_\Omega\big(\rho|u_1|^2+P\big)\,x_1^{\varepsilon_0}\cdot\zeta_A\,dx_1dx_2\leq C\int_A \rho u_1^2\,dx_1dx_2+C.
\end{equation}
Note that $\zeta_0\leq\zeta_A$, \eqref{C306} directly gives the first part of \eqref{C305}.

Similarly, let us multiply \eqref{C301}$_2$ by $x_1^{-m}\cdot\zeta_1$ and deduce that
\begin{equation*}
\begin{split}
\int_\Omega\big(\rho u_1^2+P\big)\,x_1^{-m}\cdot\zeta_1\,dx_1dx_2
&\leq C\int_\Omega(\rho u_1^2+P)\cdot|\zeta'_1|\,dx_1dx_2+C\\
&\leq C\int_A \rho u_1^2\,dx_1dx_2+C,
\end{split}
\end{equation*}
which is due to \eqref{C306} and the fact $|\zeta'_1|\leq C\cdot x_1^{\varepsilon_0}\cdot\zeta_A$ for any $x_1\in[1,2]$.
\end{proof}

Next, we consider $\zeta_1\triangleq(1-\zeta_0)$ which is a truncation away from the axis, then by virtue of \eqref{CC503}, we have
\begin{equation}\label{CC506}
\int_\Omega P\cdot x_1^{-2}\zeta_1\,dx_1dx_2\leq M.
\end{equation}
The key step of obtaining the global compactness assertions on $(0,\infty)\times\mathbb{T}$ is to improve \eqref{CC506} as below.
\begin{proposition}\label{P52}
Under the conditions of Lemma \ref{T12}, there is a constant $C$ depending on $M$, such that
\begin{equation}\label{C507}
\int_\Omega P\cdot x_1^{-9/8} \zeta_1\,dx_1dx_2
+\int_\Omega\rho u_1^2\cdot x_1^{-1/16}\zeta_1\,dx_1dx_2\leq C.
\end{equation}
\end{proposition}

\begin{proof}
Integrating \eqref{C301}$_1$ in the $x_2$ direction gives
\begin{equation*}
\frac{1}{x_1}\,\partial_1\big(\int_0^1 x_1\,\rho u_1 \,dx_2\big)=0\ \Rightarrow\  x_1\,\int_0^1\rho u_1\,dx_2=C_1.
\end{equation*}
Note that the boundary condition \eqref{C302} enforces $\int\rho u_1\,dx_2\big|_{x_1=0}=0$, thus the constant $C_1$ must be zero and we declare that the cancellation condition holds,
\begin{equation}\label{C317}
\int_0^1\rho u_1(x_1,x_2)\,dx_2=0,~~\forall x_1\in[0,\infty).
\end{equation}

Moreover, integrating \eqref{C301}$_2$ with respect to $x_2$ gives
\begin{equation}\label{C509}
\partial_1(\int_0^1 P\,dx_2)+\frac{1}{x_1}\,\partial_1(\int_0^1 x_1\,\rho u_1^2 \,dx_2)-\frac{1}{x_1}\partial_1(\int_0^1 x_1\,\partial_1u_1\,dx_2)=-\frac{1}{x_1^2}\int_0^1 u_1\,dx_2.
\end{equation}

\textit{Step 1.}
Let us define the truncation $\eta(x_1)=\zeta_1(4x_1)$ and introduce the test function 
$$\psi(x_1)=\big(\int_{x_1}^\infty \int_0^1 P\cdot s^{-2}\, dsdx_2\big)\cdot\eta(x_1),$$
which is well defined and bounded uniformly due to \eqref{CC503}. Multiplying \eqref{C509} by $\psi\cdot\eta$ and integrating with respect to $x_1$, we check each term in details.

With the help of \eqref{CC503}, the first term on the left hand side of \eqref{C509} is handled by
\begin{equation}\label{C510}
\begin{split}
&\int_0^\infty \partial_1\big(\int_0^1 P\,dx_2\big)\cdot\psi\,\eta\,dx_1\\
&=\int_0^\infty\big(\int_0^1 P\,dx_2\big)^2\cdot x_1^{-2}\eta^2\,dx_1-\int_\Omega P\cdot\psi\,\eta'\,dx_1dx_2\\
&\geq\int_0^\infty\big(\int_0^1P\,dx_2\big)^2\cdot x_1^{-2}\eta\,dx_1-C.
\end{split}
\end{equation}

Then, the second term on the left is bounded by
\begin{equation}\label{C511}
\begin{split}
&\int_0^\infty\frac{1}{x_1}\,\partial_1\big(\int_0^1 x_1\,\rho u_1^2 \,dx_2\big)\cdot\psi\,\eta\,dx_1\\
&=\int_0^\infty\big(\int_0^1\rho u_1^2\,dx_2\big)\,\big(\int_0^1 P\,dx_2\big)\cdot x_1^{-2}\eta^2\,dx_1
+\int_0^\infty x_1^{-1} (\int_0^1\rho u_1^2\,dx_2)\cdot\psi\,\eta\,dx_1\\
&\quad-\int_0^\infty\big(\int_0^1\rho u_1^2\,dx_2\big)\cdot\psi\,\eta'\,dx_1\\
&\geq\int_0^\infty\big(\int_0^1\rho u_1^2\,dx_2\big)\cdot\big(\int_0^1 P\,dx_2\big)\cdot x_1^{-2}\eta^2\,dx_1-C\int_A\rho|u|^2\,dx_1dx_2-C.
\end{split}
\end{equation}
Note that the last line is due to \eqref{C305}   and the second positive term on the second line of \eqref{C511} has been dropped out.

Similarly, we also argue that
\begin{equation}\label{C512}
\begin{split}
&\int_0^\infty\frac{1}{x_1}\partial_1\big(\int_0^1 x_1\,\partial_1u_1\,dx_2\big)\cdot\psi\,\eta\,dx_1\\
&=\int_0^\infty(\int_0^1\partial_1u_1\,dx_2)\,(\int_0^1 P\,dx_2)\cdot x_1^{-2}\eta^2\,dx_1+\int_0^\infty x_1^{-1}(\int_0^1\partial_1u_1\,dx_2)\cdot\psi\,\eta\,dx_1\\
&\quad-\int_0^\infty\big(\int_0^1\partial_1u_1\,dx_2\big)\cdot\psi\,\eta'\,dx_1\\
&\leq C\left(\int_\Omega|\partial_1u_1|^2x_1\,dx_1dx_2\right)^{\frac{1}{2}}\left(\int_0^\infty\big(\int_0^1P\,dx_2\big)^2\cdot x_1^{-2}\eta\,dx_1+1\right)^{\frac{1}{2}}\\
&\leq\frac{1}{2}\int_0^\infty\left(\int_0^1 P\,dx_2\right)^2\cdot x_1^{-2}\eta\,dx_1+C.
\end{split}
\end{equation}
We mention that the last line of \eqref{C512} is due to \eqref{C303}.
Moreover, in view of \eqref{C303}, the right hand side of \eqref{C509} is directly bounded by
\begin{equation}\label{C513}
-\int_0^\infty\frac{1}{x_1^2}\left(\int_0^1 u_1\,dx_2\right)\cdot\psi\,\eta\,dx_1\leq C\,\left(\int_\Omega\frac{|u_1|^2}{x_1}\,dx_1dx_2\right)^{\frac{1}{2}}\leq C.
\end{equation}

Thus, by combining \eqref{C510}--\eqref{C513}, we arrive at
\begin{equation}\label{C514}
\begin{split}
\int_0^\infty\big(\int_0^1\rho u_1^2\,dx_2\big)\cdot\big(\int_0^1P\,dx_2\big)\cdot x_1^{-2}\eta^2\,dx_1\leq C\int_A\rho u_1^2\,dx_1dx_2+C.
\end{split}
\end{equation}
In addition, by virtue of \eqref{C303} and \eqref{C317}, we infer that
\begin{equation}\label{C515}
\begin{split}
\int_A\rho u_1^2\,dx_1dx_2
&=\int_{10}^{20}\left(\int_0^1\rho u_1\big(u_1-\int_0^1 u_1\,dx_2\big)\,dx_2\right)\,dx_1\\
&\leq\int_{10}^{20}\left(\int_0^1\rho |u_1|\,dx_2\right)\cdot\left(\int_0^1|\partial_2u_1|\,dx_2\right)\,dx_1\\
&\leq\left(\int_{10}^{20}\big(\int_0^1\rho|u_1|\,dx_2\big)^2x_1^{-1}dx_1\right)^{\frac{1}{2}}
\left(\int _\Omega|\nabla u|^2\,x_1dx_1dx_2\right)^{\frac{1}{2}}\\
&\leq C\,\left(\int_0^\infty\big(\int_0^1 P\,dx_2\big)\cdot\big(\int_0^1\rho u_1^2\,dx_2\big)\cdot x_1^{-2}\eta^2\,dx_1\right)^{\frac{1}{2}}.
\end{split}
\end{equation}
Substituting \eqref{C515} into \eqref{C514} also leads to
\begin{equation}\label{C516}
\begin{split}
\int_0^\infty\big(\int_0^1\rho u_1^2\,dx_2\big)&\cdot\big(\int_0^1 P\,dx_2\big)\cdot x_1^{-2}\eta^2\,dx_1\leq C.\\
\end{split}
\end{equation}
Then, \eqref{C516} together with \eqref{C305} and \eqref{C515} also gives
\begin{equation}\label{C517}
\int_\Omega\big(\rho u_1^2+P\big)\,\big(x_1^{\varepsilon_0}\cdot\zeta_0+x_1^{-m}\cdot\zeta_1\big)\,dx_1dx_2\leq C.
\end{equation}

\textit{Step 2.}
Now, we multiply \eqref{C509} by $(x_1^{-1/8}\eta)$ and repeat the process given in \eqref{C510}--\eqref{C513}. Note that \eqref{CC503} ensures that
\begin{equation*} 
\begin{split}
&\int_0^\infty\partial_1\big(\int_0^1 P\,dx_2\big)\cdot x_1^{-1/8}\eta\,dx_1\\
&=\frac{1}{8}\int_\Omega P\cdot x_1^{-9/8}\eta\,dx_1dx_2
-C\int_\Omega P\cdot x_1^{-1/8}\eta'\,dx_1dx_2\\
&\geq\frac{1}{8}\int_\Omega P\cdot x_1^{-9/8}\eta\,dx_1dx_2-C.
\end{split}
\end{equation*}
Next we apply \eqref{C517} to deduce that
\begin{equation*}
\begin{split}
&\int_0^\infty\frac{1}{x_1}\,\partial_1\big(\int_0^1 x_1\,\rho u_1^2 \,dx_2\big)\cdot x_1^{-1/8}\eta\,dx_1\\
&=\frac{1}{8}\int_\Omega\rho u_1^2\cdot x_1^{-9/8}\eta\,dx_1dx_2-\int_\Omega\rho u_1^2\cdot x_1^{-1/8}\eta'\,dx_1dx_2\\
&\geq\frac{1}{8}\int_\Omega\rho u_1^2\cdot x_1^{-9/8}\eta\,dx_1dx_2-C.
\end{split}
\end{equation*}
Moreover, \eqref{C303} also leads to
\begin{equation*}
\begin{split}
&\int_0^\infty\frac{1}{x_1}\partial_1\big(\int_0^1 x_1\,\partial_1u_1\,dx_2\big)
\cdot x_1^{-1/8}\eta\,dx_1\\
&=\frac{9}{8}\int_0^\infty\big(\int_0^1\partial_1u_1\,dx_2\big)\cdot x_1^{-9/8}\eta\,dx_1-\int_0^\infty\big(\int_0^1\partial_1u_1\,dx_2\big)\cdot x_1^{-1/8}\eta'\,dx_1\\
&\leq C\left(\int_\Omega|\partial_1u_1|^2x_1 dx_1dx_2\right)^{\frac{1}{2}}\leq C,
\end{split}
\end{equation*}
and the estimate
\begin{equation*}
\begin{split}
-\int_0^\infty\frac{1}{x_1^2}\left(\int_0^1 u_1\,dx_2\right)\cdot
x_1^{-1/8}\eta\,dx_1\leq C\left(\int_\Omega\frac{|u_1|^2}{x_1} dx_1dx_2\right)^{\frac{1}{2}}\leq C. 
\end{split}
\end{equation*}
In conclusion, we declare that
\begin{equation}\label{C518}
\int_\Omega\big(P+\rho u_1^2\big)\cdot x_1^{-9/8}\eta\,dx_1dx_2 \leq C,
\end{equation}
which gives the first part of \eqref{C507}.

\textit{Step 3.} Let us introduce another test function
$$\phi(x_1)=\big(\int_{x_1}^\infty\int_0^1 P\cdot s^{-9/8}\,dsdx_2\big)\cdot\eta,$$
which is well defined and bounded uniformly in view of \eqref{C518}. We multiply \eqref{C509} by $\phi\cdot\eta$ and carry out the same calculations as \eqref{C510}--\eqref{C513} to deduce that
\begin{equation}\label{C519}
\int_0^\infty\big(\int_0^1\rho u_1^2\,dx_2\big)\cdot\big(\int_0^1P\,dx_2\big)\cdot x_1^{-9/8}\eta^2\,dx_1\leq C.
\end{equation}
In particular, we apply the method given by \eqref{C515} and check that
\begin{equation*} 
\begin{split}
&\int_\Omega\rho u_1^2\cdot x_1^{-1/16}\eta^2\,dx_1dx_2\\
&\leq\int_0^\infty \big(\int_0^1\rho|u_1|\,dx_2\big)\cdot\big(\int_0^1|\partial_2u_1|\,dx_2\big)\cdot x_1^{-1/16}\eta^2\,dx_1\\
&\leq C\left(\int_0^\infty\big(\int_0^1\rho|u_1|\,dx_2\big)^2 x_1^{-9/8}\eta^2\,dx_1\right)^{\frac{1}{2}}
\left(\int _\Omega|\nabla u|^2x_1dx_1dx_2\right)^{\frac{1}{2}}\\
&\leq C\,\left(\int_0^\infty\big(\int_0^1P\,dx_2\big)\cdot\big(\int_0^1\rho u_1^2\,dx_2\big)\cdot x_1^{-9/8}\eta^2\,dx_1\right)^{\frac{1}{2}},
\end{split}
\end{equation*}
which combined with \eqref{C518} and \eqref{C519} gives \eqref{C507} and finishes the proof.
\end{proof}

As a direct corollary of Propositions \ref{P31}--\ref{P52}, we give upper bound of $P_\infty$ which describes the behaviour of $\rho$ at the far field.
\begin{corollary}\label{C51}
Under the conditions of Lemma \ref{T1}, there is  a constant $C$ depending on $M$, such that $P_\infty\leq C.$
\end{corollary}
\begin{proof}
Note that the regularity condition \eqref{C118} ensures that we can integrate \eqref{C509} over $[x_1,\infty)$ to deduce that
\begin{equation*}
\begin{split}
&P_\infty-\int_0^1\big(P+\rho u_1^2-\partial_1u_1\big)(x_1,x_2)\,dx_2\\
&=-\int_{x_1}^\infty\int_0^1\big(\rho u_1^2\cdot s^{-1}-\partial_1u_1\cdot s^{-1}+u_1\cdot s^{-2}\big)\,dsdx_2,
\end{split}
\end{equation*}
which combined with \eqref{C303} and \eqref{C507} yields that, for $x_1\geq 1$
\begin{equation*}
P_\infty\leq \int_0^1\big(P+\rho u_1^2+|\partial_1u_1|\big)(x_1,x_2)\,dx_2+C.
\end{equation*}
Let us multiply the above inequality by $\omega$ and integrate over $[2,\infty)$, then we make use of \eqref{C116}, \eqref{CC503}, \eqref{C303}, and \eqref{C507} to derive that
\begin{equation*}
P_\infty\leq C.
\end{equation*}
The proof is therefore completed.
\end{proof}

Now, we turn to the most difficult part which deals with the estimates of $P-P_\infty$. The sign of $P-P_\infty$ is undetermined, thus we must carefully analysis the behaviour of such term as $x_1\rightarrow\infty$.

\begin{proposition}
Under the conditions of Theorem \ref{T12}, for any $x_1\geq 3/2$, there is a constant $C$ independent of $x_1$, such that 
\begin{equation}\label{C520}
\begin{split}
&\bigg|\int_{x_1}^\infty \int_0^1\big(P-P_\infty\big)\cdot s^{-1/2}\,dsdx_2\bigg|\\
&\leq C\cdot x_1^{3/4}\left(\int_0^\infty\big(\int_0^1(P-P_\infty)\,dx_2\big)^2x_1^{-1/2}\,\zeta_1^2\,dx_1+1\right)^{\frac{1}{2}}.
\end{split}
\end{equation}
\end{proposition}
\begin{remark}
We mention that the regularity conditions \eqref{C203}--\eqref{C118} are sufficient to ensure that the crucial term
\begin{equation}\label{crucial}
\int_0^\infty\big(\int_0^1(P-P_\infty)\,dx_2\big)^2x_1^{-1/2}\,\zeta_1^2\,dx_1
\end{equation}
is finite $($see \cite[Section 6.8]{L2}$)$. However, the upper bound of \eqref{crucial} still depends on $\varepsilon$ at the present stage. 
\end{remark}
\begin{proof}
We introduce the truncation function $\zeta_{x_1}(s)=\zeta_1(s-x_1)$ for $x_1\geq 0$ and write \eqref{C509} in the variable $s$,
\begin{equation}\label{C521}
-\partial_s\int_0^1(P-P_\infty)\,dx_2-\frac{1}{s}\,\partial_s\int_0^1 s\,\rho u_1^2 \,dx_2+\frac{1}{s}\,\partial_s\int_0^1s\,\partial_1u_1\,dx_2=\frac{1}{s^2}\int_0^1u_1\,dx_2.
\end{equation}
Multiplying \eqref{C521} by $s^{1/2}\zeta_{x_1}(s)$ and integrating with respect to $s$, we check each term in details.

The first term on the left hand side of \eqref{C521} is treated by
\begin{equation}\label{C522}
\begin{split}
&-\int_0^\infty\partial_s\big(\int_0^1(P-P_\infty)\,dx_2\big)\cdot s^{1/2}\zeta_{x_1}\,ds\\
&=\frac{1}{2}\int_\Omega(P-P_\infty)\cdot s^{-1/2}\zeta_{x_1}\,dsdx_2+\int_\Omega(P-P_\infty)\cdot s^{1/2}\zeta_{x_1}'\,dsdx_2.\\
\end{split}
\end{equation}
We specially mention that the regularity condition \eqref{C118} guarantees that the integration by parts in the second line of \eqref{C522} is available.  

Then we apply \eqref{C507} to deduce that the second term is bounded by
\begin{equation}\label{C523}
\begin{split}
&\bigg|\int_0^\infty\frac{1}{s}\,\partial_s\big(\int_0^1 s\,\rho u_1^2 \,dx_2\big)\cdot s^{1/2}\zeta_{x_1}\,ds\bigg|\\
&\leq\bigg|\frac{1}{2}\int_\Omega\rho u_1^2\cdot s^{-1/2}\zeta_{x_1}\,dsdx_2\bigg|+\bigg|\int_\Omega\rho u_1^2\cdot s^{1/2}\zeta_{x_1}'\,dsdx_2\bigg|\\
&\leq C\int_\Omega\rho u_1^2\cdot s^{-1/16}\zeta_{x_1}\,dsdx_2+\int_\Omega\rho u_1^2\cdot s^{1/2}|\zeta_{x_1}'|\,dsdx_2\\
&\leq C\left(1+x_1^{9/16}\right)\int_\Omega\rho u_1^2\cdot s^{-1/16}\zeta_{x_1}\,dsdx_2\leq C\left(1+x_1^{9/16}\right).
\end{split}
\end{equation}
To handle the remaining terms, we utilize \eqref{C303} to deduce that
\begin{equation}\label{C524}
\begin{split}
&\bigg|\int_0^\infty\frac{1}{s^2}\big(\int_0^1 u_1\,dx_2\big)\cdot s^{1/2}\zeta_{x_1}\,ds\bigg|\leq C\left(\int_\Omega\frac{|u_1|^2}{x_1}\,dx_1dx_2\right)^{\frac{1}{2}}\leq C.\\
\end{split}
\end{equation}
Moreover, we also have
\begin{equation}\label{C545}
\begin{split}
&\bigg|\int_0^\infty\frac{1}{s}\,\partial_s\big(\int_0^1s\,\partial_1u_1\,dx_2\big)\cdot s^{1/2}\zeta_{x_1}\,ds\bigg|\\
&\leq\bigg|\frac{1}{2}\int_\Omega\partial_1u_1\cdot s^{-1/2}\zeta_{x_1}\,dsdx_2\bigg|+\bigg|\int_\Omega\partial_1u_1\cdot s^{1/2}\zeta_{x_1}'\,dxdx_2\bigg|\\
&\leq C\left(\int_\Omega|\nabla u_1|^2x_1\,dx_1dx_2\right)^{\frac{1}{2}}\leq C,
\end{split}
\end{equation}
where the last line is due to $|\mathrm{supp}\,\zeta_{x_1}'|\leq 1$.
Combining \eqref{C522}--\eqref{C545}, we arrive at
\begin{equation}\label{C525}
\begin{split}
&\frac{1}{2}\bigg|\int_\Omega(P-P_\infty)\cdot s^{-1/2}\zeta_{x_1}\,dsdx_2\bigg|\\
&\leq C\big(1+x_1^{9/16}\big)+\bigg|\int_\Omega(P-P_\infty)\cdot s^{1/2}\zeta_{x_1}'\,dsdx_2\bigg|.
\end{split}
\end{equation}

There are two cases we must distinguish. If $1\leq x_1\leq 10$, with the help of \eqref{CC503} and Corollary \ref{C51}, we directly declare that
\begin{equation*}
\bigg|\int_\Omega(P-P_\infty)\cdot s^{1/2}\zeta_{x_1}'\,dsdx_2\bigg|\leq C.
\end{equation*}
While for the case $x_1>10$, we check that
\begin{equation*}
\begin{split}
&\bigg|\int_\Omega(P-P_\infty)\cdot s^{1/2}\zeta_{x_1}'\,dsdx_2\bigg|\\
&\leq\left(\int_0^\infty\big(\int_0^1(P-P_\infty)\,dx_2\big)^2x_1^{-1/2}\,|\zeta_{x_1}'|\,dx_1\right)^{\frac{1}{2}}\cdot\left(\int_0^\infty s^{3/2}|\zeta_{x_1}'|\,ds\right)^{\frac{1}{2}}\\
&\leq C\,x_1^{3/4}\left(\int_0^\infty\big(\int_0^1(P-P_\infty)\,dx_2\big)^2x_1^{-1/2}\,\zeta_1^2\,dx_1\right)^{\frac{1}{2}}.
\end{split}
\end{equation*}
Substituting these estimates into \eqref{C525}, we arrive at
\begin{equation}\label{C526}
\begin{split}
&\bigg|\int_\Omega\big(P-P_\infty\big)\cdot s^{-1/2}\zeta_{x_1}\,dsdx_2\bigg|\\
&\leq C\cdot x_1^{3/4}\left(\int_0^\infty\big(\int_0^1(P-P_\infty)\,dx_2\big)^2x_1^{-1/2}\,\zeta_1^2\,dx_1+1\right)^{\frac{1}{2}}.
\end{split}
\end{equation}
Note that for $x_1\geq 3/2$, we have
\begin{equation*}
\begin{split}
&\bigg|\int_{x_1}^\infty \int_0^1\big(P-P_\infty\big)\cdot s^{-1/2}\,dsdx_2\bigg|\\
&\leq\bigg|\int_\Omega\big(P-P_\infty\big)\cdot s^{-1/2}\zeta_{x_1}\,dsdx_2\bigg|+\bigg|\int_{x_1}^\infty\int_0^1(P-P_\infty) \cdot s^{-1/2}(1-\zeta_{x_1})\,dsdx_2\bigg|\\
&\leq\bigg|\int_\Omega\big(P-P_\infty\big)\cdot s^{-1/2}\zeta_{x_1}\,dsdx_2\bigg|+C\,x_1^{5/8}\int_{x_1}^\infty\int_0^1(P+P_\infty)\cdot s^{-9/8}\,dsdx_2,
\end{split}
\end{equation*}
which combined with \eqref{C507}, \eqref{C526} and Corollary \ref{C51} gives 
\begin{equation*} 
\begin{split}
&\bigg|\int_{x_1}^\infty \int_0^1\big(P-P_\infty\big)\cdot s^{-1/2}\,dsdx_2\bigg|\\
&\leq C\cdot x_1^{3/4}\left(\int_0^\infty\big(\int_0^1(P-P_\infty)\,dx_2\big)^2x_1^{-1/2}\,\zeta_1^2\,dx_1+1\right)^{\frac{1}{2}}.
\end{split}
\end{equation*} 
The proof is therefore completed.
\end{proof}

Then we establish the necessary estimates of the axial component $u_2$, which make full use of the structure of the equations \eqref{C301}. As a direct corollary, we also obtain the proper (weighted) estimates on $\|u\|_{L^2(\Omega)}$.
\begin{proposition}\label{P35}
Under the conditions of Lemma \ref{T1}, there is a constant $C$ such that $\forall x_1\geq 2$, we have
\begin{equation}\label{C316}
\bigg|\int_0^1 u_2(x_1,x_2)\,dx_2\bigg|\leq C.
\end{equation}
In particular, we argue that
\begin{equation}\label{C334}
\int_{2}^{\infty}\int_0^1 \big(|u|^2+|\nabla u|^2)x_1^{-2}dx_1dx_2\leq C.
\end{equation}
\end{proposition}
\begin{proof}
Recalling that
integrating \eqref{C301}$_1$ in the $x_2$ direction leads to the cancellation condition \eqref{C317}
\begin{equation*} 
\int_0^1\rho u_1(x_1,x_2)\,dx_2=0,~~\forall x_1\in[0,\infty).
\end{equation*}

Similarly, we integrate \eqref{C301}$_3$ in the $x_2$ direction and arrive at
\begin{equation*}
\frac{1}{x_1}\, \partial_1\left(\int_0^1 x_1\,\big(\rho u_1u_2-\partial_1u_2\big)\,dx_2\right)=0\ \Rightarrow\ x_1 \int_0^1\big(\rho u_1u_2-\partial_1u_2\big)\, dx_2=C_2.
\end{equation*}
Once more, the boundary condition \eqref{C302} guarantees $\int_0^1(\rho u_1u_2-\partial_1u_2)\,dx_2\big|_{x_1=0}=0$, as a result, the constant $C_2$ must be 0 and we argue that
\begin{equation}\label{C318}
\partial_1\big(\int_0^1 u_2\,dx_2\big)=\int_0^1\rho u_1u_2\,dx_2,~~\forall x_1\in[0,\infty).
\end{equation}

For $x_1\geq 2$, integrating \eqref{C318} over $[x_1,\infty)$ yields that
\begin{equation}\label{C319}
\bigg|\int_0^1 u_2(x_1,x_2)\,dx_2\bigg|=\bigg|\int_{x_1}^\infty\int_0^1\rho u_1u_2(s,x_2)\,dsdx_2\bigg|,
\end{equation}
since the boundary condition \eqref{C302} ensures that $u_2\rightarrow 0$ as $x_1\rightarrow\infty$. Substituting \eqref{C317} into the right hand side of \eqref{C319} implies that
\begin{equation}\label{C320}
\begin{split}
&\bigg|\int_{x_1}^\infty\int_0^1\rho u_1u_2 \,dsdx_2\bigg|\\
&=\bigg|\int_{x_1}^\infty\int_0^1\rho u_1\,\big(u_2-\int_0^1 u_2\,dx_2\big)\,dsdx_2\bigg|\\
&\leq\int_{x_1}^\infty\big(\int_0^1\rho|u_1|\,dx_2\big)\cdot\big(\int_0^1|\partial_2u_2|\,dx_2\big)\,ds\\
&\leq\big(\int_{x_1}^\infty\big(\int_0^1\rho|u_1|\,dx_2\big)^2s^{-1}ds\big)^{\frac{1}{2}}
\,\big(\int_\Omega|\nabla u|^2x_1dx_1dx_2\big)^{\frac{1}{2}}\\
&\leq C\,\left(\int_{x_1}^\infty\big(\int_0^1 P\,dx_2\big)\cdot\big(\int_0^1\rho u_1^2\,dx_2\big)\cdot s^{-1}ds\right)^{\frac{1}{2}},
\end{split}
\end{equation}
where the last line is due to \eqref{C303}. We consider $\eta_1(x_1)\triangleq\zeta_1(4x_1)$, then the term in the last line of \eqref{C320} is bounded by
\begin{equation}\label{CC308}
\begin{split}
&\int_{x_1}^\infty\big(\int_0^1 P\,dx_2\big)\cdot\big(\int_0^1\rho u_1^2\,dx_2\big)\cdot s^{-1}ds\\
&\leq\int_0^\infty\big(\int_0^1P\,dx_2\big)\cdot\big(\int_0^1\rho u_1^2\,dx_2\big)\cdot x_1^{-1/2}\zeta_1^2\,dx_1\\
&=\int_0^\infty\big(\int_0^1(P-P_\infty)\,dx_2\big)\,\big(\int_0^1\rho u_1^2\,dx_2\big)\, x_1^{-1/2}\zeta_1^2\,dx_1\\
&\qquad+P_\infty \int_\Omega\rho u_1^2\, x_1^{-1/2}\zeta_1^2\,dx_1dx_2,
\end{split}
\end{equation}

To handle \eqref{CC308}, we first introduce the test function
$$\xi(x_1)=\big(\int_{x_1}^\infty\int_0^1(P-P_\infty)\cdot s^{-1/2}\,dsdx_2\big)\cdot\zeta_1(x_1),$$
which is well defined due to \eqref{C525} and write \eqref{C521} in the variable $x_1$,
\begin{equation}\label{C529}
\begin{split}
&\partial_1\int_0^1(P-P_\infty)\,dx_2+\frac{1}{x_1}\, \partial_1(\int_0^1 x_1 \rho u_1^2 \,dx_2)-\frac{1}{x_1}\,\partial_1(\int_0^1x_1 \partial_1u_1\,dx_2)\\
&=-\frac{1}{x_1^2}\int_0^1u_1\,dx_2.
\end{split}
\end{equation}

Multiplying \eqref{C529} by $\xi\cdot\zeta_1$ and integrating with respect to $x_1$, we check each terms in details.
The first term on the left hand side of \eqref{C529} is given by
\begin{equation}\label{C530}
\begin{split}
&\int_0^\infty\partial_1\big(\int_0^1(P-P_\infty)\,dx_2\big)\,\xi\cdot\zeta_1\,dx_1\\
&=\int_0^\infty\left(\int_0^1(P-P_\infty)\,dx_2\right)^2x_1^{-1/2}\zeta_1^2\,dx_1-\int_0^\infty\big(\int_0^1(P-P_\infty)\,dx_2\big)\,\xi\cdot\zeta_1'\,dx_1\\
&\geq\int_0^\infty\left(\int_0^1(P-P_\infty)\,dx_2\right)^2x_1^{-1/2}\zeta_1^2\,dx_1-C,
\end{split}
\end{equation}
where the last line is due to \eqref{CC503}, \eqref{C520}, Corollary \ref{C51}, and the fact $\mathrm{supp}\,\zeta_1'\subset[1,2]\times\mathbb{T}$.

Then we consider the second term of \eqref{C529}.
\begin{equation}\label{C531}
\begin{split}
&\int_0^\infty\frac{1}{x_1}\,\partial_1\big(\int_0^1 x_1\,\rho u_1^2 \,dx_2\big)\,\xi\cdot\zeta_1\,dx_1\\
&=\int_0^\infty x_1^{-1}\big(\int_0^1\rho u_1^2\,dx_2\big)\,\xi\cdot\zeta_1\,dx_1-\int_0^\infty\big(\int_0^1\rho u_1^2\,dx_2\big)\,\xi\cdot\zeta_1'\,dx_1\\
&\quad+\int_0^\infty\big(\int_0^1\rho u_1^2\,dx_2\big)\cdot\big(\int_0^1(P-P_\infty)\,dx_2\big)\cdot x_1^{-1/2}\zeta_1^2\,dx_1.\\
\end{split}
\end{equation}
Let us apply \eqref{C507} and \eqref{C520} to check the first term on the second line of \eqref{C531}.
\begin{equation}\label{C532}
\begin{split}
&-\int_0^\infty x_1^{-1}\big(\int_0^1\rho u_1^2\,dx_2\big)\,\xi\cdot\zeta_1\,dx_1\\
&\leq C\big(\int_\Omega\rho u_1^2\cdot x_1^{-1/4}\zeta_1\,dx_1
dx_2\big)\,\big(\int_0^\infty\big(\int_0^1(P-P_\infty)\,dx_2\big)^2x_1^{-1/2}\,\zeta_1^2\,dx_1+1\big)^{\frac{1}{2}}\\
&\leq \frac{1}{8}\int_0^\infty\big(\int_0^1(P-P_\infty)\,dx_2\big)^2x_1^{-1/2}\,\zeta_1^2\,dx_1+C.
\end{split}
\end{equation}
While the fact $\mathrm{supp}\,\zeta_1'\subset[1,2]\times\mathbb{T}$ together with \eqref{C507} and \eqref{C520} also gives
\begin{equation}\label{C533}
\int_0^\infty\big(\int_0^1\rho u_1^2\,dx_2\big)\,\xi\cdot\zeta_1'\,dx_1\leq C.
\end{equation}
Substituting \eqref{C532} and \eqref{C533} into \eqref{C531}, we arrive at
\begin{equation}\label{C534}
\begin{split}
&\int_0^\infty\frac{1}{x_1}\,\partial_1\big(\int_0^1 x_1\,\rho u_1^2 \,dx_2\big)\,\xi\cdot\zeta_1\,dx_1\\
&\geq\int_0^\infty\big(\int_0^1\rho u_1^2\,dx_2\big)\cdot\big(\int_0^1(P-P_\infty)\,dx_2\big)\cdot x_1^{-1/2}\zeta_1^2\,dx_1\\
&\quad-\frac{1}{8}\int_0^\infty\big(\int_0^1(P-P_\infty)\,dx_2\big)^2x_1^{-1/2}\,\zeta_1^2\,dx_1-C.
\end{split}
\end{equation}

Next, we turn to the third term of \eqref{C529}.
\begin{equation}\label{C535}
\begin{split}
&\int_0^\infty\frac{1}{x_1}\,\partial_1\big(\int_0^1 x_1\,\partial_1u_1\,dx_2\big)\,\xi\cdot\zeta_1\,dx_1\\
&=\int_0^\infty x_1^{-1}\big(\int_0^1\partial_1u_1\,dx_2\big)\,\xi\cdot\zeta_1\,dx_1-\int_0^\infty\big(\int_0^1\partial_1u_1\,dx_2\big)\,\xi\cdot\zeta_1'\,dx_1\\
&\quad+\int_0^\infty\big(\int_0^1\partial_1u_1\,dx_2\big)\cdot\big(\int_0^1(P-P_\infty)\,dx_2\big)\cdot x_1^{-1/2}\zeta_1^2\,dx_1.
\end{split}
\end{equation}
In view of \eqref{C303} and \eqref{C520}, we argue that
\begin{equation}\label{C536}
\begin{split}
&\int_0^\infty x_1^{-1}\big(\int_0^1\partial_1u_1\,dx_2\big)\,\xi\cdot\zeta_1\,dx_1\\
&\leq C\big(\int_\Omega|\partial_1u_1|\cdot x_1^{-1/4}\zeta_1\,dx_1
dx_2\big)\,\big(\int_0^\infty\big(\int_0^1(P-P_\infty)\,dx_2\big)^2x_1^{-1/2}\,\zeta_1^2\,dx_1+1\big)^{\frac{1}{2}}\\
&\leq \frac{1}{8}\int_0^\infty\big(\int_0^1(P-P_\infty)\,dx_2\big)^2x_1^{-1/2}\,\zeta_1^2\,dx_1+C.
\end{split}
\end{equation}
While the fact $\mathrm{supp}\,\zeta_1'\subset[1,2]\times\mathbb{T}$ together with \eqref{C303} and \eqref{C520} also gives
\begin{equation}\label{C537}
-\int_0^\infty\big(\int_0^1\partial_1u_1\,dx_2\big)\,\xi\cdot\zeta_1'\,dx_1\leq C.\\
\end{equation}
Note that \eqref{C303} leads to
\begin{equation}\label{C538}
\begin{split}
&\int_0^\infty\big(\int_0^1\partial_1u_1\,dx_2\big)\cdot\big(\int_0^1(P-P_\infty)\,dx_2\big)\cdot x_1^{-1/2}\zeta_1^2\,dx_1\\
&\quad\leq\frac{1}{8}\int_0^\infty\big(\int_0^1(P-P_\infty)\,dx_2\big)^2x_1^{-1/2}\,\zeta_1^2\,dx_1+C.\\
\end{split}
\end{equation}
Substituting \eqref{C536}, \eqref{C537}, and \eqref{C538} into \eqref{C535} leads to 
\begin{equation}\label{C539}
\begin{split}
&\int_0^\infty\frac{1}{x_1}\,\partial_1\big(\int_0^1 x_1\,\partial_1u_1\,dx_2\big)\,\xi\cdot\zeta_1\,dx_1\\
&\leq\frac{1}{4}\int_0^\infty\big(\int_0^1(P-P_\infty)\,dx_2\big)^2x_1^{-1/2}\,\zeta_1^2\,dx_1+C.
\end{split}
\end{equation}

Once more, we utilize \eqref{C303} and \eqref{C520} to handle the right hand side of \eqref{C529} by
\begin{equation}\label{C540}
\begin{split}
&-\int_0^\infty\frac{1}{x_1^2}\big(\int_0^1 u_1\,dx_2\big)\,\xi\cdot\zeta_1\,dx_1\\
&\leq C\big(\int_\Omega|u_1|\cdot x_1^{-5/4}\zeta_1\,dx_1
dx_2\big)\,\big(\int_0^\infty\big(\int_0^1(P-P_\infty)\,dx_2\big)^2x_1^{-1/2}\,\zeta_1^2\,dx_1+1\big)^{\frac{1}{2}}\\
&\leq \frac{1}{8}\int_0^\infty\big(\int_0^1(P-P_\infty)\,dx_2\big)^2x_1^{-1/2}\,\zeta_1^2\,dx_1+C.
\end{split}
\end{equation}

Combining \eqref{C530}, \eqref{C534}, \eqref{C539}, and \eqref{C540}, we arrive at
\begin{equation}\label{C541}
\begin{split}
\int_0^\infty\big(\int_0^1\rho u_1^2\,dx_2\big)\cdot\big(\int_0^1(P-P_\infty)\,dx_2\big)\cdot x_1^{-1/2}\zeta_1^2\,dx_1\leq C.
\end{split}
\end{equation}
We mention that \eqref{C541} merely provide the upper bound of the left hand side term.

Now, substituting \eqref{C541} into \eqref{CC308}, we apply \eqref{C507} and Corollary \ref{C51} to deduce that for any $x_1\geq 2$,
\begin{equation*} 
\bigg|\int_{x_1}^\infty\int_0^1\rho u_1u_2 \,dsdx_2\bigg|\leq C,
\end{equation*}
which along with \eqref{C319} implies that for any $x_1\geq 2$ 
\begin{equation*}
\bigg|\int_0^1u_2(x_1,x_2)\,dx_2\bigg|\leq C,
\end{equation*}
and provides \eqref{C316}. Finally, we take advantage of \eqref{C303} and \eqref{C316} to argue that
\begin{equation*}
\begin{split}
&\int_{2}^\infty\int_0^1u_2^2\cdot x_1^{-2} \,dx_1dx_2\\
&\leq C\int_{2}^\infty\int_0^1\big(u_2-\int_0^1 u_2\,dx_2\big)^2\cdot x_1^{-2} \,dx_1dx_2+C\int_{2}^\infty\left(\int_0^1u_2\,dx_2\right)^2\, x_1^{-2} \,dx_1\\
&\leq C\int_\Omega|\partial_2 u_2|^2x_1dx_1dx_2+C\int_{2}^\infty\left(\int_0^1 u_2\,dx_2\right)^2\, x_1^{-2} \,dx_1\leq C,
\end{split}
\end{equation*}
which combined with \eqref{C303} yields \eqref{C334}. The proof is therefore completed.
\end{proof}

\subsection{Estimates away from the symmetric axis}\label{SS31}
\quad With Propositions \ref{P31}--\ref{P35} in hands, we can close the estimates on the domain away from the axis.  
We first utilize Proposition \ref{P32} to establish the weighted potential estimates of $\rho$ and $\rho|u|^2$.  
Let us introduce a truncation $\psi(x)\in C^\infty(\mathbb{R}^2)$ near $0$,
\begin{equation*}
\begin{cases}
\psi=1~~\mbox{if}\  |x|\leq 1,\\
\psi=0~~\mbox{if}\  |x|\geq 2,\\
0\leq\psi\leq 1~~\forall x\in\mathbb{R}^2.
\end{cases}
\end{equation*}
\begin{proposition}\label{P33}
Under the conditions of Lemma \ref{T1}, for any  $m\geq 6$, $0<\beta<1$, and $\tilde{x}\in\Omega$, there is a constant $C$ depending only on $m$ and $\beta$ such that
\begin{equation}\label{C307}
\begin{split}
&\int_\Omega\frac{\rho|u|^2+P}{|x-\tilde{x}|^\beta}\,\psi(x-\tilde{x})\,x_1^{-m}\cdot\zeta_1^2 \,dx_1dx_2\leq C\left(\int_\Omega\rho u_2^2\,x_1^{-m}\cdot\zeta_1^2\,dx_1dx_2+1\right).\\
\end{split}
\end{equation}
\end{proposition}
\begin{remark}
Note carefully that the truncation in the first term on the right hand side of \eqref{C307} only involves $\zeta_1$ instead of $\zeta_1'$. This little trick helps us close the estimates away from the axis directly.
\end{remark}
\begin{proof}
Adding \eqref{C301}$_{(i+1)}$ multiplied by $(x_i-\tilde{x}_i)/|x-\tilde{x}|^\beta\cdot\psi(x-\tilde{x})\cdot x_1^{-m}\,\zeta_1^2$ together, and integrating the result over $\Omega$ yields
\begin{equation}\label{C308}
\begin{split}
&\int_\Omega(\rho u_i u_j+P\,\delta_{ij})\cdot\partial_i\big(\frac{x_j-\tilde{x}_j}{|x-\tilde{x}|^\beta}\big)\cdot\psi(x-\tilde{x})\,x_1^{-m}\,\zeta_1^2\, dx_1dx_2\\
&\leq -\int_\Omega\big(\rho u_1u_i\,x_1\big)\cdot\frac{x_i-\tilde{x}_i}{|x-\tilde{x}|^\beta}\cdot\partial_1(\psi\, x_1^{-m-1}\,\zeta_1^2)\,dx_1dx_2\\
&\quad-\int_\Omega\big(\rho u_2u_i\big)\cdot\frac{x_i-\tilde{x}_i}{|x-\tilde{x}|^\beta}\cdot\partial_2\psi\cdot x_1^{-m}\,\zeta_1^2\,dx_1dx_2\\
&\quad+C\int_\Omega\big(|g|+P+|u_1|+\frac{|\nabla u|}{|x-\tilde{x}|^\beta}\big)\cdot x_1^{-m} \cdot(|\psi\cdot\zeta_1|+|\nabla(\psi\cdot\zeta_1)|)\,dx_1dx_2.
\end{split}
\end{equation}
By means of \eqref{C303} and \eqref{C305}, we argue that
\begin{equation*}
\begin{split}
&\int_\Omega\big(|g|+P+|u_1|+\frac{|\nabla u|}{|x-\tilde{x}|^\beta}\big)\cdot x_1^{-m} \cdot(|\psi\cdot\zeta_1|+|\nabla(\psi\cdot\zeta_1)|)dx_1dx_2\\
&\leq C\,\left(\int_\Omega\big(\frac{|u_1|^2}{x_1}+|\nabla u_1|^2x_1\big)dx_1dx_2\right)^{\frac{1}{2}}+C\int_A(\rho|u_1|^2+P)dx_1dx_2+C\\
&\leq C.
\end{split}
\end{equation*}
Moreover, the first two terms on the right hand side of \eqref{C308} are handled in the same manner. For example, with the help of \eqref{C305}, we have
\begin{equation*}
\begin{split}
&-\int_\Omega(\rho u_1u_i\,x_1)\cdot\frac{x_i-\tilde{x}_i}{|x-\tilde{x}|^\beta}\cdot\partial_1(\psi\, x_1^{-m-1}\,\zeta_1^2)\,dx_1dx_2\\
&\leq C\int_\Omega\rho|u_2|^2x_1^{-m}\cdot\zeta_1^2\,dx_1dx_2+C\int_\Omega\rho|u_1|^2x_1^{-m}\cdot|\zeta_1'|^2\,dx_1dx_2\\
&\leq C\int_\Omega\rho|u_2|^2\,x_1^{-m}\cdot\zeta_1^2\,dx_1dx_2+C.
\end{split}
\end{equation*}
Then, we turn to the left hand side of \eqref{C308}. Observe that
\begin{equation}\label{C348}
\begin{split}
\rho u_iu_j\cdot\partial_i\left(\frac{x_j}{|x|^\beta}\right)
=&\rho u_iu_j\cdot\left(\frac{\delta_{ij}}{|x|^\beta}-\beta\cdot\frac{x_ix_j}{|x|^{\beta+2}}\right)\geq(1-\beta)\cdot\frac{\rho|u|^2}{|x|^\beta},\\
\partial_i\left(\frac{x_i}{|x|^\beta}\right)
&=\frac{2}{|x|^\beta}-\beta\cdot\sum_i\frac{x_i^2}{|x|^{\beta+2}}=\frac{2-\beta}{|x|^{\beta}}.
\end{split}
\end{equation}

Substituting all these results into \eqref{C308}, we arrive at
\begin{equation*}
\begin{split}
&\int_\Omega\frac{\rho|u|^2+P}{|x-\tilde{x}|^\beta}\,\psi(x-\tilde{x})\,x_1^{-m}\cdot\zeta_1^2 \,dx_1dx_2 \leq C\left(\int_\Omega\rho u_2^2\cdot x_1^{-m}\cdot\zeta_1^2\,dx_1dx_2+1\right),\\
\end{split}
\end{equation*}
which is \eqref{C307} and finishes the proof.
\end{proof}

Now, let us combine Propositions \ref{P31}--\ref{P33} and claim that the density $\rho$ gains extra integrability when we consider the domain away from the axis. 
\begin{proposition}\label{P34}
Under the conditions of Lemma \ref{T1}, for $m\geq 6$ and $\theta=11/30$, there is a constant $C$ such that
\begin{equation}\label{C311}
\int_\Omega\big(\rho^{1+\theta}+\varepsilon\rho^{\alpha+\theta}\big)\,x_1^{-m-1} \zeta_1^2\,dx_1dx_2\leq C.
\end{equation}
\end{proposition}
\begin{proof}
The complete proof is rather long, so we divide it into several steps.

\textit{Step 1. Substitution of the test functions.}

To make notations clear, we first set the truncation $\eta(x_1)\triangleq 1-\zeta_0(2x_1)$ and introduce the notation
$$\widetilde{h}\triangleq h\cdot x_1^{-1}\eta.$$
We define the test functions
\begin{equation*}
\begin{split}
f_1&\triangleq\partial_1\Delta^{-1}\big(\widetilde{\rho^\theta}-\int_0^1\widetilde{\rho^\theta}dx_2\big)-\int_{x_1}^\infty(\int_0^1\widetilde{\rho^\theta}dx_2)\,ds,\\
f_2&\triangleq\partial_2\Delta^{-1}\big(\widetilde{\rho^\theta}\big),
\end{split}
\end{equation*}
where $\Delta^{-1}$ is the inversion of $\Delta$ on  $\mathbb{R}\times\mathbb{T}$ introduced by Lemma \ref{L23}.
  
Let us make a brief discussion on them. Note that in view of \eqref{C205} and \eqref{C206}, $f$ satisfies the cancellation condition and provides an ``inversion" of $\mathrm{div}$, 
\begin{equation}\label{C312}
\begin{split}
\int_0^1\partial_1\Delta^{-1}\left(\widetilde{\rho^\theta}-\int_0^1\widetilde{\rho^\theta}dx_2\right)dx_2=\int_0^1 f_2\,dx_2=0,~~\mathrm{div}f=\widetilde{\rho^\theta}.
\end{split}
\end{equation}
In addition, for any $p\in[1,5/2]$, we have $\theta\cdot p=11/12<1$, thus the interpolation argument ensures that,
\begin{equation}\label{C313}
\|\widetilde{\rho^\theta}\|_{L^p(\Omega)}\leq C\|(\widetilde{\rho})^\theta\cdot \eta\,x_1^{-1+\theta}\|_{L^p(\Omega)}\leq C\|\rho\|_{\mathcal{L}^1(\Omega)}\leq C.
\end{equation}
Consequently, according to \eqref{C312}, \eqref{C313}, the elliptic estimates \eqref{C205}, and Sobolev embedding theorem, we argue that $\forall q\in(1,5/2]$,
\begin{equation}\label{C314}
\begin{split}
\|\nabla f\|_{L^q(\Omega)}\leq C\,\|\widetilde{\rho^\theta}\|_{L^q(\Omega)}\leq C,~~\|f\|_{L^\infty(\Omega)}\leq \|\nabla f\|_{L^{5/2}(\Omega)}\leq C.
\end{split}
\end{equation}
 
Now, we substitute the test functions into \eqref{C301}. Adding \eqref{C301}$_2$ multiplied by $f_1\cdot x_1^{-m}\,\zeta_1^2$ and \eqref{C302}$_3$ multiplied by $f_2\cdot x_1^{-m}\,\zeta_1^2$ together, we integrate the result over $\Omega$ and apply \eqref{C312} together with the fact $\eta\,\cdot\,\zeta_1=\zeta_1$ to deduce that
\begin{equation*}
\begin{split}
&\int_\Omega (\rho^{1+\theta}+\varepsilon\rho^{\alpha+\theta})\cdot x_1^{-m-1} \zeta_1^2\,dx_1dx_2=J_0+J_1+J_2+J_3,\\
\end{split}
\end{equation*}
where the remaining terms $J_i$ are given by
\begin{equation*}
\begin{split}
J_0&=-\int_\Omega g\cdot f\cdot x_1^{-m}\zeta_1^2\,dx_1dx_2,\\
J_1&=-\int_\Omega\partial_1(x_1\,\partial_1u_i)\,f_i\cdot x_1^{-m-1}\zeta_1^2\,dx_1dx_2-\int_\Omega\partial_{22}u_i\cdot f_i\cdot x_1^{-m}\zeta_1^2\,dx_1dx_2,\\
J_2&=-\int_\Omega P\, f_i\cdot\partial_i(x_1^{-m}\zeta_1^2)\,dx_1dx_2+\int_\Omega u_1\cdot f_1\cdot x_1^{-m-2}\zeta_1^2\,dx_1dx_2,\\
J_3&=\int_\Omega\partial_1(x_1\rho u_1u_i)\,f_i\cdot x_1^{-m-1}\zeta_1^2\,dx_1dx_2+\int_\Omega\partial_2(\rho u_2u_i)\,f_i\cdot x_1^{-m}\zeta_1^2\,dx_1dx_2,\\
\end{split}
\end{equation*}

With the help of \eqref{C305} and \eqref{C314}, we calculate that $J_1$ is controlled by
\begin{equation}\label{C355}
\begin{split}
J_1&=\int_\Omega x_1\,\partial_1u_i\cdot\partial_1\big(f_i\cdot x_1^{-m-1}\zeta_1^2\big)\,dx_1dx_2+\int_\Omega\partial_2u_i\cdot\partial_2\big(f_i\cdot x_1^{-m}\zeta_1^2\big)\,dx_1dx_2\\
&\leq C \left(\|\nabla f\|_{L^2(\Omega)}+\|f\|_{L^\infty(\Omega)}\right)\,\left(\int_\Omega|\nabla u|^2x_1dx_1dx_2\right)^{\frac{1}{2}}\leq C,\\
\end{split}
\end{equation}
while $J_0$ and $J_2$ are controlled by
\begin{equation}\label{C356}
\begin{split}
J_0&\leq\|f\|_{L^\infty(\Omega)}\|g\|_{\mathcal{L}^2(\Omega)}\leq C,\\
J_2&\leq C\|f\|_{L^\infty(\Omega)}\int_\Omega P\,x_1^{-m}\zeta_1\,dx_1dx_2
+C\|f\|_{L^\infty(\Omega)}\left(\int_\Omega u_1^2\cdot x_1^{-1}\,dx_1dx_2\right)^{\frac{1}{2}}\\
&\leq C.\\
\end{split}
\end{equation}

Then, let us focus on the principal term $J_3$.  
\begin{equation*}
\begin{split}
J_3=-\int_\Omega\rho u_iu_j\cdot\partial_if_j\cdot x_1^{-m}\zeta_1^2\,dx_1dx_2-\int_\Omega x_1\rho u_1u_i\cdot f_1\,\partial_1(x_1^{-m-1}\zeta_1^2)\,dx_1dx_2.
\end{split}
\end{equation*}
By means of \eqref{C305} and \eqref{C314}, the second part is directly bounded by
\begin{equation}\label{C340}
\begin{split}
&\int_\Omega x_1\rho u_1u_i\cdot f_1\,\partial_1(x_1^{-m-1}\zeta_1^2)\,dx_1dx_2\\
&\leq C\|f\|_{L^\infty(\Omega)}\left(\int_\Omega\rho u_1^2\cdot x_1^{-m}(\zeta_1^2+|\zeta'_1|^2)\,dx_1dx_2+\int_\Omega\rho u_2^2\cdot x_1^{-m}\zeta_1^2\,dx_1dx_2\right)\\
&\leq C\left(\int_\Omega\rho u_2^2\cdot x_1^{-m}\cdot\zeta_1^2\,dx_1dx_2+1\right).
\end{split}
\end{equation}
Next, we rearrange the first part as
\begin{equation}\label{C315}
\begin{split}
&\int_\Omega\rho u_iu_j\cdot\partial_if_j\cdot x_1^{-m}\zeta_1^2\,dx_1dx_2\\
&=\int_\Omega(\rho u\cdot x_1^{-m+1}\zeta_1^2)\cdot\nabla\partial_j\Delta^{-1}( \rho^\theta\eta\,x_1^{-1})\cdot(u_j\eta\,x_1^{-1})\,dx_1dx_2\\
&=\int_\Omega(\rho u\cdot x_1^{-m+1}\zeta_1^2)\cdot\nabla\partial_j\Delta^{-1}(\widetilde{\rho^\theta})\cdot\widetilde{u}_j\,dx_1dx_2.
\end{split}
\end{equation}
The key issue is to properly control \eqref{C315}.

\textit{Step 2. Application of the potential estimates \eqref{C307}.}

We establish some auxiliary assertions. Let us consider the Cauchy-Riemann equation over $ \mathbb{R}\times\mathbb{T}$,
\begin{equation*}
\nabla F+\nabla^\bot\omega=\rho u\cdot x_1^{-m+1}\zeta_1^2-\int_0^1\rho u\cdot x_1^{-m+1}\zeta_1^2\,dx_2.
\end{equation*}
It holds that $\int F\,dx_2=\int\omega\,dx_2=0$ and
\begin{equation}\label{C331}
\begin{split}
&\nabla F+\nabla^\bot\big(\omega-\int_{x_1}^\infty\int_0^1\rho u_2\cdot s^{-m+1}\zeta_1^2 \,dsdx_2\big)=\rho u\cdot x_1^{-m+1}\zeta_1^2.
\end{split}
\end{equation}
due to \eqref{C317}, and we denote $\pi$ as the term in the large bracket of the above equation. Let us first apply the elliptic estimates \eqref{C205} to derive that, for any $\delta>0$,
\begin{equation}\label{C324}
\|\nabla F\|_{L^{1+\delta}(\mathbb{R}\times\mathbb{T})}+\|\nabla\omega\|_{L^{1+\delta}(\mathbb{R}\times\mathbb{T})}\leq C(\delta)\,\|\rho u\,x_1^{-m+1}\zeta_1^2\|_{L^{1+\delta}(\Omega)}.
\end{equation}
Since $F$ and $\omega$ satisfy the cancellation condition, Sobolev embedding theorem guarantees that, for $1/q=1/(1+\delta)-1/2=(1-\delta)/(2+2\delta)$, we also have
\begin{equation}\label{C325}
\|F\|_{L^{q}(\mathbb{R}\times\mathbb{T})}+\|\omega\|_{L^{q}(\mathbb{R}\times\mathbb{T})} \leq C \,\|\rho u\,x_1^{-m+1}\zeta_1^2\|_{L^{1+\delta}(\Omega)}.
\end{equation}

Next, let $\mathbf{E}_1(\cdot)$ be the periodic extension of $\mathbb{R}\times\mathbb{T}$ to $\mathbb{R}\times[-1,2]$.  
Thus, by virtue of \eqref{C315}, we obtain the Cauchy-Riemann equations on $\mathbb{R}\times[-1,2]$,
\begin{equation*}
\nabla\mathbf{E}_1(F)+\nabla^\bot\mathbf{E}_1\big(\omega-\int_{x_1}^\infty\int_0^1\rho u_2\cdot s^{-m+1}\zeta_1^2\,dsdx_2\big)=\mathbf{E}_1\big(\rho u\cdot x_1^{-m+1}\zeta_1^2\big),
\end{equation*}
meanwhile \eqref{C324} together with \eqref{C325} also implies that
\begin{equation}\label{C326}
\begin{split}
\|\mathbf{E}_1(F)\|_{L^q(\mathbb{R}\times[-1,2])}+\|\mathbf{E}_1(\omega)\|_{L^q(\mathbb{R}\times[-1,2])}\leq C\,\|\rho u\,x_1^{-m+1}\zeta_1^2\|_{L^{1+\delta}(\Omega)}.\\
\end{split}
\end{equation}

Moreover, we consider the Cauchy-Riemann equations on $\mathbb{R}^2$,
\begin{equation*}
\nabla\mathbf{E}_2(F)+\nabla^\bot\mathbf{E}_2(\omega)=\mathbf{E}_1\big(\rho u\cdot x_1^{-m+1}\zeta_1^2\big),
\end{equation*}
where the right hand side is 0 out of $\mathbb{R}\times[-1,2]$. The standard theory of the singular integral \cite[Chapter VI]{SW} ensures that for any $\tilde{x}\in\mathbb{R}^2$,
\begin{equation}\label{C330}
|\mathbf{E}_2(F)|(\tilde{x})+|\mathbf{E}_2(\omega)|(\tilde{x}) \leq C\int_{\mathbb{R}^2}\frac{\mathbf{E}_1\big(\rho u\cdot x_1^{-m+1}\zeta_1^2\big)}{|x-\tilde{x}|}dx.
\end{equation}
In addition, the Hardy-Littlewood-Sobolev inequality \eqref{C204} implies that
\begin{equation}\label{C327}
\|\mathbf{E}_2(F)\|_{L^q(\mathbb{R}^2)}+\|\mathbf{E}_2(\omega)\|_{L^q(\mathbb{R}^2)}\leq C\,\|\rho u\,x_1^{-m+1}\zeta_1^2\|_{L^{1+\delta}(\Omega)}.
\end{equation}

Then, we apply \eqref{C307} to improve the integrability of $\mathbf{E}_2(F)$ and $\mathbf{E}_2(\omega)$. Observe that, for $\beta\tilde{\theta}+(2-\epsilon)(1-\tilde{\theta})=1$, 
\begin{equation*}
\begin{split}
&\int_{\mathbb{R}^2}\frac{\mathbf{E}_1\big(\rho u\cdot x_1^{-m+1}\zeta_1^2\big)}{|x-\tilde{x}|}dx
\\
&\leq\left(\int_{\mathbb{R}^2}\frac{\mathbf{E}_1\big(\rho u\cdot x_1^{-m+1}\zeta_1^2\big)}{|x-\tilde{x}|^\beta}dx\right)^{\tilde{\theta}}\cdot\left(\int_{\mathbb{R}^2}\frac{\mathbf{E}_1\big(\rho u\cdot x_1^{-m+1}\zeta_1^2\big)}{|x-\tilde{x}|^{2-\epsilon}}dx\right)^{1-\tilde{\theta}}.
\end{split}
\end{equation*}
Note that by virtue of \eqref{C307}, we infer that
\begin{equation*}
\begin{split}
&\left\|\int_{\mathbb{R}^2}\frac{\mathbf{E}_1\big(\rho u\cdot x_1^{-m+1}\zeta_1^2\big)}{|x-\tilde{x}|^\beta}dx\right\|_{L^\infty(\mathbb{R}^2)}\\
&\leq C\,\sup_{\tilde{x}}\int_\Omega\frac{\rho|u|^2+P}{|x-\tilde{x}|^\beta}\,\psi(x-\tilde{x})\cdot
x_1^{-m+1}\zeta_1^2\,dx_1dx_2+\|\rho u\cdot x_1^{-m+1}\zeta_1^2\|_{L^1(\Omega)}\\
&\leq C\int_\Omega\rho u_2^2\,x_1^{-m+1}\cdot\zeta_1^2\,dx_1dx_2+C\int_A(\rho u_1^2+\varepsilon\rho^{\alpha})dx_1dx_2+C.
\end{split}
\end{equation*}
Hardy-Littlewood-Sobolev inequality \eqref{C204} also implies that, for $1/\tilde{q}=1/(1+\delta)-\epsilon/2$,
\begin{equation*}
\begin{split}
&\left\|\int_{\mathbb{R}^2}\frac{\mathbf{E}_1\big(\rho u\cdot x_1^{-m+1}\zeta_1^2\big)}{|x-\tilde{x}|^{2-\epsilon}}dx\right\|_{L^{\tilde{q}}(\mathbb{R}^2)}\leq C\,\|\rho u\,x_1^{-m+1}\zeta_1^2\|_{L^{1+\delta}(\Omega)}.
\end{split}
\end{equation*}
According to interpolation arguments and \eqref{C330}, we deduce that for $s=\tilde{q}/(1-\tilde{\theta})$,
\begin{equation}\label{C329}
\begin{split}
&\|\mathbf{E}_2(F)\|_{L^{s}(\mathbb{R}^2)}+\|\mathbf{E}_2(\omega)\|_{L^{s}(\mathbb{R}^2)}\\
&\leq C\big(\|\rho u_2^2\,x_1^{-m+1}\zeta_1^2\|_{L^{1}(\Omega)}+\,\|\rho u\,x_1^{-m+1}\zeta_1^2\|_{L^{1+\delta}(\Omega)}+ 1\big).\\
\end{split}
\end{equation}

At last, we need to investigate the error terms 
\begin{equation}\label{C332}
\begin{split}
&\mathbf{R}(F)\triangleq\mathbf{E}_2(F)-\mathbf{E}_1(F),\\
&\mathbf{R}(\omega)\triangleq\mathbf{E}_2(\omega)-\mathbf{E}_1\big(\omega-\int_{x_1}^\infty\int_0^1\rho u_2\cdot s^{-m+1}\zeta_1^2\,dsdx_2\big).
\end{split}
\end{equation}
Direct computation implies $\Delta \mathbf{R}(F)=\Delta \mathbf{R}(\omega)=0$ in $\mathbb{R}\times[-1,2]$. Thus, the interior elliptic estimates \eqref{C208} together with \eqref{C326} and \eqref{C327} guarantee that
\begin{equation}\label{C328}
\begin{split}
&\|\mathbf{R}(F)\|_{L^\infty(\Omega)}+\|\mathbf{R}(\omega)\|_{L^\infty(\Omega)}\\
&\leq C\sum_{i=1,2}\left(\|\mathbf{E}_i(F)\|_{L^q(\mathbb{R}^2)}+\|\mathbf{E}_i(\omega)\|_{L^q(\mathbb{R}^2)}\right)+C\,\|\rho u\,x_1^{-m+1}\zeta_1^2\|_{L^{1}(\Omega)}\\
&\leq C\,\left(\|\rho u\,x_1^{-m+1}\zeta_1^2\|_{L^{1+\delta}(\Omega)}+\|\rho u\,x_1^{-m+1}\zeta_1^2\|_{L^{1}(\Omega)}\right).
\end{split}
\end{equation}

\textit{Step 3. Estimates on the principal term.}

With all preparations done, we go back to the principal term \eqref{C315}. In view of \eqref{C331}, we declare that
\begin{equation}\label{C333}
\begin{split}
&\int_\Omega(\rho u\cdot x_1^{-m+1}\zeta_1^2)\cdot\nabla\partial_j\Delta^{-1}(\widetilde{\rho^\theta})\cdot\widetilde{u}_j\,dx_1dx_2\\
&=\int_\Omega \nabla F \cdot\nabla\partial_j\Delta^{-1}(\widetilde{\rho^\theta})\cdot\widetilde{u}_j\,dx_1dx_2+\int_\Omega \nabla^\bot\pi \cdot\nabla\partial_j\Delta^{-1}(\widetilde{\rho^\theta})\cdot\widetilde{u}_j\,dx_1dx_2,\\
\end{split}
\end{equation}
with $\pi$ determined by \eqref{C331}.
The first term on the right hand side of \eqref{C333} satisfies
\begin{equation}\label{C338}
\begin{split}
&\int_\Omega\nabla F\cdot\nabla\partial_j\Delta^{-1}(\widetilde{\rho^\theta})\cdot \widetilde{u}_j\,dx_1dx_2\\
&=-\int_\Omega F\, \widetilde{u}\cdot\nabla \widetilde{\rho^\theta} \,dx_1dx_2-\int_\Omega F\, \nabla\partial_j\Delta^{-1}(\widetilde{\rho^\theta})\cdot\nabla \widetilde{u}_j \,dx_1dx_2\triangleq K_1+K_2.\\ 
\end{split}
\end{equation}

By virtue of \eqref{C111}, we deduce that
$$\widetilde{u}\cdot\nabla\widetilde{\rho^\theta}=
-\theta\,\widetilde{\mathrm{div}u}\cdot\widetilde{\rho^\theta}-(1+\theta)\,x_1^{-1}\widetilde{u}_1\cdot\widetilde{\rho^\theta}+ \widetilde{u}\cdot\nabla\log\eta\cdot \widetilde{\rho^\theta},$$
which combined with \eqref{C332} leads to
\begin{equation*}
\begin{split}
&|K_1|+|K_2|\\
&\leq \int_\Omega|F|\cdot \left(\widetilde{\rho^\theta}+|\nabla\partial_j\Delta^{-1}(\widetilde{\rho^\theta})|\right)\,\big(|\widetilde{u}|+|\nabla\widetilde{u}|\big) \,dx_1dx_2\\
&= \int_\Omega|\mathbf{E}_1(F)|\cdot \left(\widetilde{\rho^\theta}+|\nabla\partial_j\Delta^{-1}(\widetilde{\rho^\theta})|\right)\,\big(|\widetilde{u}|+|\nabla\widetilde{u}|\big) \,dx_1dx_2\\
&\leq\int_\Omega\big(|\mathbf{E}_2(F)|+|\mathbf{R}(F)|\big)\cdot\left(\widetilde{\rho^\theta}+|\nabla\partial_j\Delta^{-1}(\widetilde{\rho^\theta})|\right)\,\big(|\widetilde{u}|+|\nabla\widetilde{u}|\big) \,dx_1dx_2.\\
\end{split}
\end{equation*}

Next, we apply \eqref{C303}, \eqref{C334}, \eqref{C313}, \eqref{C329},  and \eqref{C328} to deduce that
\begin{equation*}
\begin{split}
&\int_\Omega |\mathbf{E}_2(F)| \cdot \left(\widetilde{\rho^\theta}+|\nabla\partial_j\Delta^{-1}(\widetilde{\rho^\theta})|\right)\,\big(|\widetilde{u}|+|\nabla\widetilde{u}|\big) \,dx_1dx_2\\
&\leq C\,\|\mathbf{E}_2(F)\big\|_{L^{q}(\Omega)}\big\|\widetilde{u}\big\|_{H^{1}(\Omega)}+C\big\|\mathbf{E}_2(F)\big\|_{L^{s}(\Omega)}\,\big\|\nabla u\big\|_{\mathcal{L}^2(\Omega)}\\
&\leq C\left(\|\rho u_2^2\,x_1^{-m+1}\zeta_1^2\|_{L^{1}(\Omega)}+\,\|\rho u\,x_1^{-m+1}\zeta_1^2\|^2_{L^{1+\delta}(\Omega)}+1\right),
\end{split}
\end{equation*}
as well as the estimates
\begin{equation*}
\begin{split}
&\int_\Omega |\mathbf{R}(F)| \cdot \left(\widetilde{\rho^\theta}+|\nabla\partial_j\Delta^{-1}(\widetilde{\rho^\theta})|\right)\,\big(|\widetilde{u}|+|\nabla\widetilde{u}|\big) \,dx_1dx_2\\
&\leq C\,\|\mathbf{R}(F)\|_{L^\infty(\Omega)}\,\|\tilde{u}\|_{H^1(\Omega)}\\
&\leq C\left(\|\rho u\,x_1^{-m+1}\zeta_1^2\|^2_{L^{1+\delta}(\Omega)}+\|\rho u\,x_1^{-m+1}\zeta_1^2\|^2_{L^{1}(\Omega)}+1\right).
\end{split}
\end{equation*}

Moreover, observe that for $\kappa=2\delta(1+\theta)/(\theta+2\delta\theta)<1/(100m)$,
\begin{equation*}
\begin{split}
\|\rho\cdot\zeta_1^2\|_{L^{1+2\delta}(\Omega)}
&\leq C\left(\int_\Omega\rho\,x_1dx_1dx_2\right)^{1-\kappa}\left(\int_\Omega\rho^{1+\theta}x_1^{-m-1}\zeta_1^2dx_1dx_2\right)^{\frac{\kappa}{1+\theta}}\\
&\leq C\left(\int_\Omega\rho^{1+\theta}x_1^{-m-1}\zeta_1^2\,dx_1dx_2\right)^{\frac{\kappa}{1+\theta}},
\end{split}
\end{equation*}
which combined with \eqref{C334} yields
\begin{equation}\label{C336}
\begin{split}
\|\,\rho u\cdot x_1^{-m+1}\zeta_1^2\,\|^2_{L^{1+\delta}(\Omega)}
&\leq C\,\|\,\rho\cdot x_1^{-m+2}\zeta_1^2\,\|^2_{L^{1+2\delta}(\Omega)}\|\widetilde{u}\|^2_{H^1(\Omega)}\\
&\leq C\left(\int_\Omega\big(\rho^{1+\theta}+\varepsilon\rho^{\alpha+\theta}\big)x_1^{-m-1}\zeta_1^2dx_1dx_2\right)^{\frac{2\kappa}{1+\theta}+\frac{\alpha }{\alpha+\theta}}\\
&\leq \epsilon\int_\Omega\left(\rho^{1+\theta}+\varepsilon\rho^{\alpha+\theta}\right)x_1^{-m-1}\zeta_1^2\,dx_1dx_2+C.
\end{split}
\end{equation}
Similar process also gives
\begin{equation}\label{C337}
\begin{split}
&\|\rho |u|^2\,x_1^{-m+1}\zeta_1^2\|_{L^{1}(\Omega)}+\|\rho u\,x_1^{-m+1}\zeta_1^2\|^2_{L^{1}(\Omega)}\\
&\leq \epsilon\int_\Omega\left(\rho^{1+\theta}+\varepsilon\rho^{\alpha+\theta}\right)x_1^{-m-1}\zeta_1^2\,dx_1dx_2+C.
\end{split}
\end{equation}

Substituting \eqref{C336} and \eqref{C337} into \eqref{C338}, we arrive at
\begin{equation}\label{C339}
\begin{split}
&\int_\Omega\nabla F\cdot\nabla\partial_j\Delta^{-1}(\widetilde{\rho^\theta})\cdot \widetilde{u}_j\,dx_1dx_2\leq \epsilon\int_\Omega\left(\rho^{1+\theta}+\varepsilon\rho^{\alpha+\theta}\right)x_1^{-m-1}\zeta_1^2dx_1dx_2+C.
\end{split}
\end{equation}
Following the same approach given above, we declare  that \eqref{C339} is valid if we replace $\nabla F$ by $\nabla^\bot\pi$. Consequently, we apply \eqref{C340}, \eqref{C337}, and \eqref{C339} to conclude that
\begin{equation}\label{C341}
\begin{split}
&J_3\leq \epsilon\int_\Omega\big(\rho^{1+\theta}+\varepsilon\rho^{\alpha+\theta}\big)x_1^{-m-1}\zeta_1^2dx_1dx_2+C.
\end{split}
\end{equation}

Setting $\epsilon<1/4$, \eqref{C341} together with the estimates of $J_0$, $J_1$, and $J_2$ leads to
\begin{equation*}
\int_\Omega\big(\rho^{1+\theta}+\varepsilon\rho^{\alpha+\theta}\big)\,x_1^{-m-1} \zeta_1^2\,dx_1dx_2\leq C,
\end{equation*}
which is \eqref{C311}. The proof is therefore completed.
\end{proof}

We end this section by collecting all estimates obtained above in the next corollary. 
\begin{corollary} 
Under the conditions of Lemma \ref{T1}, there is a constant $C$ depending on $\Phi$ such that
\begin{equation}\label{C342}
\begin{split}
&\int_\Omega\big(\rho u_1^2+P\big)\,\big(x_1^{\varepsilon_0}\cdot\zeta_0+x_1^{-9/8}\cdot\zeta_1\big)\,dx_1dx_2\leq C,\\
\int_\Omega\rho&|u|^2\Phi(x_1)\cdot\zeta_1\,dx_1dx_2\leq C,~~\|u\cdot x_1^{-1}\|_{H^1([1/2,\infty)\times\mathbb{T})}\leq C.
\end{split}
\end{equation}
\end{corollary}
\begin{proof}
Combining \eqref{C303}, \eqref{C305}, \eqref{C334}, \eqref{C311}, and \eqref{C337} immediately yields \eqref{C342}.
\end{proof}

\subsection{Estimates near the symmetric axis}\label{SS32}
\quad The methods in this subsection are also independent of the behaviour of $\rho$ at the far field.
We first establish the near side version of Proposition \ref{P35}, where another weight is required to control $\|u_2\|_{L^2([0,4]\times\mathbb{T})}$.
\begin{proposition} 
Under the conditions of Lemma \ref{T1}, there is a constant $C$, such that for any $x_1\in[0,4]$,
\begin{equation}\label{C343}
\begin{split}
&x_1\cdot\bigg|\int_0^4 u_2(x_1,x_2)\,dx_2\bigg|\leq C\int_\Omega\rho|u|^2\Phi \,dx_1dx_2+C.\\
\end{split}
\end{equation}
In particular, we also have
\begin{equation}\label{C344}
\int_0^4\int_0^1(|u|^2+|\nabla u|^2)\,x_1^2\,dx_1dx_2\leq C\left(\int_\Omega\rho|u|^2\Phi \,dx_1dx_2\right)^2+C.
\end{equation}
\end{proposition}
\begin{proof}
Recalling that \eqref{C318} together with \eqref{C342} ensures that, for $x_1\in [0,4]$,
\begin{equation*}
\begin{split}
&x_1\cdot\bigg|\int_0^1 u_2(x_1,x_2)\,dx_2\bigg|\\
&\leq x_1\cdot\bigg|\int_0^1 u_2(4,x_2)\,dx_2\bigg|+x_1\cdot\bigg|\int_{x_1}^1\int_0^4\rho u_1u_2\,dx_1dx_2\bigg|\\
&\leq C+\int_0^1\int_0^1\rho|u|^2x_1dx_1dx_2,
\end{split}
\end{equation*}
which provides \eqref{C343}. To illustrate \eqref{C344}, let us apply \eqref{C303} to deduce that
\begin{equation*}
\begin{split}
&\int_{0}^4\int_0^1 u_2^2\cdot x_1^{2} \,dx_1dx_2\\
&\leq C\int_{0}^4\int_0^1\big(u_2-\int_0^1 u_2\,dx_2\big)^2\cdot x_1^{2} \,dx_1dx_2+C\int_{0}^4\left(x_1\cdot\int_0^1 u_2\,dx_2\right)^2 \,dx_1\\
&\leq C\int_\Omega|\partial_2 u_2|^2x_1dx_1dx_2+C\int_{0}^4\left(x_1\cdot\int_0^1 u_2\,dx_2\right)^2\,dx_1\\
&\leq C\left(\int_\Omega\rho|u|^2\Phi\,dx_1dx_2\right)^2+C,
\end{split}
\end{equation*}
which combined with \eqref{C303} yields \eqref{C344}.
We therefore complete the proof. 
\end{proof}

Recalling that $\hat\Omega=\mathbb{R}^2\times\mathbb{T}$, we suppose $\hat x=(\hat{x}',\hat{x}_3)\in\mathbb{R}^2\times\mathbb{T}$ and $x\in\mathbb{R}^3$. Let us introduce two truncation functions $\hat\zeta_0(\hat x),\,\hat\psi(\hat x)\in C^\infty(\hat\Omega)$ by
\begin{equation}\label{C370}
\begin{cases}
\hat\zeta_0=1~~|\hat{x}'|\leq 12,\\
\hat\zeta_0=0~~|\hat{x}'|\geq 15,\\
0\leq\hat\zeta_0\leq 1~~\forall\hat x\in\hat\Omega,
\end{cases}
\end{equation}
\begin{equation}\label{C378}
\begin{cases}
\hat\psi=1~~|x|\leq 1/4,\\
\hat\psi=0~~|x|\geq 1/2,\\
0\leq\hat\psi\leq 1~~\forall x\in\mathbb{R}^3.
\end{cases}
\end{equation}
We mention that $\hat\zeta_0$ and $\hat\psi$ are both periodic in $\hat x_3$ variable which provide  truncations near the symmetric axis and the origin in $\hat\Omega$ respectively.
\begin{proposition}
Under the conditions of Lemma \ref{T1}, for $0<\beta<1$ and $m\geq 6$, there is a constant $C$ depending on $\beta$ such that, $\forall\tilde{x}\in\Omega$
\begin{equation}\label{C345}
\begin{split}
&\int_\Omega\frac{\rho|u|^2+P}{|x-\tilde{x}|^\beta}\,\hat\psi(x-\tilde{x})\,x_1^{m}\cdot\zeta_0\,dx_1dx_2 \leq C\int_\Omega\rho |u|^2\Phi\,dx_1dx_2+C.\\
\end{split}
\end{equation}

Moreover, if we go back to $\hat\Omega$ and define $(\hat\rho,\hat{u})$ by \eqref{C346},
 the potential estimates hold that, $\forall\tilde{x}\in\hat\Omega$,
\begin{equation}\label{C347}
\begin{split}
&\int_{\hat\Omega}\frac{\hat\rho(|\hat u|^2+1)}{|\hat x-\tilde{x}|^\beta}\,\hat\psi(\hat x-\tilde{x})\cdot\hat\zeta_0\, d\hat x \leq C\int_\Omega\rho |u|^2\Phi\,dx_1dx_2+C.\\
\end{split}
\end{equation}
\end{proposition}
\begin{proof}
Let us repeat the proof of Proposition \ref{P33} and change the weight $x^{-m}\zeta_1^2$ by $x^{m}\zeta_0$. The process similar with \eqref{C308} combined with \eqref{C342} leads to
\begin{equation*} 
\begin{split}
&(1-\beta)\cdot\int_\Omega\frac{\rho|u|^2+P}{|x-\tilde{x}|^\beta}\,\psi(x-\tilde{x})\,x_1^{m}\cdot\zeta_0\,dx_1dx_2 \leq C\int_\Omega\rho |u|^2\Phi\,dx_1dx_2+C,\\
\end{split}
\end{equation*}
which gives \eqref{C345} and finish the first part.

Next, according to Remark \ref{R12}, $(\hat\rho,\hat u)$ is a weak solution to \eqref{C350} in $\hat\Omega$. Thus, following the process of \cite[Lemma 2.2]{JZ}, we multiply \eqref{C350}$_2$ by 
$(\hat{x}-\tilde{x})/|\hat x-\tilde{x}|^\beta\cdot\hat\psi(\hat x-\tilde{x})\cdot\hat\zeta_0(\hat x)$ and deduce that
\begin{equation}\label{C349}
\begin{split}
&\int_{\hat\Omega}\big(\hat\rho\,\hat{u}_i\hat{u}_j
+\hat P\delta_{ij}\big)\cdot\partial_i\big(\frac{\hat x_j-\tilde{x}_j}{|\hat x-\tilde{x}|^\beta}\big)\cdot\hat\psi(\hat x-\tilde{x})\cdot \hat\zeta_0\,d\hat x\\
&\leq\int_{\hat\Omega}(|\hat{g}|+\hat\rho\,|\hat u|^2+\hat P+\frac{|\nabla\hat u|}{|\hat x-\tilde{x}|^\beta}\big)\,(|\hat\psi\cdot\hat\zeta_0|+|\nabla(\hat\psi\cdot\hat\zeta_o)|)\,d\hat x\\
&\leq C\left(\|\hat\rho\,|\hat u|^2\|_{L^1(Q\times\mathbb{T})}+\|\hat P\|_{L^1(Q\times\mathbb{T})}+\|\nabla\hat u\|_{L^2(Q\times\mathbb{T})}+1\right),
\end{split}
\end{equation}
where $Q=[-15,15]\times[-15,15]\subset\mathbb{R}^2$. By making use of \eqref{C348}, we transform the right hand side of \eqref{C349} into $\Omega$ and arrive at
\begin{equation*}
\begin{split}
&\int_{\hat\Omega}\frac{\hat\rho(|\hat u|^2+1)}{|\hat x-\tilde{x}|^\beta}\,\hat\psi(\hat x-\tilde{x})\cdot\hat\zeta_0\, d\hat x\\
&\leq C\int_\Omega\rho|u|^2\Phi\,dx_1dx_2+C\int_0^{30}\int_0^1\big(1+P\,x_1+|\nabla u|^2x_1+|u_1|^2x_1^{-1}\big)\,dx_1dx_2\\
&\leq C\int_\Omega\rho|u|^2\Phi\,dx_1dx_2+C,
\end{split}
\end{equation*}
where the last line is due to \eqref{C303} and \eqref{C342}. The proof is therefore completed.
\end{proof}

Finally, we are in a position to close the estimates of $\rho|u|^2$. In view of \eqref{C342}, it is sufficient to control the part near the axis.
\begin{proposition}\label{P38}
Under the conditions of Lemma \ref{T1}, there is a constant $C$ depending on $\Phi(x_1)$, such that
\begin{equation}\label{C351}
\int_\Omega\rho|u|^2\Phi\cdot\zeta_0\,dx_1dx_2\leq C.
\end{equation}
\end{proposition}
\begin{proof}
\textit{Step 1. The extra integrability of $\rho$ near the axis.}

We mainly follow the method given by Proposition \ref{P34}. Taking $\theta=11/30$ and $m\geq 6$, the task in this step is to derive that, 
\begin{equation}\label{C352}
\int_\Omega(\rho^{1+\theta}+\varepsilon\rho^{\alpha+\theta})\,x_1^{m+1}\zeta_0\,dx_1d_2\leq C\left(\int_\Omega\rho|u|^2\Phi\,dx_1dx_2\right)^4+C.
\end{equation} 

We define the truncation function $\eta(x_1)\triangleq\zeta_0(x_1/4)$ and introduce the test function
$$\varphi=\nabla\Delta^{-1}\big(\rho^\theta\eta\,x_1^2-\int_\Omega\rho^\theta\eta\,x_1^2\,dx_1dx_2\big),$$
where $\Delta^{-1}$ is the inversion of $\Delta$ on the torus $[0,8]\times\mathbb{T}$ introducing by Lemma \ref{L24}. Note that $\varphi$ is well defined, since the integral average has been module out.

We mention that $\varphi$ has zero average over $[0,8]\times\mathbb{T}$,
thus the elliptic estimates \eqref{C207} together with Sobolev embedding theorem implies that
\begin{equation}\label{C353}
\|\varphi\|_{L^{\infty}([0,8]\times\mathbb{T})}+\|\nabla\varphi\|_{L^{30/11}([0,8]\times\mathbb{T})}\leq C\,\|\rho^\theta\eta\,x_1\|_{L^{30/11}(\Omega)}\leq C.
\end{equation}

Adding \eqref{C301}$_2$ multiplied by $\varphi_1\cdot x_1^{m}\,\zeta_0$ and \eqref{C301}$_3$ multiplied by $\varphi_2\cdot x_1^{m}\,\zeta_0$ together, we integrate the result over $\Omega$ and apply the fact $\zeta_0\cdot\eta=\zeta_0$ to deduce that
\begin{equation}\label{C354}
\int_\Omega(\rho^{1+\theta}+\varepsilon\rho^{\alpha+\theta})\,x_1^{m+1}\zeta_0\,dx_1dx_2\leq L_0+L_1+L_2+L_3,
\end{equation}
where the remaining terms $L_i$ are given by
\begin{equation*}
\begin{split}
L_0&=\big(\int_\Omega P\cdot x_1^m\zeta_0\,dx_1dx_2\big)\cdot\big(\int_\Omega\rho^\theta
\eta\,x_1\,dx_1dx_2\big)-\int_\Omega g\cdot\varphi\cdot x_1^{m+1}\zeta_0\,dx_1dx_2,\\
L_1&=-\int_\Omega\partial_1(x_1\partial_1 u_i)\,\varphi_i\cdot x_1^{m-1}\zeta_0\,dx_1dx_2-\int_\Omega\partial_{22}u_i\cdot\varphi_i\cdot x_1^m\zeta_0\,dx_1dx_2,\\
L_2&=-\int_\Omega P\,\varphi\cdot\partial_i(x_1^{m}\zeta_0)\,dx_1dx_2+\int_\Omega u_1\cdot \varphi_1\cdot x_1^{m-2}\zeta_0\,dx_1dx_2,\\
L_3&=\int_\Omega\partial_1(x_1\rho u_1u_i)\,\varphi_i\cdot x_1^{m-1}\zeta_0\,dx_1dx_2+\int_\Omega\partial_2(\rho u_2u_i)\,\varphi_i\cdot x_1^{m}\zeta_0\,dx_1dx_2.\\
\end{split}
\end{equation*}
Similar with \eqref{C355} and \eqref{C356}, we apply \eqref{C303},  \eqref{C342}, and \eqref{C353} to derive that
\begin{equation}\label{C361}
\begin{split}
L_0+L_1+L_2\leq C.
\end{split}
\end{equation}

Then, we focus on the principal term $L_3$. Observe that
\begin{equation*}
\begin{split}
L_3=-\int_\Omega\rho u_iu_j\cdot\partial_i\varphi_j\cdot x_1^m\zeta_0\,dx_1dx_2-\int_\Omega\rho u_1u_i\cdot\varphi_i\cdot\partial_1(x_1^{m-1}\zeta_0)\,dx_1dx_2.
\end{split}
\end{equation*}
By means of \eqref{C342} and \eqref{C353}, we deduce that the second term is handled by
\begin{equation}\label{C357}
\int_\Omega\rho u_1u_i\cdot\varphi_i\cdot\partial_1(x_1^{m-1}\zeta_0)\,dx_1dx_2\leq C\int_\Omega\rho|u|^2\Phi\,dx_1dx_2+C.
\end{equation}
Moreover, let us rearrange the first term as
\begin{equation}\label{C358}
\begin{split}
&\int_\Omega\rho u_iu_j\cdot\partial_i\varphi_j\cdot x_1^m\,\zeta_0\,dx_1dx_2\\
&=\int_\Omega(\rho u\,x_1^{m-2}\zeta_0)\cdot\nabla\partial_j\Delta^{-1}\big(\widetilde{\rho^\theta}-\int_\Omega\widetilde{\rho^\theta}\,dx_1dx_2\big)\cdot \widetilde{u}_j\,dx_1dx_2,\\
\end{split}
\end{equation}
where we have applied the notation
$\widetilde{h}\triangleq(h\cdot x_1^2\,\zeta_0).$

Now, let us define the pair
\begin{equation*}
\begin{split}
&F=\Delta^{-1}\mathrm{div}(\rho u\,x_1^{m-2}\zeta_0),~~\omega=\Delta^{-1}\mathrm{rot}(\rho u\,x_1^{m-2}\zeta_0),\\
&\Rightarrow\int_{[0,8]\times\mathbb{T}}F\,dx_1dx_2=\int_{[0,8]\times\mathbb{T}}\omega\,dx_1dx_2=0.
\end{split}
\end{equation*}
Note that $(F,\omega)$ solves the Cauchy-Riemann equations on periodic domain $[0,8]\times\mathbb{T}$,
\begin{equation*}
\begin{split}
&\nabla F+\nabla^\bot\omega=\rho u\,x_1^{m-2}\zeta_0-\int_{[0,8]\times\mathbb{T}}\rho u\,x_1^{m-2}\zeta_0\,dx_1dx_2.\\
\end{split}
\end{equation*}
The elliptic estimates \eqref{C207} and Sobolev embedding theorem ensure that
\begin{equation}\label{C359}
\begin{split}
&\|F\|_{L^q([0,8]\times\mathbb{T})}+\|\omega\|_{L^q([0,8]\times\mathbb{T})}\leq C\|\rho u\,x_1^{m-2}\zeta_0\|_{L^{1+\delta}(\Omega)},
\end{split}
\end{equation}
with $1/q=1/(1+\delta)-1/2$.

Following the same calculation given by \eqref{C333} and \eqref{C338}, we apply $\nabla^\bot\circ\nabla=0$ and the equality
$$\widetilde{u}\cdot\nabla\widetilde{\rho^\theta}=
-\theta\,\widetilde{\mathrm{div}u}\cdot\widetilde{\rho^\theta}+(2-\theta)\,x_1^{-1}\widetilde{u}_1\cdot\widetilde{\rho^\theta}+ \widetilde{u}\cdot\nabla\log\eta\cdot \widetilde{\rho^\theta},$$
which is due to \eqref{C111}, to transform \eqref{C358} into
\begin{equation}\label{C367}
\begin{split}
&\int_\Omega\rho u_iu_j\cdot\partial_i\varphi_j\cdot x_1^m\,\zeta_0\,dx_1dx_2\\
&=\int_\Omega(\nabla F+\nabla^\bot\omega)\cdot\nabla\partial_j\Delta^{-1}\big(\widetilde{\rho^\theta}-\int_\Omega\widetilde{\rho^\theta}\,dx_1dx_2\big)\cdot \widetilde{u}_j\,dx_1dx_2+\mathcal{R}\\
&\leq C\big(\|F\|_{L^4([0,8]\times\mathbb{T})}+\|\omega\|_{L^4([0,8]\times\mathbb{T})}\big)\cdot\|\widetilde{\rho^\theta}\|_{L^2(\Omega)}\cdot\|\widetilde{u}\|_{H^1(\Omega)}+\mathcal{R},\\
\end{split}
\end{equation}
where the remaining term $\mathcal{R}$ satisfies 
\begin{equation*}
\begin{split}
\mathcal{R}&\leq C\,\|\rho u\,x_1^{m-2}\zeta_0\|_{L^1([0,8]\times\mathbb{T})} \cdot\|\widetilde{\rho^\theta}\|_{L^2(\Omega)}\cdot\|\widetilde{u}\|_{H^1(\Omega)}\\
&\leq C\left(\int_\Omega\rho|u|^2\Phi\,dx_1dx_2+1\right)^2,
\end{split}
\end{equation*}
which is due to \eqref{C344}.

To finish the argument, we follow \textit{Step 2.} of the proof for Proposition \ref{P34} to handle $\|F\|_{L^4([0,8]\times\mathbb{T})}$ and $\|\omega\|_{L^4([0,8]\times\mathbb{T})}$.
Let $\mathbf{E}_3(\cdot)$ be the periodic extension from the torus $[0,8] \times\mathbb{T}$ to $[-8,16]\times[-1,2]$. In particular, by virtue of \eqref{C359}, the next equation is valid in $[-8,16]\times[-1,2]$, 
$$\nabla\mathbf{E}_3(F)+\nabla^\bot\mathbf{E}_3(\omega)=\mathbf{E}_3\left(\rho u\,x_1^{m-2}\zeta_0-\int_{[0,8]\times\mathbb{T}}\rho u\,x_1^{m-2}\zeta_0\,dx_1dx_2\right),$$
and the estimates given below is true as well,
\begin{equation}\label{C362}
\|\mathbf{E}_3(F)\|_{L^q([0,8]\times\mathbb{T})}+\|\mathbf{E}_3(\omega)\|_{L^q([0,8]\times\mathbb{T})}\leq C\|\rho u\,x_1^{m-2}\zeta_0\|_{L^{1+\delta}(\Omega)}.
\end{equation}

Moreover, let us consider the Cauchy-Riemann equations on $\mathbb{R}^2$,
$$\nabla\mathbf{E}_2(F)+\nabla^\bot\mathbf{E}_2(\omega)=\mathbf{E}_3\left(\rho u\,x_1^{m-2}\zeta_0-\int_{[0,8]\times\mathbb{T}}\rho u\,x_1^{m-2}\zeta_0\,dx_1dx_2\right),$$
where the right hand side is zero out of $[-8,16]\times[-1,2]$. The standard theory of the singular integral \cite[Chapter VI]{SW} ensures that $\forall\tilde{x}\in\mathbb{R}^2$,
\begin{equation}\label{C363}
\begin{split}
&|\mathbf{E}_2(F)|(\tilde{x})+|\mathbf{E}_2(\omega)|(\tilde{x})\\
&\leq C\int_{\mathbb{R}^2}\frac{\mathbf{E}_3\big(\rho u\cdot x_1^{m-2}\zeta_0\big)}{|x-\tilde{x}|}dx+C\int_\Omega\rho|u|^2\Phi\,dx_1dx_2.
\end{split}
\end{equation}

Setting $\beta\tilde{\theta}+(2-\epsilon)(1-\tilde{\theta})=1$, we have
\begin{equation*}
\begin{split}
&\int_{\mathbb{R}^2}\frac{\mathbf{E}_3\big(\rho u\cdot x_1^{m-2}\zeta_0\big)}{|x-\tilde{x}|}dx\\
&\leq\left(\int_{\mathbb{R}^2}\frac{\mathbf{E}_3\big(\rho u\cdot x_1^{m-2}\zeta_0\big)}{|x-\tilde{x}|^\beta}dx\right)^{\tilde{\theta}}\cdot\left(\int_{\mathbb{R}^2}\frac{\mathbf{E}_3\big(\rho u\cdot x_1^{m-2}\zeta_0\big)}{|x-\tilde{x}|^{2-\epsilon}}dx\right)^{1-\tilde{\theta}}.
\end{split}
\end{equation*}
Note that by virtue of \eqref{C345}, we infer that
\begin{equation*}
\begin{split}
&\left\|\int_{\mathbb{R}^2}\frac{\mathbf{E}_3\big(\rho u\cdot x_1^{m-2}\zeta_0\big)}{|x-\tilde{x}|^\beta}dx\right\|_{L^\infty(\mathbb{R}^2)}\\
&\leq C\,\sup_{\tilde{x}}\int_\Omega\frac{\rho|u|^2+P}{|x-\tilde{x}|^\beta}\,\psi(x-\tilde{x})\cdot
x_1^{m-2}\zeta_0 \,dx_1dx_2+\|\rho u\cdot x_1^{m-2}\zeta_0\|_{L^1(\Omega)}\\
&\leq C\int_\Omega\rho|u|^2\Phi\,dx_1dx_2+C.
\end{split}
\end{equation*}
In addition, Hardy-Littlewood-Sobolev inequality \eqref{C204} also implies that, for $1/\tilde{q}=1/(1+\delta)-\epsilon/2$,
\begin{equation*}
\begin{split}
&\left\|\int_{\mathbb{R}^2}\frac{\mathbf{E}_3\big(\rho u\cdot x_1^{m-2}\zeta_0\big)}{|x-\tilde{x}|^{2-\epsilon}}dx\right\|_{L^{\tilde{q}}(\mathbb{R}^2)}\leq C\|\rho u\,x_1^{m-2}\zeta_0\|_{L^{1+\delta}(\Omega)}.
\end{split}
\end{equation*}
According to interpolation arguments and \eqref{C363}, we deduce that for $s=\tilde{q}/(1-\tilde{\theta})$,
\begin{equation}\label{C364}
\begin{split}
&\|\mathbf{E}_2(F)\|_{L^{4}([-8,16]\times[-1,2])}+\|\mathbf{E}_2(\omega)\|_{L^{4}([-8,16]\times[-1,2])}\\
&\leq C\|\rho u\,x_1^{m-2}\zeta_0\|_{L^{1+\delta}(\Omega)}+C\int_\Omega\rho|u|^2\Phi\,dx_1dx_2+C.\\
\end{split}
\end{equation}

At last, we need to investigate the error terms 
\begin{equation*} 
\begin{split}
&\mathbf{R}(F)\triangleq\mathbf{E}_3(F)-\mathbf{E}_2(F),~~\mathbf{R}(\omega)\triangleq\mathbf{E}_3(\omega)-\mathbf{E}_2(\omega).
\end{split}
\end{equation*}
Direct computation implies $\Delta \mathbf{R}(F)=\Delta \mathbf{R}(\omega)=0$ in $[-8,16]\times[-1,2]$. Thus, the interior elliptic estimates \eqref{C208} together with \eqref{C362} and \eqref{C364} guarantee that
\begin{equation}\label{C365}
\begin{split}
&\|\mathbf{R}(F)\|_{L^\infty([0,8]\times\mathbb{T})}+\|\mathbf{R}([0,8]\times\mathbb{T})\|_{L^\infty([0,8]\times\mathbb{T})}\\
&\leq C\sum_{i=2,3}\big(\|\mathbf{E}_i(F)\|_{L^q([-8,16]\times[-1,2])}+\|\mathbf{E}_i(\omega)\|_{L^q([-8,16]\times[-1,2])} \big)\\
&\leq C\|\rho u\,x_1^{m-2}\zeta_0\|_{L^{1+\delta}(\Omega)}+C\int_\Omega\rho|u|^2\Phi\,dx_1dx_2+C.
\end{split}
\end{equation}

Observe that, similar interpolation arguments as  \eqref{C336} implies that, for $\varepsilon<1/1000$ and $\kappa=2\delta(1+\theta)/(\theta+2\delta\theta)<1/(100m)$,
\begin{equation*}
\begin{split}
&\|\rho\cdot x_1\,\zeta_0\|_{L^{1+2\delta}(\Omega)}\\
&\leq C\big(\int_\Omega\rho\cdot x_1^{\varepsilon}\,\zeta_0\,dx_1dx_2\big)^{1-\kappa}\big(\int_\Omega\rho^{1+\theta}x_1^{m+1}\zeta_0\,dx_1dx_2\big)^{\frac{\kappa}{1+\theta}}\\
&\leq C\big(\int_\Omega\rho^{1+\theta}x_1^{m+1}\zeta_0\,dx_1dx_2\big)^{\frac{\kappa}{1+\theta}},
\end{split}
\end{equation*}
where the last line is due to \eqref{C342}. Consequently, we argue that
\begin{equation*}
\begin{split}
&\|\rho u\,x_1^{m-2}\zeta_0\|_{L^{1+\delta}(\Omega)}\\
&\leq C\|\rho\cdot x_1\,\zeta_0\|_{L^{1+2\delta}(\Omega)}\|\widetilde{u}\|_{H^1(\Omega)}\\
&\leq C\big(\int_\Omega\rho^{1+\theta}x_1^{m+1}\zeta_0\,dx_1dx_2\big)^{\frac{\kappa}{1+\theta}}\cdot\big(\int_\Omega\rho|u|^2\Phi\,dx_1dx_2+1\big),
\end{split}
\end{equation*}
which combined with \eqref{C364} and \eqref{C365}  leads to
\begin{equation}\label{C366}
\begin{split}
&\|F\|_{L^4([0,8]\times\mathbb{T})}+\|\omega\|_{L^4([0,8]\times\mathbb{T})}\\
&\leq C\big(\int_\Omega\rho^{1+\theta}x_1^{m+1}\zeta_0\,dx_1dx_2+1\big)^{\frac{\kappa}{1+\theta}}\cdot\big(\int_\Omega\rho|u|^2\Phi\,dx_1dx_2+1\big).
\end{split}
\end{equation}

Cominging \eqref{C344}, \eqref{C354}--\eqref{C357}, \eqref{C367}, and \eqref{C366} yields
\begin{equation}\label{C368}
\begin{split}
&\int_\Omega(\rho^{1+\theta}+\varepsilon\rho^{\alpha+\theta})\,x_1^{m+1}\zeta_0\,dx_1d_2\\
&\leq \epsilon\int_\Omega(\rho^{1+\theta}+\varepsilon\rho^{\alpha+\theta})\,x_1^{m+1}\zeta_0\,dx_1d_2+ C\big(\int_\Omega\rho|u|^2\Phi\,dx_1dx_2\big)^4+C,
\end{split}
\end{equation}
since $\kappa<1/2$. Setting $\epsilon<1/4$ in \eqref{C368}, we obtain \eqref{C352} and finish the  first step.

\textit{Step 2. Controlling the singularity near the symmetric axis.}

In view of \eqref{C309}, we have $\Phi\cdot\zeta_0=x_1\cdot\zeta_0$, thus direct computations yield,
\begin{equation}\label{C379}
\begin{split}
&\int_\Omega\rho|u|^2\Phi\,\zeta_0\,dx_1dx_2\\
&=\int_\Omega\big(\rho^{1+\theta}x_1^{m}\zeta_0\big)^{\alpha_1}\,\big(\rho^{\beta}|u|^{2\beta+2}\cdot x_1^{3/2}\,\zeta_0\big)^{\alpha_2}\,\big(\rho\cdot x_1^{\varepsilon_0} \,\zeta_0\big)^{\alpha_3}dx_1dx_2,
\end{split}
\end{equation}
where $\alpha_i$, $m$, and $\varepsilon$ are given by
\begin{equation}\label{C380}
\begin{split}
\alpha_1=&\frac{(1-\beta)/\theta}{1+\beta},~~\alpha_2=\frac{1}{1+\beta},~~\alpha_3=\frac{\beta-(1-\beta)/\theta}{1+\beta},\\
&m=\frac{\theta(2\beta-1)}{4(1-\beta)},~~\varepsilon_0=\frac{\theta(2\beta-1)}{4(\theta+1)\beta-4}.
\end{split}
\end{equation}

We mention that these parameters are determined by
\begin{equation*}
\begin{cases}
&\alpha_1+\alpha_2+\alpha_3=1,\\
&(1+\theta)\,\alpha_1+\beta\alpha_2+\alpha_3=1,\\
&2(\beta+2)\cdot\alpha_2=2,\\
&m\alpha_1+3\alpha_2/2+\varepsilon_0\alpha_3=1.\\
\end{cases}
\end{equation*}
In particular, when $\beta$ is close to $1$, we have $$\alpha_1<1/1000,~~\alpha_2<2/3,~~m\geq 10,~~0<\varepsilon_0<1/2.$$

Consequently, by means of \eqref{C342}, \eqref{C352} and \eqref{C379}, we arrive at
\begin{equation}\label{C369}
\begin{split}
&\int_\Omega\rho|u|^2\Phi\,\zeta_0\,dx_1dx_2\\
&\leq C\,\left(\int_\Omega\rho|u|^2\Phi\,\zeta_0\,dx_1dx_2+1\right)^{\frac{1}{4}}\left(\int_\Omega\rho^\beta|u|^{2\beta+2}\cdot x_1\,\zeta_0\,dx_1dx_2+1\right)^{\frac{2}{3}}.
\end{split}
\end{equation}
Let us apply the method given by \cite[Appendix A]{JZ} to estimate the second term of the right hand side of \eqref{C369}.

Suppose that $(\hat\rho,\hat{u})$, $\hat\zeta_0$,  and $\hat\psi$ are given by \eqref{C346}, \eqref{C370}, and \eqref{C378} respectively. Then we go back to $\hat\Omega$ and deduce that
\begin{equation}\label{C371}
\int_\Omega\rho^\beta|u|^{2\beta+2}\cdot x_1\,\zeta_0\,dx_1dx_2\leq C\int_{\hat\Omega}\hat\rho^\beta\,|\hat u|^{2\beta+2}\cdot\hat\eta_0^2\,d\hat x,
\end{equation}
where $\hat\eta_0(\hat{x}',\hat{x}_3)=\hat\zeta_0(3 \hat{x}',\hat{x}_3)$ for any $(\hat{x}',\hat{x}_3)\in\mathbb{R}^2\times\mathbb{T}$.

Observe that, for the balls $\hat{B}_7=\{x\in\mathbb{R}^3|\,|x|\leq 7\}$ and $\hat{B}_6=\{x\in\mathbb{R}^3|\,|x|\leq 6\}$, we apply \eqref{C342} to declare that 
\begin{equation*}
\begin{split}
\|\hat u\|_{H^1(\hat{B}_7)}&\leq C\|\hat u\|_{L^2(\hat{B}_7\setminus\hat{B}_6)}+C\|\nabla\hat u\|_{L^2(\hat{B}_7)},\\
&\leq C\int_{1/2}^{10}\int_0^1(|u|^2+|\nabla u|^2)\cdot x_1\,dx_1dx_2\leq C.
\end{split}
\end{equation*}
Then, according to \cite[Chapter V]{Evans}, we can extend $(\hat u\cdot\hat\eta_0)$ to a function $\hat v\in H_0^1(\hat{B}_7)$ which is compactly supported in the ball $\hat{B}_6$ and satisfies
\begin{equation}\label{C372}
\|\hat v\|_{H^1(\hat{B}_6)}\leq C\|\hat u\|_{H^1(\hat{B}_7)}\leq C.
\end{equation}

In addition, suppose that $\mathbf{E}_4(\cdot)$ is the periodic extension from $\mathbb{R}^2\times\mathbb{T}$ to  $\mathbb{R}^3$. Let us solve the Dirichlet problem in the ball $\hat{B}_7$,
\begin{equation*}
\begin{cases}
\Delta\hat{h}=\mathbf{E}_4\big(\hat\rho^\beta|\hat u|^{2\beta}\big)~~~~\mathrm{in}\ \hat{B}_7,\\
\hat{h}=0~~~~\mathrm{on}\ \partial \hat{B}_7.
\end{cases}
\end{equation*}
We apply the interior elliptic estimates \eqref{C209} together with \eqref{C347} to argue that
\begin{equation}\label{C376}
\begin{split}
\|\hat{h}\|_{L^\infty(\hat{B}_6)}
&\leq C\left(\sup_{\tilde{x}\in\hat{B}_6}\int_{\hat{B}_7}\frac{\mathbf{E}_4\big(\hat\rho^\beta|\hat u|^{2\beta}\big)}{|x-\tilde{x}|}\,d\hat{x}+\big\|\mathbf{E}_4\big(\hat\rho^\beta|\hat u|^{2\beta}\big)\big\|_{L^{1/\beta}(\hat{B}_7)}\right)\\
&\leq C\sup_{\tilde{x}\in\hat{B}_6}\int_{\hat\Omega}\frac{\hat\rho(|\hat u|^2+1)}{|\hat x-\tilde{x}|^\beta}\,\hat\psi(\hat x-\tilde{x})\cdot\hat\zeta_0\, d\hat x+C\int_{\hat{\Omega}}\hat\rho\,|\hat u|^2\cdot\hat\zeta_0\,d\hat{x}+C\\
&\leq C\int_\Omega\rho |u|^2\Phi\,dx_1dx_2+C.
\end{split}
\end{equation}

Now, observe that
\begin{equation}\label{C373}
\begin{split}
\int_{\hat\Omega}\hat\rho^\beta\,|\hat u|^{2\beta+2}\,\hat\zeta_0^2\,d\hat{x}
\leq\int_{\hat{B}_7}\Delta\hat{h}\cdot|\hat v|^2\,d\hat x
\leq C\,\|\hat{h}\|_{L^\infty(\hat{B}_6)}\,\|\hat v\|_{H^1(\hat{B}_6)}.
\end{split}
\end{equation}
Let us illustrate the last inequality in \eqref{C373}. In fact, we have
\begin{equation}\label{C375}
\begin{split}
\int_{\hat{B}_7}\Delta\hat{h}\cdot|\hat v|^2\,d\hat{x}
&=-2\int_{\hat{B}_7}\nabla\hat{h}\cdot\nabla\hat{v}\cdot\hat{v}\,d\hat{x}\\
&\leq\left(\int_{\hat{B}_7}|\nabla\hat{h}|^2\cdot|\hat{v}|^2\,d\hat x\right)^{\frac{1}{2}}\left(\int_{\hat{B}_7}|\nabla\hat v|^2\,d\hat x\right)^{\frac{1}{2}}.
\end{split}
\end{equation}
Meanwhile, it also holds that
\begin{equation}\label{C374}
\begin{split}
&\int_{\hat{B}_7}|\nabla\hat{h}|^2\cdot|\hat{v}|^2\,d\hat x\\
&=-\int_{\hat{B}_7}\hat{h}\cdot\Delta\hat{h}\cdot|\hat v|^2\,d\hat x-\int_{\hat{B}_7}\nabla\hat{h}\cdot\hat{h}\cdot\nabla\hat{v}\cdot\hat{v}\,d\hat{x}\\
&\leq\|\hat{h}\|_{L^\infty(\hat{B}_6)} \int_{\hat{B}_7}\Delta\hat{h}\cdot|\hat v|^2\,d\hat{x}+\frac{1}{2}\int_{\hat{B}_7}|\nabla\hat{h}|^2\,|\hat{v}|^2\,d\hat{x}+C\int_{\hat{B}_7}\hat{h}^2\,|\nabla\hat{v}|^2\,d\hat{x}.
\end{split}
\end{equation}
Combining \eqref{C375} and \eqref{C374} leads to \eqref{C373}.

Consequently, substituting \eqref{C372}, \eqref{C376} and \eqref{C373} into \eqref{C371}, we arrive at
\begin{equation}\label{C377}
\int_\Omega\rho^\beta|u|^{2\beta+2}\cdot x_1\,\zeta_0\,dx_1dx_2\leq C\int_\Omega\rho|u|^2\Phi\,dx_1dx_2+C.
\end{equation}
Then \eqref{C377} combined with \eqref{C342} and \eqref{C369} gives
\begin{equation*}
\int_\Omega\rho|u|^2\Phi\,\zeta_0\,dx_1dx_2\leq C\,\left(\int_\Omega\rho|u|^2\Phi\,\zeta_0\,dx_1dx_2+1\right)^{\frac{11}{12}},
\end{equation*}
which implies \eqref{C351}. We therefore finish the proof.
\end{proof}

We end this section by finishing Lemma \ref{L31}.

\textit{Proof of Lemma \ref{L31}.} Combining \eqref{C342} and \eqref{C351}, we apply the fact $\zeta_0+\zeta_1^2\geq 3/4$ to derive that
$$\int_\Omega\rho|u|^2\Phi\,dx_1dx_1\leq C\int_\Omega\rho|u|^2\Phi\,\big(\zeta_0+\zeta_1^2\big)\,dx_1dx_2\leq C,$$
which is \eqref{C310}. The proof of Lemma \ref{L31} is finally completed.
\thatsall

\subsection{Extension to the bounded domain $(0,1)\times\mathbb{T}$}\label{SS33}
\quad Now we consider the approximation system on the bounded domain $\Omega=(0,1)\times\mathbb{T}$
\begin{equation}\label{CC301}
\begin{cases}
\frac{1}{x_1}\partial_{1}(x_1\rho u_{1})+\partial_{2}(\rho u_2)=0,\\
\frac{1}{x_1}\partial_1(x_1\rho u_1^2)+\partial_{2}(\rho u_1 u_2)-\frac{1}{x_1}\partial_{1}(x_1\partial_{1}u_1)-\partial_{22}u_1+\frac{1}{x_1^2}u_1+\partial_{1}P=g_1,\\
\frac{1}{x_1}\partial_1(x_1\rho u_1 u_2)+\partial_2(\rho u_2^2)-\frac{1}{x_1}\partial_1(x_1\partial_1u_2)-\partial_{22}u_2+\partial_2P=g_2,
\end{cases}
\end{equation}
where the pressure $P=\rho+\varepsilon \rho^{\alpha}$, and the boundary values are given by 
\begin{equation}\label{CC302}
\begin{split}
u_1\big|_{\{x_1=0\}}&=0,~~\partial_1u_2\big|_{\{x_1=0\}}=0,~~u\big|_{\{x_1=1\}}=0.
\end{split}
\end{equation}
Moreover, the total mass of the fluid is provided by
\begin{equation}\label{CC309}
\int_\Omega P\, x_1dx_1dx_2=M.
\end{equation}

We mention that the a priori estimates established in Subsections \ref{SS30}--\ref{SS32} are all local in natural. Consequently, we can directly extend them to the bounded domain by simply replacing truncation functions.

Let us repeat the process through Propositions \ref{P31} to \ref{P38} to derive that
\begin{lemma}\label{L32}
Under the conditions of Lemma \ref{L21}, for  $0\leq k<1$ and $\theta=11/30$, there is a constant $C(k)$ determined by $k$ such that
\begin{equation}\label{CC303}
\int_{[0,k]\times\mathbb{T}} (\rho^{1+\theta}+\varepsilon\rho^{\alpha+\theta})\cdot x_1^6\,dx_1dx_2+\int_{[0,k]\times\mathbb{T}}\rho|u|^2 \,x_1 dx_1dx_2\leq C(k).
\end{equation}
\end{lemma}

First, we define a series of truncation functions $\{\phi_k(x_1)\}_{k\in(0,1)}\subset C^\infty([0,1])$ by
\begin{equation}\label{CC304}
\begin{cases}
\phi_k(x_1)=1~~\forall x_1\in[0,1-k],\\
\phi_k(x_1)=0~~\forall x_1\in[1-k/2,1],\\
0\leq\phi_k(x_1)\leq 1,~~|\phi_k'(x_1)|\leq 4k~~\forall x_1\in[0,1].
\end{cases}
\end{equation}
Then the arguments in Subsection \ref{SS30} are concluded by
\begin{proposition} 
Under the conditions of Lemma \ref{T12}, for   $0\leq\varepsilon_0\leq 1$, there are constants $C_1$ determined by $g$ and $C_2$ determined by $\varepsilon_0$ and $g$, such that
\begin{equation}\label{CC305} 
\begin{split}
&\int_\Omega\left(|\nabla u|^2x_1+\frac{|u_1|^2}{x_1}\right)dx_1dx_2\leq C_1,\\
\int_\Omega\big(\rho u_1^2&+P\big)\, x_1^{\varepsilon_0} \,dx_1dx_2\leq C_2\left(\int_A \rho u_1^2\,dx_1dx_2+1\right),\\
\end{split}
\end{equation}
where $A=[1/2,3/4]\times\mathbb{T}\subset\Omega$.
\end{proposition}
\begin{proof}
The energy estimates \eqref{CC305}$_1$ follow by the same calculations of \eqref{C304} and the boundary conditions \eqref{CC302}. Then
multiplying \eqref{CC301}$_2$ by $x_1^{\varepsilon_0}\cdot \phi_{1/2}$ and integrating over $\Omega$, we apply \eqref{C306}  to derive 
\begin{equation}\label{CC306}
\int_\Omega(\rho u_1^2+P)\,x_1^{\varepsilon_0}\cdot\phi_{1/2}\,dx_1dx_2\leq C\int_A(\rho|u_1|^2+P)\,dx_1dx_2+C.
\end{equation}
Next, let us multiply \eqref{CC301}$_2$ by $(1-x_1)\cdot\varphi$, where $\varphi(x_1)\in C^\infty([0,1])$ is a truncation function satisfying
\begin{equation}\label{CC307}
\begin{cases}
\varphi(x_1)=1~~\forall x_1\in[1/2,1],\\
\varphi(x_1)=0~~\forall x_1\in[0,1/4],\\
0\leq\varphi(x_1)\leq 1~~\forall x_1\in[0,1].
\end{cases}
\end{equation}
Then we check that
\begin{equation*} 
\begin{split}
&\int_\Omega \left(\frac{\rho u_1^2}{x_1}+P\right) \varphi\, dx_1dx_2-\int_\Omega(\rho u_1^2+P)(1-x_1)\,\varphi'\,dx_1dx_2\\
&=\int_\Omega\partial_1u_1\big(x_1^{-1}\varphi- (1-x_1)\,\varphi'\big)dx_1dx_2-\int_\Omega\frac{u_1}{x_1^2}\cdot(1-x_1)\, \varphi\,dx_1dx_2.
\end{split}
\end{equation*}
According to \eqref{CC307}, we have $\mathrm{supp}\,\varphi'\subset[1/4,1/2]$, therefore it is valid that
\begin{equation*}
\begin{split}
\int_\Omega (\rho u_1^2+P)\,\varphi\,dx_1dx_2\leq C\int_{[1/4,1/2]\times\mathbb{T}}(\rho u_1^2+P)dx_1dx_2+C,
\end{split}
\end{equation*}
which together with \eqref{CC309} and \eqref{CC306} yields 
\begin{equation*}
\begin{split}
\int_\Omega\big(\rho u_1^2+P\big)\, x_1^{\varepsilon_0} \,dx_1dx_2 
&\leq C \int_A(\rho u_1^2+P)\,dx_1dx_2+C\\
&\leq C_2\left(\int_A \rho u_1^2\,dx_1dx_2+1\right).
\end{split}
\end{equation*}
The proof is therefore completed.
\end{proof}

It is much easier to establish the corresponding version of Proposition \ref{P35} on $(0,1)\times\mathbb{T}$, since we do not need to investigate the far field behaviour of $P$.

\begin{proposition} 
Under the conditions of Lemma \ref{T1}, there is a constant $C$ such that $\forall x_1\geq 1/2$, we have
\begin{equation}\label{CC310}
\bigg|\int_0^1 u_2(x_1,x_2)\,dx_2\bigg|\leq C.
\end{equation}
In particular, we argue that
\begin{equation}\label{CC311}
\int_\Omega\big(\rho u_1^2+P\big)\, x_1^{\varepsilon_0} \,dx_1dx_2+\int_{[1/2,1]\times\mathbb{T}}\big(|u|^2+|\nabla u|^2)\,dx_1dx_2\leq C.
\end{equation}
\end{proposition}
\begin{proof}
Following the progress of \eqref{C317} and \eqref{C318}--\eqref{C320}, we still obtain that
\begin{equation}\label{CC312}
\begin{split} 
&\int_0^1\rho u_1(x_1,x_2)\,dx_2=0,~~\forall x_1\in[0,1],\\
&\partial_1\big(\int_0^1 u_2\,dx_2\big)=\int_0^1\rho u_1u_2\,dx_2,~~\forall x_1\in[0,1]\\
&\bigg|\int_0^1 u_2(x_1,x_2)\,dx_2\bigg|=\bigg|\int_{x_1}^1\int_0^1\rho u_1u_2(s,x_2)\,dsdx_2\bigg|,~~\forall x_1\in[1/2,1]
\end{split}
\end{equation}
Moreover, we apply  \eqref{CC312}$_1$ to \eqref{CC312}$_3$ and obtain that $\forall x_1\in[1/2,1]$
\begin{equation}\label{CC313}
\begin{split}
\bigg|\int_0^1 u_2(x_1,x_2)\,dx_2\bigg|&=\bigg|\int_{x_1}^\infty\int_0^1\rho u_1u_2 \,dsdx_2\bigg|\\
&=\bigg|\int_{x_1}^1\int_0^1\rho u_1\,\big(u_2-\int_0^1 u_2\,dx_2\big)\,dsdx_2\bigg|\\
&\leq\int_{x_1}^1\big(\int_0^1\rho|u_1|\,dx_2\big)\cdot\big(\int_0^1|\partial_2u_2|\,dx_2\big)\,ds\\
&\leq\big(\int_{x_1}^1\big(\int_0^1\rho|u_1|\,dx_2\big)^2 ds\big)^{\frac{1}{2}}
\,\big(\int_\Omega|\nabla u|^2x_1dx_1dx_2\big)^{\frac{1}{2}}\\
&\leq C\,\left(\int_{1/2}^1\big(\int_0^1 P\,dx_2\big)\cdot\big(\int_0^1\rho u_1^2\,dx_2\big)\cdot dx_1\right)^{\frac{1}{2}},
\end{split}
\end{equation}
where the last line is due to \eqref{CC305}.
To handle the last line of \eqref{CC313}, we directly introduce the test function
$$\xi(x_1)= \varphi(x_1)\cdot\int_{x_1}^1\int_0^1P\, dsdx_2,$$
which is well defined and bounded by $4M$ due to \eqref{CC309}. Let us integral \eqref{CC301}$_2$ with respect to $x_2$ and obtain  
\begin{equation}\label{CC314}
\begin{split}
&\partial_1\int_0^1P\,dx_2+\frac{1}{x_1}\, \partial_1(\int_0^1 x_1 \rho u_1^2 \,dx_2)-\frac{1}{x_1}\,\partial_1(\int_0^1x_1 \partial_1u_1\,dx_2)\\
&=-\frac{1}{x_1^2}\int_0^1u_1\,dx_2.
\end{split}
\end{equation}
Multiplying \eqref{CC314} by $\xi$ and integrating with respect to $x_1$, we mention that $\xi(1)=0$  and check each terms in details.
The first term on the left hand side of \eqref{CC314} satisfies
\begin{equation}\label{CC315}
\begin{split}
&\int_0^1\partial_1\big(\int_0^1 P \,dx_2\big)\,\xi\,dx_1\\
&=\int_0^1\left(\int_0^1P\,dx_2\right)^2\varphi\,dx_1-\int_0^1\big(\int_0^1 P \,dx_2\big)\,\big(\int_{x_1}^1\int_0^1P\,dsdx_2\big)\cdot\varphi'\,dx_1\\
&\geq\int_0^1\left(\int_0^1P\,dx_2\right)^2\varphi\,dx_1-C,
\end{split}
\end{equation}
where the last line is due to \eqref{CC309}.
Then we consider the second term of \eqref{CC314}.
\begin{equation}\label{CC316}
\begin{split}
&\int_0^1\frac{1}{x_1}\,\partial_1\big(\int_0^1 x_1\,\rho u_1^2 \,dx_2\big)\,\xi \,dx_1\\
&=\int_0^1 x_1^{-1}\big(\int_0^1\rho u_1^2\,dx_2\big)\,\xi \,dx_1+\int_0^1\big(\int_0^1\rho u_1^2\,dx_2\big)\cdot\big(\int_0^1 P\,dx_2\big)\cdot\varphi\,dx_1\\
&\quad-\int_0^1\big(\int_0^1\rho u_1^2\,dx_2\big)\,\big(\int_{x_1}^1\int_0^1P\,dsdx_2\big)\cdot\varphi'\,dx_1\\
&\geq \int_0^1\big(\int_0^1\rho u_1^2\,dx_2\big)\cdot\big(\int_0^1 P\,dx_2\big)\cdot\varphi\,dx_1-\int_A\rho u_1^2\,dx_1dx_2-C.
\end{split}
\end{equation}
Note that the first term in the second line of \eqref{CC316} is positive and the last line is due to \eqref{CC305}.
Next, we make use of \eqref{CC305} to handle the third term of \eqref{CC314} via
\begin{equation}\label{CC317}
\begin{split}
&\int_0^1\frac{1}{x_1}\,\partial_1\big(\int_0^1 x_1\,\partial_1u_1\,dx_2\big)\,\xi\,dx_1\\
&=\int_0^1\big(\int_0^1\partial_1u_1\,dx_2\big)\cdot\big(\int_0^1 P\,dx_2\big)\cdot \varphi\,dx_1+\int_0^1 x_1^{-1}\big(\int_0^1\partial_1u_1\,dx_2\big)\,\xi \,dx_1\\
&\quad-\int_0^1\big(\int_0^1\partial_1u_1\,dx_2\big)\,\big(\int_{x_1}^1\int_0^1P\,dsdx_2\big)\cdot\varphi'\,dx_1\\
&\leq \frac{1}{8} \int_0^1\left(\int_0^1P\,dx_2\right)^2\varphi\,dx_1
+\int_{[1/4,1]\times\mathbb{T}}|\nabla u|^2\,dx_1dx_2+C\\
&\leq \frac{1}{8} \int_0^1\left(\int_0^1P\,dx_2\right)^2\varphi\,dx_1
+C.
\end{split}
\end{equation}
Similarly, \eqref{CC305} also ensures that the right hand side of \eqref{CC314} is bounded by 
\begin{equation}\label{CC318}
\begin{split}
\left|\int_0^1\frac{1}{x_1^2}\big(\int_0^1 u_1\,dx_2\big)\,\xi \,dx_1\right| 
&\leq C\int_{[1/4,1]\times\mathbb{T}}|u_1|\,dx_1dx_2 \leq C.
\end{split}
\end{equation}
Combining \eqref{CC315}--\eqref{CC318}, we arrive at
\begin{equation}\label{CC319}
\begin{split}
\int_0^1\big(\int_0^1\rho u_1^2\,dx_2\big)\cdot\big(\int_0^1 P\,dx_2\big)\,\varphi \,dx_1\leq C\int_A\rho u_1^2\,dx_1dx_2+C.
\end{split}
\end{equation}
In addition, by virtue of \eqref{CC305} and \eqref{CC312}$_1$, we infer that
\begin{equation}\label{CC320}
\begin{split}
\int_A\rho u_1^2\,dx_1dx_2
&=\int_{1/2}^{3/4}\left(\int_0^1\rho u_1\big(u_1-\int_0^1 u_1\,dx_2\big)\,dx_2\right)\,dx_1\\
&\leq\int_{1/2}^{3/4}\left(\int_0^1\rho |u_1|\,dx_2\right)\cdot\left(\int_0^1|\partial_2u_1|\,dx_2\right)\,dx_1\\
&\leq\left(\int_{1/2}^{3/4}\big(\int_0^1\rho|u_1|\,dx_2\big)^2x_1^{-1}dx_1\right)^{\frac{1}{2}}
\left(\int _\Omega|\nabla u|^2\,x_1dx_1dx_2\right)^{\frac{1}{2}}\\
&\leq C\,\left(\int_0^1\big(\int_0^1 P\,dx_2\big)\cdot\big(\int_0^1\rho u_1^2\,dx_2\big)\cdot\varphi\,dx_1\right)^{\frac{1}{2}}.
\end{split}
\end{equation}
Substituting \eqref{CC320} into \eqref{CC319} also leads to
\begin{equation}\label{CC322}
\begin{split}
\int_A\rho u_1^2\,dx_1dx_2+\int_0^1\big(\int_0^1\rho &u_1^2\,dx_2\big)\cdot\big(\int_0^1 P\,dx_2\big)\cdot\varphi\,dx_1\leq C.\\ 
\end{split}
\end{equation}
Thus in view of \eqref{CC313} and \eqref{CC322}, we derive that for any $x_1\in[1/2,1]$
\begin{equation*}
\begin{split}
\bigg|\int_0^1&u_2(x_1,x_2)\,dx_2\bigg|\leq C, 
\end{split}
\end{equation*}
which provides \eqref{CC310}. Moreover, \eqref{CC305} together with \eqref{CC322} gives
\begin{equation}\label{CC323}
\int_\Omega\big(\rho u_1^2+P\big)\, x_1^{\varepsilon_0} \,dx_1dx_2\leq C.
\end{equation}
Finally, we take advantage of \eqref{CC305} and \eqref{CC310} to argue that
\begin{equation*}
\begin{split}
&\int_{1/2}^1\int_0^1u_2^2 \,dx_1dx_2\\
&\leq C\int_{1/2}^1\int_0^1\big(u_2-\int_0^1 u_2\,dx_2\big)^2\,dx_1dx_2+C\int_{1/2}^1\left(\int_0^1u_2\,dx_2\right)^2\, \,dx_1\\
&\leq C\int_\Omega|\partial_2 u_2|^2x_1dx_1dx_2+C\int_{1/2}^1\left(\int_0^1 u_2\,dx_2\right)^2 \,dx_1\leq C,
\end{split}
\end{equation*}
which along with \eqref{CC305} and \eqref{CC323}  yields \eqref{CC311}. The proof is completed.
\end{proof}

After obtaining crucial estimates \eqref{CC311}, we can follow Subsections \ref{SS31} and \ref{SS32} to finish Lemma \ref{L32}.

\textit{Proof of Lemma \ref{L32}.} The approaches adopted in Subsection \ref{SS31} are directly available at the present stage, once we replace the weight $x_1^{-m}$ in all test functions by the truncations $\varphi\cdot\phi_k(x_1)$ given by \eqref{CC304} and \eqref{CC307}. Therefore, we declare the following estimates away from the axis with $\theta=11/30$.
\begin{equation}\label{CC324}
\int_\Omega\big(\rho^{1+\theta}+\varepsilon\rho^{\alpha+\theta}+\rho|u|^2\big)\,\varphi\cdot\phi_k\,dx_1dx_2\leq C(k),
\end{equation}
for some constant $C(k)$ determined by $k$. Meanwhile, the near side estimates in Subsection \ref{SS32} are valid in $(0,1)\times\mathbb{T}$ as well, thus we deduce that
\begin{equation}\label{CC325}
\int_\Omega\rho|u|^2\,x_1\cdot\phi_{1/2}\,dx_1dx_2\leq C.
\end{equation}
Combining \eqref{CC324} and \eqref{CC325} leads to
\begin{equation}\label{CC326}
\int_\Omega\rho|u|^2\,x_1\cdot\phi_k\,dx_1dx_2\leq C(k).
\end{equation}
In addition, \eqref{C352} also gives
\begin{equation*}
\int_{[0,1/2]\times\mathbb{T}} (\rho^{1+\theta}+\varepsilon\rho^{\alpha+\theta})\cdot x_1^6\,dx_1dx_2\leq C\left(\int_\Omega\rho|u|^2\cdot\phi_{3/4}\,dx_1dx_2\right)^4+C\leq C,
\end{equation*}
which along with \eqref{CC324} and \eqref{CC326} provides \eqref{CC303} and finishes the proof.
\thatsall

\section{Construction of non-trivial weak solutions }

\quad In this section, we make use of a priori estimates in Section \ref{S3} to construct non-trivial weak solutions to the system \eqref{C102} in $(0,1)\times\mathbb{T}$ and $(0,\infty)\times\mathbb{T}$.

\subsection{Local compactness assertions--the proof of Theorem \ref{TT1}.}\label{S4}
\quad First, let us prove Lemma \ref{T1} which ensures that the approximation solutions of the problem \eqref{C113}--\eqref{CC101} in $(0,\infty)\times\mathbb{T}$ converging to a weak solution of the problem \eqref{C102}--\eqref{C105}. The corresponding arguments on $(0,1)\times\mathbb{T}$ follow by the same method, so we omit the proof.

\textit{Proof of Lemma \ref{T1}.}
Let $\{(\rho^\varepsilon,u_1^\varepsilon,u_2^\varepsilon) \}_{\varepsilon>0}$ be a sequence of solutions to the system \eqref{C113} in $(0,\infty)\times\mathbb{T}$ under conditions \eqref{C114}--\eqref{CC101}, which is guaranteed by Lemma \ref{L21} and satisfies \eqref{C203}--\eqref{C111}. We also introduce the domains $$A_n=\left\{(x_1,x_2)\in\Omega\big|\,\frac{1}{n}\leq x_1\leq n\right\}.$$

By applying estimates \eqref{C311}, \eqref{C342}, \eqref{C344}, \eqref{C351}, and \eqref{C352} to $(\rho^\varepsilon,
u^\varepsilon_1,u^\varepsilon_2)$, we derive that $\forall n\in\mathbb{N}^+$ and $P^\varepsilon=\rho^\varepsilon+\varepsilon(\rho^\varepsilon)^\alpha$, there is a constant $C_n$ depending only on $n$, such that
\begin{equation}\label{C402}
\begin{split}
\|P^\varepsilon\cdot(\rho^\varepsilon)^\alpha\|_{L^1(A_n)}+\|u^\varepsilon\|_{H^1(A_n)}\leq C_n.
\end{split}
\end{equation}

Thus, Sobolev embedding theorem together with the Cantor diagonal argument ensures that, we can define $(\rho,u_1,u_2)$ on $\Omega$, such that for any $n\in\mathbb{N}^+$ and $q=4(1+\theta)/\theta$, it holds that
\begin{equation}\label{C403}
\begin{split}
&\quad\|\rho^{1+\theta}\|_{L^1(A_n)}+\|u\|_{H^1(A_n)}\leq C_n,\\
&\quad P^\varepsilon\cdot x_1\rightharpoonup\rho\cdot x_1\ \mbox{in $L^{1+\theta/\alpha}\big(A_n\big)$,}\\
&\quad\rho^\varepsilon\cdot x_1\rightharpoonup\rho\cdot x_1\ \mbox{in $L^{1+\theta}\big(A_n \big)$,}\\
&u^\varepsilon \rightharpoonup u\ \mbox{in $H^{1}\big(A_n\big)$,}~~u^\varepsilon \rightarrow u\ \mbox{in $L^{q}\big(A_n\big)$.}\\
\end{split}
\end{equation}
Combining \eqref{C403}$_3$ and \eqref{C403}$_4$, we argue that
\begin{equation}\label{C404}
\rho^\varepsilon u^\varepsilon_i u^\varepsilon_j\rightharpoonup\rho\,u_i u_j\ \mbox{in $L^r(A_n)$ with $r=\frac{2+2\theta}{2+\theta}$.}
\end{equation}

We still need to investigate the system near the symmetric axis. Let us define the domain $A_0=(0,1)\times\mathbb{T}$. In view of \eqref{C303} and \eqref{C354}, we derive that
\begin{equation}\label{C405}
\begin{split}
&\int_{\Omega}\left(|\nabla u^\varepsilon|^2 x_1+\frac{|u_1^\varepsilon|^2}{x_1}\right)dx_1dx_2\leq C,\\
&\int_{A_0}\left(P^\varepsilon\, x_1^{\frac{1}{2}}+P^\varepsilon\cdot(\rho^\varepsilon)^{\theta}\, x_1^{10} \right)dx_1dx_2\leq C,\\
\end{split}
\end{equation}
which in particular implies that, for $\kappa=\theta/19$ and $\gamma_1=18/19$
\begin{equation}\label{C406}
\begin{split}
\int_{A_0}(\rho^\varepsilon)^{1+\kappa}x_1\,dx_1dx_2&=\int_{A_0}\big(\rho^\varepsilon\cdot x_1^{1/2}\big)^{\gamma_1}\cdot\big((\rho^\varepsilon)^{1+\theta}\, x_1^{10}\big)^{1-\gamma_1}\,dx_1dx_2\leq C,\\
\varepsilon\int_{A_0}(\rho^\varepsilon)^{\alpha+\kappa}x_1\,dx_1dx_2&=\varepsilon\int_{A_0}\big((\rho^\varepsilon)^{\alpha}\,x_1^{1/2}\big)^{\gamma_1} \big((\rho^\varepsilon)^{\alpha+\theta}\,x_1^{10}\big)^{1-\gamma_1}\,dx_1dx_2\leq C.\\
\end{split}
\end{equation}
Combining \eqref{C403}, \eqref{C405} and \eqref{C406} leads to
\begin{equation}\label{C407}
\begin{split}
&P^\varepsilon\cdot x_1\rightharpoonup\rho\cdot x_1\ \mbox{in $L^{1+\kappa/\alpha}\big(A_0\big)$,}\\
\nabla u^\varepsilon\rightharpoonup&\nabla u,~~ 
 u_1^\varepsilon/x_1\rightharpoonup u_1/x_1\ \mbox{both in $\mathcal{L}^{2}\big(\Omega\big)$.}
\end{split}
\end{equation}
In addition, according to \eqref{C379}, we choose $\delta<1/100$ and deduce that
\begin{equation*} 
\begin{split}
&\int_{A_0}\rho^\varepsilon|u^\varepsilon|^2\cdot x_1^{1-\delta}\,dx_1dx_2\\
&\leq\int_\Omega\big((\rho^\varepsilon)^{1+\theta}x_1^{m}\zeta_0\big)^{\alpha_1}\,\big((\rho^\varepsilon)^{\beta}|u^\varepsilon|^{2\beta+2}\cdot x_1^{5/4}\,\zeta_0\big)^{\alpha_2}\,\big(\rho^\varepsilon\cdot x_1^{\varepsilon_0} \,\zeta_0\big)^{\alpha_3}dx_1dx_2,\\
\end{split}
\end{equation*}
where $\alpha_i$, $m$ and $\varepsilon$ are determined by \eqref{C380}. Then let us repeat the process given in \eqref{C369}--\eqref{C377} and apply Lemma \ref{L31} to infer that
\begin{equation}\label{C408}
\int_{A_0}\rho^\varepsilon|u^\varepsilon|^2\cdot x_1^{1-\delta}\,dx_1dx_2\leq C.
\end{equation}

Combining \eqref{C404} and \eqref{C408}, we next illustrate that
\begin{equation}\label{C409}\\
\rho^\varepsilon u^\varepsilon_i u^\varepsilon_j\cdot x_1\rightharpoonup\rho\,u_i u_j\cdot x_1\ \mbox{in $L^1(A_0)$.}
\end{equation}
In fact, for any $\epsilon>0$, we can find some $K, N\in\mathbb{N}^+$, such that
\begin{equation*}
\begin{split}
&\int_{A_0\cap\{\rho^\varepsilon|u^\varepsilon|^2\geq K\}}\rho^\varepsilon|u^\varepsilon|^2\cdot x_1\,dx_1dx_2\\
&\leq\int_{\{x_1<N^{-1}\}\cap\{\rho^\varepsilon|u^\varepsilon|^2\geq K\}}\rho^\varepsilon|u^\varepsilon|^2\cdot x_1\,dx_1dx_2+\int_{A_N\cap\{\rho^\varepsilon|u^\varepsilon|^2\geq K\}}\rho^\varepsilon|u^\varepsilon|^2\,dx_1dx_2\\
&\leq N^{-\delta}\int_{A_0}\rho^\varepsilon|u^\varepsilon|^2\cdot x_1^{1-\delta}\,dx_1dx_2
+K^{-\theta/(2+\theta)}\big\|\rho^\varepsilon|u^\varepsilon|^2\big\|_{L^{r}(A_N)}\\
&\leq CN^{-\delta}+C(N)\,K^{-\theta/(2+\theta)}.
\end{split}
\end{equation*}
where the last line is due to \eqref{C402}, \eqref{C404}, and \eqref{C408}. Thus, we first set $N=(C/\epsilon)^{1/\delta}$ and then choose $K=\big(C(N)/\epsilon\big)^{(2+\theta)/\theta}$ to declare that
$$\int_{A_0\cap\{\rho^\varepsilon|u^\varepsilon|^2\geq K\}}\rho^\varepsilon|u^\varepsilon|^2\cdot x_1\,dx_1dx_2\leq 2\epsilon,$$
which implies that the sequence $\big\{\rho^\varepsilon u^\varepsilon_i u^\varepsilon_j\,x_1\big\}_{\varepsilon>0}$ is uniformly integrable on $A_0$. Consequently, we can find some $f_{ij}\in L^1(A_0)$, such that
\begin{equation*}
\rho^\varepsilon u^\varepsilon_i u^\varepsilon_j\cdot x_1\rightharpoonup f_{ij}\ \mbox{in $L^1(A_0)$.}
\end{equation*}
However, in view of \eqref{C404}, we must have $f_{ij}=\rho\,u_i u_j\,x_1$, which implies \eqref{C409}.

Now, let us select the test function $\varphi\in \big(C_0^\infty(\Omega)\big)^2$ satisfying $$\varphi_1\big|_{x_1=0}=\partial_1\varphi_2\big|_{x_1=0}=0,$$
and substitute it into \eqref{C109} to declare that
\begin{equation}\label{C410}
\begin{split}
&\int_{\Omega}\big(\rho^\varepsilon u_i^\varepsilon u_j^\varepsilon\cdot\partial_i\varphi_j+P^\varepsilon\,\big(\frac{\varphi_1}{x_1}+\mathrm{div}\varphi\big)+g\cdot\varphi\big)x_1dx_1dx_2\\
&=\int_\Omega\big(\nabla u^\varepsilon\cdot\nabla\varphi+(u_1^\varepsilon /x_1)\cdot(\varphi_1/x_1)\big)\,x_1dx_1dx_2.
\end{split}
\end{equation}
Observe that $\varphi_1|_{x_1=0}=0$ ensures that $(\varphi_1/x_1)\in L^\infty(A_0)$, thus 
we apply \eqref{C403}, \eqref{C404}, \eqref{C407}  together with \eqref{C409} and pass to the limit in \eqref{C410} to derive that
\begin{equation}\label{C411}
\begin{split}
&\int_{\Omega}\big(\rho u_i u_j \cdot\partial_i\varphi_j+\rho\,\big(\frac{\varphi_1}{x_1}+\mathrm{div}\varphi\big)+g\cdot\varphi\big)x_1dx_1dx_2\\
&=\int_\Omega\big(\nabla u\cdot\nabla\varphi+(u_1 /x_1)\cdot(\varphi_1/x_1)\big)\,x_1dx_1dx_2.
\end{split}
\end{equation}
Similar arguments also implies that $\forall\phi\in C_0^\infty(\Omega)$, 
\begin{equation}\label{C412}
\int_{\Omega}\rho u\,\nabla\phi\cdot x_1dx_1dx_2=0.
\end{equation}

Next, observe that \eqref{C403} together with \eqref{C407} guarantees that for any $N>0$,
\begin{equation}\label{C413}
\rho^\varepsilon\cdot x_1\rightharpoonup\rho\cdot x_1\ \mbox{in $L^1([0,N]\times\mathbb{T})$.}
\end{equation} 
Fixing $m\in\mathbb{N}^+$, we consider the truncation function $\zeta_m(x_1)=\zeta_0(x_1/m)$ where $\zeta_0$ is given by \eqref{C382}. In view of \eqref{CC503}, we have
$$\int_\Omega\rho^\varepsilon\cdot x_1\,\zeta_n\,dx_1dx_2\leq C_m,$$
for some constant $C_m$ depending on $n$, which combined with \eqref{C413} leads to
\begin{equation}\label{C414}
\int_\Omega\rho\cdot x_1\,\zeta_m\,dx_1dx_2=\lim_{\varepsilon\rightarrow 0}\int_\Omega\rho^\varepsilon\cdot x_1\,\zeta_m\,dx_1dx_2\leq C_m.
\end{equation}
Consequently, we collect \eqref{C404}, \eqref{C407}, \eqref{C409}, and \eqref{C414} to derive that, for $i=1,2,$
\begin{equation}\label{C416}
\begin{split}
\rho\in \mathcal{L}^1_{loc}(\Omega),~~
\rho|u_i|^2\in \mathcal{L}^1_{loc}(\Omega),~~  \nabla u_i\in \mathcal{L}^2(\Omega),\\
\end{split}
\end{equation}
and the regularity condition \eqref{C106} holds.
Combining \eqref{C411}, \eqref{C412}, \eqref{C414}, and \eqref{C416} implies that $(\rho,u_1,u_2)$ is a weak solution to the problem \eqref{C102}--\eqref{C105}, moreover \eqref{C130} follows directly form \eqref{C316}, \eqref{C342}, \eqref{C343}, and \eqref{C351}. Therefore the proof of Lemma \ref{T1} is completed.
\thatsall

As a direct corollary of local compactness assertions in Lemma \ref{T1}, we prove Theorem \ref{TT1} and construct a weak solution to the system \eqref{C102} under conditions \eqref{CC105}.

\textit{Proof of Theorem \ref{TT1}.} Let $\{(\rho^\varepsilon,u_1^\varepsilon,u_2^\varepsilon) \}_{\varepsilon>0}$ be a sequence of solutions to the system \eqref{C113} in $(0,1)\times\mathbb{T}$ under conditions \eqref{CC102}--\eqref{CC103}, which is provided by Lemma \ref{L21} and satisfies \eqref{C203}--\eqref{C111}.
 
Lemma \ref{T1} already ensures that we can find a subsequence of $\{(\rho^\varepsilon,u_1^\varepsilon,u_2^\varepsilon) \}_{\varepsilon>0}$ converging weakly to a weak solution $(\rho,u_1,u_2)$ of the system \eqref{C113} in $(0,1)\times\mathbb{T}$. Thus we need only to verify \eqref{C119}.

In fact, according to \eqref{CC309}, we derive that
\begin{equation*}
\int_\Omega P^\varepsilon\cdot x_1\,dx_1dx_2=M.
\end{equation*}
Therefore, by passing to the subsequence, we may assume that $\{P^\varepsilon\cdot x_1\}_{\varepsilon>0}$ converges $*$-weakly to some measure $d\mu$ on $\Omega$. Moreover, since the domain $(0,1)\times\mathbb{T}$ is compact, the constant function $\Psi\equiv 1$ is a proper test function, thus we declare that
\begin{equation}\label{CC401}
\int_\Omega d\mu=\int_\Omega\Psi\, d\mu=\lim_{\varepsilon\rightarrow 0}\int_\Omega \Psi\,(P^\varepsilon\cdot x_1)\,dx_1dx_2=M.
\end{equation}
However, on each compact set $A_\delta\triangleq[\delta,1-\delta]\times\mathbb{T}$ with $\delta<1/4$, the weak compactness assertions \eqref{C403} give that
\begin{equation*}
P^\varepsilon\cdot x_1\rightharpoonup\rho\cdot x_1\ \mbox{in $L^{r_0}\big(A_\delta\big)$, for some $r_0>1$,}
\end{equation*} 
which enforces $d\mu=\rho\cdot x_1$. Consequently, \eqref{CC401} implies that
\begin{equation*}
\int_\Omega\rho\,x_1\,dx_1dx_2=M,
\end{equation*} 
and \eqref{CC104} is valid. The proof is therefore completed.
\thatsall

\subsection{Global compactness assertions--the proof of Theorem \ref{T12}.}\label{S5}
\quad
Now we construct non-trivial weak solutions on $(0,\infty)\times\mathbb{T}$. We mention that  Lemma \ref{T1} is not strong enough to exclude the trivial case, because by virtue of  the concentration compactness theory due to Lions \cite{Lions}, the sequence $\{f_n\}_{n=1}^\infty$ on $\Omega=(0,\infty)\times\mathbb{T}$, which satisfies
$$\int_\Omega f_n \,dx_1dx_2= M,$$
may still vanish completely during the weak limit process. For instance, such phenomenon can occur  when we construct weak solutions to the 3D Cauchy problem by solving approximation systems on the ball $B_R=\{x\in\mathbb{R}^3\,|\,|x|\leq R\}$.

Supposing that the external force $g=0$, we consider the following approximation system on $B_R$, 
\begin{equation}\label{C417}
\begin{cases}
\mathrm{div}(\rho^Ru^R)=0\ \mathrm{in}\ B_R,\\
\mathrm{div}(\rho^Ru^R\otimes u^R)-\Delta u^R+\nabla\rho^R=0\ \mathrm{in}\ B_R,\\
\int_{B_R}\rho^R\,dx=M.
\end{cases}
\end{equation}
Direct computation yields that
$$\rho^R=\frac{3M}{4\pi R^3},~~u_R=0,$$
which gives the solution to \eqref{C417}. If $\rho^R$ is zero extended to $\mathbb{R}^3$, we derive 
\begin{equation*}
\rho^R\rightharpoonup 0 \mbox{ in $L^1(\mathbb{R}^3)$},~~u^R\rightarrow 0 \mbox{ in $H^1(\mathbb{R}^3)$}.
\end{equation*}
Note that although the total mass of $\rho^R$ is conserved for each $R$, the limiting solution $\rho$ may still vanish completely.

In contrast, if the external force is properly designed, we can avoid the trivial solution. For example, suppose that the external force is of potential type $g=\nabla\varphi$ where $\varphi$ is a positive smooth function supported on $B_1$ and satisfies
$$\int_{B_1}\varphi\,dx=M.$$
Then, we still consider the approximation system on $B_R$ with $R\geq 1$,
\begin{equation}\label{C418}
\begin{cases}
\mathrm{div}(\rho^Ru^R)=0\ \mathrm{in}\ B_R,\\
\mathrm{div}(\rho^Ru^R\otimes u^R)-\Delta u^R+\nabla\rho^R=\nabla\varphi\ \mathrm{in}\ B_R,\\
\int_{B_R}\rho^R\,dx=M.
\end{cases}
\end{equation}
Clearly, we check that
$$\rho^R=\varphi,~~u_R=0,$$
give a solution to \eqref{C418}. If $\rho^R$ is zero extended to $\mathbb{R}^3$, we derive  
\begin{equation*}
\rho^R\rightarrow\varphi \mbox{ in $L^1(\mathbb{R}^3)$},~~u^R\rightarrow 0 \mbox{ in $H^1(\mathbb{R}^3)$}.
\end{equation*}
Therefore the limiting solution is non-trivial in this case.
\begin{remark} 
The examples given above explain a part of the reason why we can not construct a solution on $(0,\infty)\times\mathbb{T}$, satisfying 
\begin{equation}\label{C419}
\int_\Omega\rho\,x_1dx_1dx_2=M,
\end{equation}
in Lemma \ref{L21}. It seems that the fast decay rate of $g$ at the far field is not sufficient to exclude the trivial solution. In order to obtain \eqref{C419},  the external force $g$ should provide some effects of concentration to avoid $\rho$ spreading out like \eqref{C417}.
\end{remark}
 
To get over the difficulties mentioned above, we must establish a type of concentration assertions.
Let us prove Theorem \ref{T12} and construct a non-trivial weak solution to the problem \eqref{C102}--\eqref{C105}.

\textit{Proof of Theorem \ref{T12}.}
Suppose that $\{(\rho^\varepsilon,u_1^\varepsilon,u_2^\varepsilon) \}_{\varepsilon>0}$ is a sequence of approximation solutions to the systems \eqref{C113}--\eqref{CC101} in $(0,\infty)\times\mathbb{T}$, which is provided by Lemma \ref{L21} and satisfies \eqref{C203}--\eqref{C111}.

Lemma \ref{T1} already guarantees that $\{(\rho^\varepsilon,u_1^\varepsilon,u_2^\varepsilon) \}_{\varepsilon>0}$ will converge weakly to a weak solution $(\rho,u_1,u_2)$ to the system \eqref{C113}--\eqref{C114}. We need to verify   \eqref{CA102} and \eqref{C119}.

By virtue of \eqref{CC503}, we have
$$\int_\Omega P^\varepsilon\,\omega dx_1dx_2=M,$$
while \eqref{C507} ensures that
$$\int_\Omega P^\varepsilon\cdot x_1^{-9/8} dx_1dx_2\leq C.$$

Thus for any $\delta>0$, we can find some $N>0$ such that
\begin{equation*}
\begin{split}
\int_{\{x_1\geq N\}}P^\varepsilon\,\omega dx_1dx_2
&=\int_{\{x_1\geq N\}}P^\varepsilon\cdot x_1^{-2}\,dx_1dx_2\\
&\leq N^{-7/8}\cdot\int_\Omega P^\varepsilon\cdot x_1^{-9/8}\,dx_1dx_2\\
&\leq CN^{-7/8}\leq\delta,
\end{split}
\end{equation*}
which provides \eqref{CA102}.
Moreover, by virtue of \eqref{CC503}, we have
\begin{equation}\label{C552}
\bigg|\int_{\{x_1\leq N\}}P^\varepsilon\,\omega dx_1dx_2-M\bigg|\leq\delta.
\end{equation}
Meanwhile, in view of the definition \eqref{C116},  we apply the weak compactness assertions \eqref{C403} and \eqref{C407} on $[0,N]\times\mathbb{T}$ and let $\varepsilon\rightarrow 0$ in \eqref{C552} to deduce that
\begin{equation*}
\bigg|\int_{\{x_1\leq N\}}\rho\,\omega dx_1dx_2-M\bigg|\leq\delta.
\end{equation*}
Note that $\delta>0$ is arbitrarily determined, thus the monotonic convergence theorem guarantees that
$$\int_\Omega\rho\,\omega\,dx_1dx_2=M.$$
We therefore finish \eqref{C119} which ensures that the solution is non-trivial. The proof of  Theorem \ref{T12} is therefore completed.
\thatsall



\end{document}